\documentclass[11pt,psamsfonts, oneside, reqno]{amsart}
\newcommand{\hkra}{\hookrightarrow}

\newcommand{\ra}{\rightarrow}

\newcommand{\set}[1]{\left\{ #1 \right\}}

\newcommand{\cc}[1]{\overline{#1}}

\newcommand{\sub}{\subset}

\newcommand{\sm}{\ensuremath{\setminus}}

\newcommand{\s}[2]{\sum\limits_{#1}{#2} }
\newcommand{\p}[2]{\prod\limits_{#1}{#2} }
\newcommand{\g}{\circ}
\newcommand{\sign}{\normalfont\text{sign}}
\newcommand{\inv}{^{-1}}

\newcommand{\norm}[1]{\lVert#1\rVert}
\newcommand{\concat}{\ast}
\newcommand{\cl}{\colon}
\newcommand{\djun}[1]{\bigsqcup\limits_{#1}}

\newcommand{\lbr}[1]{\Bigl(#1\Bigr)}
\newcommand{\emst}{\emptyset}

\newcommand{\xra}{\xrightarrow}

\newcommand{\scA}{\mathscr{A}}

\newcommand{\scC}{\mathscr{C}}
\newcommand{\scD}{\mathscr{D}}

\newcommand{\scF}{\mathscr{F}}
\newcommand{\scG}{\mathscr{G}}
\newcommand{\scH}{\mathscr{H}}
\newcommand{\scI}{\mathscr{I}}

\newcommand{\scP}{\mathscr{P}}
\newcommand{\scQ}{\mathscr{Q}}
\newcommand{\scS}{\mathscr{S}}

\newcommand{\infi}[2]{\inf\limits_{#1}{#2}}

\newcommand{\eva}{ev}

\newcommand{\pf}{\pitchfork}

\newcommand{\lc}{\underline}

\newcommand{\dul}{^\vee}

\newcommand{\ide}{\normalfont\text{id}}
\newcommand{\pt}{\normalfont\text{pt}}
\newcommand{\im}{\normalfont\text{im}}

\newcommand{\PGL}{\normalfont\text{PGL}}
\newcommand{\PU}{\normalfont\text{PU}}

\newcommand{\coker}{\normalfont\text{coker}}
\newcommand{\pr}{\normalfont\text{pr}}

\newcommand{\Mbar}{\overline{\mathcal M}}

\newcommand{\vdim}{\normalfont\text{vdim}}
\newcommand{\delbar}{\bar\partial}
\newcommand{\del}{\partial}
\newcommand{\Hom}{\normalfont\text{Hom}}

\newcommand{\Aut}{\normalfont\text{Aut}}

\newcommand{\bC}{\mathbb{C}}
\newcommand{\bD}{\mathbb{D}}

\newcommand{\bN}{\mathbb{N}}

\newcommand{\bP}{\mathbb{P}}

\newcommand{\bR}{\mathbb{R}}

\newcommand{\bZ}{\mathbb{Z}}
\newcommand{\cA}{\mathcal{A}}
\newcommand{\cB}{\mathcal{B}}
\newcommand{\cC}{\mathcal{C}}
\newcommand{\cD}{\mathcal{D}}
\newcommand{\cE}{\mathcal{E}}
\newcommand{\cF}{\mathcal{F}}

\newcommand{\cI}{\mathcal{I}}

\newcommand{\cK}{\mathcal{K}}

\newcommand{\cM}{\mathcal{M}}
\newcommand{\cN}{\mathcal{N}}
\newcommand{\cO}{\mathcal{O}}
\newcommand{\cP}{\mathcal{P}}

\newcommand{\cR}{\mathcal{R}}

\newcommand{\cT}{\mathcal{T}}

\newcommand{\cV}{\mathcal{V}}
\newcommand{\cW}{\mathcal{W}}
\newcommand{\cX}{\mathcal{X}}

\newcommand{\cZ}{\mathcal{Z}}

\newcommand{\fC}{\mathfrak{C}}

\newcommand{\fb}{\mathfrak{b}}
\newcommand{\fc}{\mathfrak{c}}

\newcommand{\ff}{\mathfrak{f}}

\newcommand{\fo}{\mathfrak{o}}

\newcommand{\fq}{\mathfrak{q}}

\newcommand{\fs}{\mathfrak{s}}

\newcommand{\codim}{\normalfont{\text{codim}}}

\newcommand{\pcd}{\normalfont\text{PD}}
\newcommand{\wt}{\widetilde}

\newcommand{\wh}{\widehat}

\newcommand{\ind}{\normalfont\text{ind}}

\newcommand{\orl}{\hspace{0.8pt}\fo}

\newcommand{\q}{\mathfrak{q}}

\newcommand{\QQ}{\mathbb{Q}}

\newcommand{\ZZ}{\mathbb{Z}}

\newcommand{\mfa}{\mathsf{a}}
\newcommand{\obs}{s}

\newtheorem{intthm}{Theorem}

\newcommand{\dmc}{\widetilde{\mathscr{O}\mathscr{C}}}
\newcommand{\cdmc}{\overline{\mathscr{O}\mathscr{C}}}

\newcommand{\opr}{\normalfont\text{opr}}
\newcommand{\stb}{\normalfont\text{st}}

\newcommand{\ov}[1]{\overline{#1}}
\newcommand{\cKc}{\cK^{\scale{\square}{0.6}}}
\newcommand{\cKd}[1]{\cK^{#1,\scale{\square}{0.6}}}

\newcommand{\cMc}{\Mbar^{\scale{\square}{0.6}}}

\newcommand{\cNc}{\cN^{\scale{\square}{0.6}}}
\newcommand{\scale}[2]{\scaleobj{#2}{#1}}
\newcommand{\cBc}{\cB^{\scale{\square}{0.6}}}
\newcommand{\cEc}{\cE^{\scale{\square}{0.6}}}
\newcommand{\cTc}{\cT^{\scale{\square}{0.6}}}
\newcommand{\xla}[1]{\xleftarrow{#1}}
\newcommand{\obj}{\normalfont\text{obj}}
\newcommand{\scfb}{\scF^{\,\bc}}
\newcommand{\bc}{\fb}
\newcommand{\idsimp}{\sigma_{\ide}}
\newcommand{\gt}{\,\wh{\g}\,}

\usepackage[utf8]{inputenc}
\usepackage[english]{babel}
\usepackage{amsmath} 
\usepackage{amssymb}
\usepackage{amsfonts}
\usepackage{mathtools}
\usepackage{float}
\usepackage{amsthm}
\usepackage[shortlabels]{enumitem}
\usepackage[dvipsnames]{xcolor}
\usepackage{tikz-cd}
\usepackage{subfiles}
\usepackage[thinlines]{easytable}
\usepackage{scalerel}[2016/12/29]
\usetikzlibrary{tqft}
\usepackage[colorinlistoftodos]{todonotes}
\usepackage[mathscr]{euscript}
\usepackage{extarrows}
\usepackage[pagebackref=true]{hyperref}
\usepackage{comment}
\usepackage{abraces}
\usepackage{upgreek}

\usepackage{appendix}
\usepackage{bm}
\hypersetup{
    colorlinks=true,
    linkcolor=Blue,
    filecolor=red,      
    urlcolor=teal,
    citecolor=NavyBlue,}

\DeclareFontFamily{U}{rcjhbltx}{}
\DeclareFontShape{U}{rcjhbltx}{m}{n}{<->rcjhbltx}{}
\DeclareSymbolFont{hebrewletters}{U}{rcjhbltx}{m}{n}
\DeclareMathSymbol{\mem}{\mathord}{hebrewletters}{109}
\DeclareMathSymbol{\nun}{\mathord}{hebrewletters}{110}

\newsavebox{\pullback}
\sbox\pullback{%
	\begin{tikzpicture}%
		\draw (0,0) -- (1ex,0ex);%
		\draw (1ex,0ex) -- (1ex,1ex);%
\end{tikzpicture}}

\tikzset{
	immersion/.tip={Glyph[glyph math command=looparrowleft, swap]}
}

\newtheorem{theorem}{Theorem}[section]
\newtheorem{lemma}[theorem]{Lemma}
\newtheorem{cor}[theorem]{Corollary}
\newtheorem{proposition}[theorem]{Proposition}
\newtheorem{conjecture}{Conjecture}

\newtheorem{construction}[theorem]{Construction}

\theoremstyle{definition}
\newtheorem{definition}[theorem]{Definition}
\newtheorem{property}[theorem]{Property}
\theoremstyle{remark}
\newtheorem{remark}[theorem]{Remark}

\newtheorem{example}[theorem]{Example}
\newtheorem{nota}[theorem]{Notation}
\newtheorem*{notation*}{Notation}

\numberwithin{equation}{section}

\makeatletter
\@addtoreset{equation}{section}
\@addtoreset{theorem}{section}
\makeatother

\newcommand{\Addresses}{{
  \bigskip
  \footnotesize
\textsc{A. Hirschi, Université Paris Cité, Sorbonne Université, CNRS, IMJ-PRG, F-75005 Paris, France}\par\nopagebreak
\text{ORCID}: \texttt{0000-0002-2392-7875}\\
  
  \textsc{K. Hugtenburg, School of Mathematics and Statistics, University of St Andrews}\par\nopagebreak
  \text{ORCID}: \texttt{0000-0002-7823-7126}
}}

\makeatletter 
\def\l@subsection{\@tocline{2}{0pt}{2pc}{6pc}{}} \makeatother

\usepackage{microtype} 
\begin{document}

\title[Open-closed Deligne--Mumford field theories II]{Open-closed Deligne--Mumford field theories: construction}
\author{Amanda Hirschi}\author{Kai Hugtenburg}

\begin{abstract} 
 Open-closed Deligne--Mumford field theories are chain-level field theories based on moduli spaces of stable curves with boundary. We associate to a relatively spin embedded Lagrangian $L \sub (X,\omega)$ such an open-closed DMFT. It extends the Fukaya $A_\infty$ algebra to curves of arbitrarily high genus and with arbitrarily many boundary components and is unique up to homotopy. This is the first step in proving Kontsevich's conjecture that the Fukaya category determines the Gromov--Witten invariants of $X$, following a strategy delineated by Costello.
\end{abstract}

\maketitle
\tableofcontents

\section{Introduction}

\subsection{Context}
Cohomological field theories were defined by Kontsevich--Manin, \cite{KM94}, in order to describe the algebraic structure of Gromov--Witten invariants and are inspired by the seminal work of Witten, \cite{Wi91}. In this paper we define a generalisation of this notion, assigning chain-level operations to chains on moduli spaces of stable curves with boundary. We call this algebraic structure an open-closed Deligne--Mumford field theory to indicate that the underlying properad differs from the properad previously used to define open-closed field theories.
Open-closed field theories were defined by Moore and Segal \cite{Moo01,Seg01} and have been widely applied in as diverse fields as string topology, mathematical physics, algebraic geometry and symplectic topology.

As with previous constructions of algebraic structures in symplectic topology, we use moduli spaces of pseudo-holomorphic curves. They were introduced to symplectic topology by Gromov, \cite{Gro85}, and have become the main tool to prove rigidity results. The first invariants based on counts of pseudo-holomorphic curves are the aforementioned Gromov--Witten invariants as well as Lagrangian Floer theory, the latter based on genus zero curves satisfying Lagrangian boundary conditions, \cite{Flo88}. 
In this case, the geometry of moduli spaces of holomorphic discs gives rise to an $A_\infty$ category, the Fukaya category defined in \cite{FOOO,Sei08}, respectively an $A_\infty$ algebra when considering a single Lagrangian. Higher-genus Lagrangian Floer theory has been only studied in special cases such as in \cite{Liu20,ES24}, partially because a good algebraic framework was lacking. 

In \cite{Cos07}, Costello defines the notion of an open-closed topological conformal field theory (TCFT), and shows that any smooth proper Calabi--Yau $A_\infty$ category $\scA$ gives rise to such a structure, yielding an action of the modular properad on the Hochschild homology of $\scA$. This is also discussed in \cite{KS09} and \cite{WW16,Wah16}. Several papers, such as \cite{Cos09} or \cite{CalTuCEI,CLT2018,AT22} show how to extract invariants from a TCFT, called \emph{categorical enumerative invariants}, given additionally a splitting of the associated Hodge filtration. There is also recent work on defining \emph{open(-closed) categorical enumerative invariants} from an open-closed TCFT, see \cite{Ulm25,Ulm26}. 
Crucially, Costello connects open-closed TCFTs to one of the most important conjectures in symplectic geometry, the Homological Mirror Symmetry conjecture of Kontsevich, \cite{KonICM}, by positing them as a key tool for obtaining the Gromov--Witten theory of a symplectic manifold from its Fukaya category.
In order to obtain a cohomological field theory from a closed TCFT, one needs to trivialise the $S^1$ action coming from rotating the interior boundary punctures, as discussed by \cite{Des22} for properads and \cite{DC14,OV20} in the operadic case.

\subsection{Main results}
This paper modifies Costello's definition by replacing the modular properad of smooth framed surfaces by modular properad of moduli spaces of stable curves with boundary, carrying marked points. As in \cite{Cos07}, we do not phrase our algebraic structures in terms of an algebra over this modular properad but as a functor from a symmetric monoidal dg category $ C_*(\Mbar_{oc})$ whose objects are pairs of natural numbers and morphisms are chains on said moduli spaces of curves with appropriate numbers of incoming/outgoing interior and boundary marked points. We allow for disconnected curves, whence disjoint union gives rise to a symmetric monoidal structure on $\dmc$. See \S\ref{subsec:base-category} for the precise definition and \S\ref{sec:field-theories} for a non-technical discussion. We can now give the main definition of this paper.

\begin{definition}\label{} An \emph{open-closed Deligne--Mumford field theory} (DMFT) is a symmetric monoidal $A_\infty$ functor $\scF\cl C_*(\Mbar_{oc}) \to \text{Ch}_\Lambda$ to the category of chain complexes.
\end{definition}

\noindent
A symmetric monoidal structure for an $A_\infty$ functor consists of certain natural transformations, defined in \S\ref{subsec:symmetric-monoidal}, which are inspired by \cite{amorim2016}. It is more involved than the definition of symmetric monoidal dg functors, and our construction relies on the fact that $C_*(\Mbar_{oc})$ is a dg category.

A first important property of open-closed DMFTs is that they define an $A_\infty$ structure on the chain complex $\scF(0,1)$, cf. Lemma~\ref{lem:curved-algebra-from-dmft}. This mirrors a similar statement in \cite{Cos07} for open-closed TCFTS. In contrast to open-closed TCFTs, an open-closed DMFT also directly yields a chain-level lift of a cohomological field theory restricting to moduli spaces of closed Riemann surfaces, cf. \S\ref{sec:field-theories}.

To be more precise, just as the Fukaya algebra of a Lagrangian is usually a curved $A_\infty$ algebra, the algebraic structure we naturally obtain from moduli spaces of pseudo-holomorphic curves forces us to consider the more general notion of a \emph{curved} open-closed DMFT, which is defined on the category $\dmc$ but does not satisfy the $A_\infty$ functor relations. We show that that curved open-closed DMFTs can still be understood as $A_\infty$ functors, albeit from an extension $\cdmc$ of the dg category $C_*(\Mbar_{oc})$; see Definition~\ref{def:OCbar} in \S\ref{subsec:base-category}.
 
The main result of this paper is the construction of a curved open-closed DMFT based on moduli spaces of pseudo-holomorphic stable maps with boundary on the Lagrangian. Since we only require the Lagrangian to be relatively spin, embedded and weakly unobstructed, we use virtual perturbations throughout. These geometric foundations, in particular coherent systems of global Kuranishi charts, are built in \cite{HH25}. 

\begin{intthm}[Theorem~\ref{thm:dmft-exists}]
	\label{thm:OCDMFT associated to L}
	Suppose $(X,\omega)$ is a closed symplectic manifold and $L \sub X$ an embedded Lagrangian equipped with a relative spin structure.
	\begin{enumerate}[label=\alph*),leftmargin=20pt,ref=\alph*]
		\item\label{curved case} $(X,L)$ admits a semi-strict operadic curved open-closed Deligne--Mumford field theory $\scF_L$ whose associated $A_\infty$ algebra $\scA_L$ is the Fukaya $A_\infty$ algebra of $L$.
		\item If $L$ is equipped with a weak bounding cochain, then it admits a semi-strict operadic \emph{uncurved} open-closed Deligne--Mumford field theory.
	\end{enumerate}
In either case, the associated closed DMFT yields the Gromov--Witten theory of $X$.
\end{intthm}

\noindent
Open-closed DMFTs allow for operations coming from curves with any number of outputs and inputs in return for relaxing the assumption that $\scF$ is a dg functor. By relying on a specific chain model, we can achieve strictness of $\scF_L$ for certain subcategories, as indicated by the properties semi-strict and operadic, which are defined in \S\ref{subsec:def-of-dmft}.
A key feature of the property of being operadic is that the open-closed DMFT is compatible with forgetting incoming marked points. In particular, $\scA_L$ is a cyclic $A_\infty$ algebra.
Restricting the open-closed DMFT to stable spheres, we have the following other immediate consequence.

\begin{cor}\label{}
	Any compact symplectic manifold $(X,\omega)$ admits a strict action of the operad of stable genus-$0$ curves on $\Omega^*(X;\Lambda)$ lifting the quantum cup product.
\end{cor}
One should compare this with \cite{AGV}, where a strict chain-level action of the operad of framed (stable) genus-$0$ curves on Hamiltonian Floer theory is constructed.

While we highlight the restrictions to discs and closed curves, the (curved) open-closed DMFT of Theorem~\ref{thm:OCDMFT associated to L} contains far more information. It can be considered a universal curve-counting theory in the sense that it is agnostic about the complexity of the curve. In particular, important structures such as (chain-level lifts of) the open-closed map
\begin{equation}
	\label{} HH_*(\scA_L)\to \text{QH}^*(X,\omega)
\end{equation}
and the closed-open map 
\begin{equation}
	\label{} \text{QH}^*(X,\omega) \to HH^*(\scA_L)
\end{equation}
are part of the data of an open-closed DMFT.

The theory of deformations of $A_\infty$ structures carries over to open-closed DMFTs. Recall from \cite{FOOO} that a \emph{weak bounding cochain} of a unital curved $A_\infty$ algebra is a Maurer--Cartan element $\bc$ so that the curvature of the deformed $A_\infty$ algebra is a multiple of the unit. The second assertion of Theorem \ref{thm:OCDMFT associated to L} follows from the first one by establishing the following result.

\begin{proposition}[Proposition~\ref{prop:deformed-dmft}]\label{prop:deformation-of-dmft}
	If $\scF$ is a suitable curved open-closed DMFT and $\bc$ a weak bounding cochain of the associated curved $A_\infty$ algebra, then the deformed curved DMFT $\scfb$ induces an (uncurved) open-closed DMFT. 
\end{proposition}

\noindent
While it may seem surprising that weak bounding cochains are sufficient to eliminate curvature in this more general context, the underlying reason is that discs are the only curves with boundary that are unstable if they carry only one or two boundary marked points.

The construction of the curved open-closed DMFT of Theorem~\ref{thm:OCDMFT associated to L}\eqref{curved case} requires numerous choices, in particular a choice of $\omega$-tame almost complex structure $J$. We establish independence from these choices up to homotopy by geometrically constructing the required homotopies.

\begin{intthm}[Theorem~\ref{thm:independence-of-dmft}]\label{thm:independence}
	The (curved) open-closed DMFTs of Theorem~\ref{thm:OCDMFT associated to L} are independendent of all auxiliary choices (and gauge equivalence of bounding cochains) up to homotopy of DMFTs.
\end{intthm}

\noindent
The basis of the operations of the open-closed DMFT of Theorem~\ref{thm:OCDMFT associated to L} as well as the symmetric monoidal structure and the homotopies of Theorem~\ref{thm:independence} are the correspondences of \S\ref{subsec:correspondences}. 
We construct them using moduli spaces of stable maps with Lagrangian boundary. Since we have to incorporate chains coming from moduli space of stable curves, these correspondences are singular in general -- even after using regularisation techniques. Inspired by \cite{She25}, we use analytic chains. This means we can apply the real analytic resolution of singularities algorithm of \cite{BM08}, extended to real analytic orbifolds with corners in \S\ref{subsec:singular-integration}.

 \begin{remark}\label{}
 	We expect the open-closed DMFT of Theorem~\ref{thm:OCDMFT associated to L} to be defined in any category in which pull-push operations based on correspondences make sense. In the same vein, while we define an $A_\infty$ functor, our construction should also yield an $\infty$-functor to suitable $\infty$-categories without further geometric work. 
 \end{remark}

\subsection{Applications}
  Mirror symmetry states that for every Calabi-Yau variety $X$ there should be a mirror Calabi-Yau variety $Y$ such that, roughly speaking, the symplectic invariants of $X$ are `isomorphic' to complex or algebraic invariants of $Y$. The first example of this was given for the quintic threefold $X$ in $\bC\bP^4$ by \cite{Can}, where a correspondence between the Gromov--Witten invariants of $X$ and the period integrals of the mirror quintic $Y$ was proven. This is known as \emph{enumerative mirror symmetry}. Kontsevich conjectured this should be the consequence of a more fundamental equivalence, called \emph{homological mirror symmetry}. 

\begin{conjecture}[\cite{KonICM}]For mirror varieties $X$ and $Y$, there exists an equivalence
	\begin{equation*}
		\cF uk(X) \cong D^bCoh(Y)
	\end{equation*}
	between the Fukaya category of $X$ and the derived category of coherent sheaves on $Y$.
	Moreover, this homological mirror symmetry implies enumerative mirror symmetry.
\end{conjecture}

\noindent
In particular, on the symplectic side, this would imply 

\begin{conjecture}[\cite{KonICM}]
	\label{con:Fuk-to-GW}
	The Fukaya category determines the Gromov-Witten invariants of the symplectic manifold.
\end{conjecture}

\noindent
In \cite{GPS15,GS25} this was proven in genus $0$ for Calabi--Yau varieties, under suitable assumptions, using \emph{Hodge-theoretic mirror symmetry}. One of the major technical ingredients is their proof that an $S^1$-equivariant enhancement of the open-closed map intertwines the Getzler--Gauss--Manin connection with the quantum connection. This result was also obtained in \cite{Hug24} under very restrictive regularity assumptions. With similar methods it was shown in \cite{Hug2} that one can also extract enumerative invariants of Lagrangians, so called \emph{genus $0$ open Gromov--Witten invariants}, from the Fukaya category under suitable regularity assumptions. The present paper removes these regularity requirements from these papers.

In order to prove Conjecture \ref{con:Fuk-to-GW} in arbitrary genus, Costello \cite{Cos07} associates categorical enumerative invariants to a (weakly Calabi--Yau) $A_\infty$ category, which should recover the usual Gromov--Witten invariants in the case of the Fukaya category. These have been shown to be quite computable, \cite{CLT2018}, but it remains to be shown that they recover `geometric' enumerative invariants. As a first step in defining categorical enumerative invariants, Costello proved that any Calabi--Yau $A_\infty$ category determines an open-closed TCFT.

When the Fukaya category is generated by a single Lagrangian, we conjecture that the open-closed TCFT associated to the open-closed DMFT of Theorem~\ref{thm:OCDMFT associated to L} is equivalent to the universal TCFT associated to the Fukaya category by Costello. This conjecture, including comparing open-closed DMFTs and TCFTs, and the construction of a geometric open-closed DMFT for the whole Fukaya category will be addressed in future work. 

\begin{remark}
	It is unknown whether the closed-open map of an arbitrary closed symplectic manifold $X$ is necessarily an isomorphism; in fact, it is unknown whether the Fukaya category is non-trivial. In situations where homological mirror symmetry has been established, one can use the generation criterion of \cite{AbouzaidGeneration,Gan16} to show that the closed-open map is an isomorphism.
\end{remark}

\begin{remark}
	In order to obtain Gromov--Witten invariants from the Fukaya category, one needs to equip the category with the correct trivialisation of the $S^1$-action/splitting of the non-commutative Hodge-de-Rham spectral sequence. When it is an isomorphism, the closed-open map provides this trivialisation. In the Calabi--Yau case, \cite{GPS15} use intrinsic properties of the Fukaya category to determine the correct splitting. For a general symplectic manifold, one might need additional data.
\end{remark}

\subsection{Field theories and a construction in symplectic topology}\label{sec:field-theories}
In this section we describe open-closed DMFTs in more detail and how they relate to more classical algebraic structures, and give an overview of the construction. 

\subsubsection{Cohomological and Deligne--Mumford field theories}
Before the first construction of Gromov--Witten invariants was given, Kontsevich and Manin described the algebraic structure they should yield, \cite{KM94}, inspired by the seminal work of Witten, \cite{Wi91}. Let $\Mbar_{g,n}$ denote the Deligne--Mumford compactification of the moduli space of (closed) Riemann surfaces of genus $g$ with $n$ marked points. A \emph{cohomological field theory} consists of a graded vector space $V$ equipped with a graded symmetric pairing $\eta$ on $V$, a distinguished element $e \in V^0$, called the \emph{unit}, and a family of linear maps 
\[
\Omega_{g,n}\cl  V^{\otimes n} \rightarrow H^*(\Mbar_{g,n},\QQ),
\]
which satisfies a list of recursion relations, the CohFT axioms, cf. \cite{KM94}. Crucially, this can be equivalently phrased as the data of a strong symmetric monoidal functor. To see this, let $\Mbar_c$ is a topological category whose set of objects are the natural numbers $\bN_0$ including $0$ and whose morphism spaces are moduli spaces of Riemann surfaces with $k^-$ incoming and $k^+$ outgoing marked points. 
We allow for moduli spaces in the unstable range, in which case we take the moduli space to be a point. These additional `unstable' morphisms give rise to the unit $e$ (the sphere with one outgoing marked point) and the pairing $\eta$ (the sphere with two incoming marked points), and to identity morphisms. Composition is given by clutching Riemann surfaces at incoming and outgoing marked points, with the composition with unstable morphisms described in Definition \ref{de:open-closed-curve-category}. For example, composition with the the unit forgets an incoming marked point. $H_*(\Mbar_c)$ denotes the category obtained from $\Mbar_c$ by taking homology of each morphism space.

\begin{lemma}\label{}
	 CohFTs are in bijection to symmetric monoidal functors $\scF\cl H_*(\Mbar_c) \rightarrow \text{Vect}_{\QQ}.$
\end{lemma}

\begin{proof}
Given such a functor $\scF$, set $V \coloneqq \scF(\{1\})$ and $\Omega_{g,n}(\alpha_1,\dots, \alpha_n)(\sigma) := \scF_{\sigma}(\alpha_1, \dots, \alpha_n)$ for $[\sigma] \in H_*(\Mbar_{g,n})$ and $\alpha_1, \dots, \alpha_n \in V$. The unit $e$ and the pairing are the image of the unique non-zero unstable element in $H_*(\Mbar_c(0,1))$, respectively in $H_*(\Mbar_c(2,0))$. The CohFT axioms follow from the functoriality of $\scF$. The fact a CohFT determines such a functor by setting $\scF(k) =V^{\otimes k}$, and using the pairing $\eta$ to turn all outputs into inputs.
\end{proof}
 
 \noindent
 In symplectic topology, CohFTs arise from moduli spaces of stable maps and haven been constructed using several different frameworks. Instead of working over $\QQ$, one usually has to work over a \emph{Novikov ring} $\Lambda_\QQ$ over $\QQ$ to deal with non-convergence issues in Gromov--Witten theory.
Let us briefly outline the construction of the functor $\scF\cl H_*(\Mbar_c)\to  \text{Vect}_{\Lambda}$ in this setting. Let $(X,\omega)$ be a closed symplectic manifold and $J$ a tame almost complex structure.
 By \cite{HS22}, the moduli space $\Mbar^{\, J,\beta}_{g,n}(X)$ of $J$-holomorphic stable maps representing $\beta\in H_2(X;\bZ)$ admits a global Kuranishi chart. That is, they construct an orbifold $\cT_\beta$, an orbibundle $\cE_\beta \rightarrow \cT_\beta$, and an orbisection section $s_\beta$ of $\cE_\beta$ so that $s\inv_\beta(0)\cong \Mbar^{\, J,\beta}_{g,n}(X)$. The construction directly yields corresponding charts for moduli spaces $\Mbar^{\, J,\beta}_{g;k^-,k^+}(X)$ of stable maps, where the marked points are partitioned into incoming and outgoing ones. Evaluating at these marked points, we obtain a correspondence
 \begin{equation}\label{eq:closed-correspondence}\begin{tikzcd}
 X^{k^-}	& \Mbar^{\, J,\beta}_{g;k^-,k^+}(X)\arrow[r,"{\eva^+}"]\arrow[l,"{\eva^-}"]&X^{k^+}
 		  \end{tikzcd} \end{equation}
 which is, moreover, equipped with a map $\stb$ to the moduli space $\Mbar_{g,k^-,k^+}$ of stable curves. Both the evaluation maps and the stabilisation map extend to the thickening $\cT$. Then, we define for $\sigma \in H_*(\Mbar_c(k^-,k^+))$ the map
 \begin{align}
	\label{eq:CohFT-def-symplectic}
	\notag\scF_\sigma\cl  H^*(X^{k^-};\Lambda) &\rightarrow H^*(X^{k^+};\Lambda),\\
	\alpha &\mapsto \sum_{\beta}\eva^+_*(\stb^*\pcd(\sigma)^* \cup (\eva^-)^*\alpha \cup s_\beta^*\tau_{\cE_\beta})Q^\beta,
	\end{align}
	where $\tau_\beta$ is the Thom class of $\cE_\beta$. By \cite{HS22}, this map is independent of auxiliary choices, in particular of $J$. In \cite{Hir23}, the CohFT axioms are shown.
	
	\smallskip
	 
Following Costello, we define a chain-level version of cohomological field theories as follows.

\begin{definition}
	A \emph{strict} Deligne--Mumford Field Theory (DMFT) is an h-split symmetric monoidal dg functor $\scF\cl C_*(\Mbar_{c}) \rightarrow \text{Ch}_{\Lambda}$ to the category of chain complexes of $\Lambda$-modules.
\end{definition}

\noindent
The property h-split requires that $\scF(k)$ is quasi-isomorphic to $\scF(1)^{\otimes k}$ instead of being isomorphic. This relaxation is required in the setting of Gromov-Witten theory because the K\"unneth isomorphism is usually not an isomorphism on the chain level.
Using the correspondences~\eqref{eq:closed-correspondence}, one could try to emulate~\eqref{eq:CohFT-def-symplectic} in any chain model that allows for pull-push operations. We have chosen to work with the chain complex of differential forms $\Omega^*(X;\Lambda)$ with (completed) Novikov coefficients, so we will describe the issues in this setting. 

\begin{enumerate}[label=(\Roman*),leftmargin=20pt,ref=\Roman*]
	\item\label{well-defined} The first issue concerns the definition of the chain-level operations since~\eqref{eq:CohFT-def-symplectic} requires Poincar\'e duality on $\Mbar_c$. We solve this problem geometrically by modifying the correspondence~\eqref{eq:closed-correspondence} using simplices in $\Mbar_c$: given a simplex $\sigma \cl \Delta^m \to \Mbar_c(k^-,k^+)$ we define the fibre product 
	$$\cK_\sigma := \cK\times_{\Mbar_c}\Delta^m.$$ 
	Since the stabilisation map is not a submersion, this is usually not a well-defined derived orbifold chart. However, assuming $\sigma$ to be \emph{analytic}, we can apply the resolution of singularities algorithm of \cite{BM08} to $\cT_\sigma$ to obtain a smooth orbifold $\wt\cT_\sigma$ with a smooth map $p_\sigma \cl \wt\cT_\sigma \to \cT$. Thus, we can define $\scF_{\sigma}\cl \Omega^*(X^{k^-};\Lambda) \rightarrow \Omega^*(X^{k^+};\Lambda)$ by
	\begin{equation}\label{eq:dmft-operations}
		\scF_\sigma(\alpha) \;=\; (\eva^+\g p_\sigma)_*p_\sigma^*((ev_{\sigma}^-)^*\alpha \wedge s^*\tau).
	\end{equation}
	We show in Lemma~\ref{lem:multiplication-up-to-codim} that the choice of $\wt\cT_\sigma$ does not affect $\scF_\sigma$.
	\item\label{functoriality} The second issue relates to functoriality.  
	The global Kuranishi charts of \cite{HS22} do not depend on the auxiliary choices \emph{up to equivalence}, which is sufficient to achieve strict compatibility/functoriality of the operation~\eqref{eq:CohFT-def-symplectic}. This is no longer true on the chain level: both the choice of global Kuranishi charts as well as the choice of Thom form, i.e., the representative of the Thom class, are relevant for the chain-level pull-push operation. Our strategy is to choose the required auxiliary data for moduli spaces of stable maps arbitrarily and then modify the operations \eqref{eq:dmft-operations} in order to achieve functoriality. This modification is also defined in terms of a correspondence and uses the fact that the global Kuranishi chart restricted to the boundary is cobordant to the fibre product global Kuranishi chart. Doing this rigorously leads to the notion of a system of cubical cobordisms, where we construct for each stable map graph $\Gamma$ a global Kuranishi chart $\cKc_\Gamma$ for $I^{E(\Gamma)}\times\Mbar_\Gamma(X)$, where \[\Mbar_\Gamma(X)\sub \p{v\in V(\Gamma)}{\Mbar_{g_v,k_v}^{\,J,\beta_v}(X)}\] is the fibre product determined by the edges of $\Gamma$. In its simplest form, when $\Gamma$ has a unique edge, $\cKc_\Gamma$ is a cobordism from $\del_\Gamma\cK_\beta$ to the fibre product of global Kuranishi charts. Given these charts, we can replace~\eqref{eq:dmft-operations} by 
	\begin{equation*}\label{} \scF_\sigma \coloneqq \s{\Gamma}{\scF_{\Gamma,\sigma}},\end{equation*}
		where $\scF_{\Gamma,\sigma}$ defined similarly to~\eqref{eq:dmft-operations} with $\cKc_\Gamma$ instead of $\cK_\beta$.
\end{enumerate}

\noindent
While the solution described in \eqref{functoriality} achieves strict functoriality in the stable range, strict functoriality of $\scF$ with respect to the unstable morphisms would imply that the pairing on $\Omega^*(X;\Lambda)$ has to be non-degenerate. Since the chain complex is infinite-dimensional, this is unachievable. The upshot is that one has to relax the condition that $\scF$ is a dg functor.

\begin{definition}
	A \emph{Deligne--Mumford Field Theory} (DMFT) is an h-split symmetric monoidal $A_\infty$ functor $\scF\cl C_*(\Mbar_{c}) \rightarrow \text{Ch}_{\Lambda}$.
\end{definition} 

\noindent
A DMFT is \emph{semi-strict} if the restriction of $\scF$ to the subcategory of stable chains $\Mbar_c^{st}$ is a dg functor. Refer to Definition \ref{def:symmetric-monoidal} for the definition of symmetric monoidal $A_\infty$.

\begin{theorem}
	Associated to any closed symplectic manifold $(X,\omega)$ is a semi-strict DMFT, which yields the Gromov--Witten CohFT of $X$ when passing to homology.
\end{theorem}

\noindent
While the construction of the higher homotopies is natural in our set-up, it requires considerable work, including the construction of a special chain model, to achieve semi-strictness. 

\subsubsection{Open-closed Deligne--Mumford field theories}
As their name suggests, open-closed DMFTs also incorporate operations from stable curves with boundary. To make this precise, let $\Mbar_{oc}$ be the category with objects pairs $y = (y_c,y_o)$ of nonnegative integers (or finite sets more generally) and morphism from $y$ to $y'$ given by possibly nodal disconnected curves with incoming/outgoing interior marked points labelled by $y_c \sqcup y'_c$ and incoming/outgoing boundary marked points labelled by $y_o\sqcup y_o'$. As before, we include moduli spaces of unstable curves as points. An \emph{open-closed Deligne--Mumford field theory} is an h-split symmetric monoidal $A_\infty$ functor \[
	\scF\cl \dmc \coloneqq C_*(\Mbar_{oc}) \rightarrow \text{Ch}_{\Lambda}.
	\]

\begin{remark}\label{}
	The key difference between an open-closed DMFT and the open-closed TCFT considered by Costello is that our moduli spaces are compactified and that we consider interior marked points as opposed to interior punctures. Moreover, Costello defines an open-closed TCFT to be a dg functor, whereas we allow for higher homotopies and consider $A_\infty$ functors.
 \end{remark}
 
 \noindent
One can restrict an open-closed DMFT to $C_*(\Mbar_c)$ to obtain a DMFT, and to the subcategory generated by discs with one output to obtain an $A_\infty$ algebra. They thus serve as a bridge between the closed and the open theory. 

\noindent
In practice (at least in symplectic topology), pseudo-holomorphic stable maps with Lagrangian boundary condition do not yield an open-closed DMFT.
Indeed, while the same issues~\eqref{well-defined} and~\eqref{functoriality} arise, the same solutions work in the case with boundary.
However, in the open setting, moduli spaces have codimension-$1$ boundary strata, which deform the de Rham differential on $\Omega^*(L;\Lambda)$. 
This deformation, called the Floer differential, need not square to zero because the Lagrangian might be \emph{curved}. Concretely, the bubbling off of unstable discs creates additional boundary strata that are not seen by moduli spaces of stable curves and which spoil the $A_\infty$ functor equations. 
To remedy this, we change the category $\dmc$ by allowing for \emph{unstable} curves in the definition of $\Mbar_{oc}$ that `have higher energy' and thus contribute differently than the other unstable curves. 

\begin{definition}
	Let $\cdmc$ be the category obtained by formally adjoining a morphism $D_0$ from $(0,0)$ to $(0,1)$ and a morphism $D_1$ from $(0,1)$ to $(0,1)$ to $C_*(\Mbar_{oc})$ and defining the differential on a morphism $\sigma$
	\[
	\partial_{oc} \sigma \coloneqq \del\sigma + (-1)^{|\sigma|}\pi^*\sigma \g_{*} D_0 - \sum_{i \in y'_o} D_1 \g_{i} \sigma + (-1)^{|\sigma|}\sum_{i \in y_o} \sigma \g_{i} D_1,
	\]
	where the chain $\pi^*\sigma$ is obtained by adding an additional incoming boundary marked point to the curves in the image of $\sigma$; see \ref{con:lifted-simplex} for a precise construction.
\end{definition}

\noindent
We have thus arrived at the definition that is closest to the geometry we see in symplectic topology and which is what we construct in \S\ref{subsec:lagrangian-dmft}.

\begin{definition}
	A \emph{curved open-closed Deligne--Mumford field theory} is an h-split symmetric monoidal $A_\infty$ functor $\scF\cl \cdmc \rightarrow \text{Ch}_{\Lambda}$.
\end{definition}

\noindent
The uncurved open-closed DMFT in Theorem~\ref{thm:OCDMFT associated to L} is obtained from a curved one by two separate purely algebraic deformations, discussed in \S\ref{subsec:bounding-cochains-and-deformations} and \S\ref{subsec:compatible-with-d1} respectively.

\addtocontents{toc}{\protect\setcounter{tocdepth}{1}}

\subsection*{Acknowledgements}  We are very grateful to Nick Sheridan for invaluable suggestions and to Mohammed Abouzaid, Lino Amorim, Yash Deshmuk and Junwu Tu for useful discussions. A.H. thanks Hiro Lee Tanaka, and Nathalie Wahl for enlightening conversations.

Part of this work was completed while both authors were in residence at the Simons Center for Geometry and Physics as well as the Mittag-Leffler Institute. We thank both institutes for their hospitality. A.H. is supported by ERC Grant No.864919 and K.H. was partially supported by EPSRC Grant EP/W015749/1.

\section{Open-closed Deligne--Mumford field theories}\label{sec:dmfts}
\addtocontents{toc}{\protect\setcounter{tocdepth}{2}}
\noindent In \S\ref{subsec:base-category} we define the notion of a curved open-closed Deligne--Mumford field theory and discuss its properties, including how it defines a curved $A_\infty$ algebra.
 In \S\ref{subsec:bounding-cochains-and-deformations}, we define deformations of curved open-closed DMFTs and when a deformation yields an uncurved open-closed DMFT. 

\subsection{Chains on moduli spaces of curves}\label{subsec:base-category}
Given a collection $\cD$ of $D$-branes, Costello defines in \cite{Cos07} a category $\cM_\cD$, whose objects are tuples $y = (y_c,s,t\cl y_o\to \cD)$ of finite ordered sets $y_c$ and $y_o$, the latter equipped with the source and target maps $s,t$ to $\cD$. The morphism spaces are, roughly, smooth Riemannian surfaces with boundary, whose boundaries are divided into three types. We will only consider the case where $\cD$ consists of a single Lagrangian $L$ and define a variation of his category, which we denote by $\Mbar_{oc}$.

\begin{definition}\label{de:open-closed-curve-category} The category $\Mbar_{oc}$ enriched over oriented orbifolds with corners has objects given by pairs $y = (y_c,y_o)$ of finite ordered sets, a `closed' set and an `open' one
	\begin{itemize}[leftmargin=20pt]
		\item To define the space of morphisms from $y$ to $y'$, let
		\[\Mbar_{oc}^\bullet(y,y') \coloneqq\djun{\substack{g \geq 0,\\h\geq \min\{|y_o|+|y_o'|,1\}}}\Mbar_{g,h;y,y'}\]
		be the disjoint union of moduli spaces of connected stable curves. The marked points are partitioned into incoming and outgoing marked points, labelled by $y_c$ and $y_o$, respectively $y'_c$ and $y'_o$. If $(g,h,|y_c|+|y_c'|,|y_o|+|y'_o|)$ is in the unstable range, we formally define the moduli space $\Mbar_{g,h;y,y'}$ to be a point.
		Then, set 
		\begin{equation*}\label{eq:morphisms}
			\qquad\Mbar_{oc}(y,y') \coloneqq \djun{\substack{y = y_1 \sqcup \dots y_m,\\y' = y'_1 \sqcup \dots\sqcup y'_m}}\prod_{i =1}^{m}\,\Mbar_{oc}^\bullet(y_i,y_i'), 
		\end{equation*}
		that is, we allow for curves that are not necessarily connected.
		\item The composition maps
		$$\psi\cl \Mbar_{oc}(y',y'')\times\Mbar_{oc}(y,y')\to \Mbar_{oc}(y,y'')$$
		in the stable range are defined by clutching at the marked points associated to $y'$.
		
		We define the composition for the allowed unstable cases for each case separately; observe that we write the unstable moduli space, although in each case it is simply a point.
		\begin{enumerate}[label=(C\arabic*),leftmargin=25pt,ref=C\arabic*]
			\item\label{identity} The morphism associated to  $\Mbar_{0,0;(1,0),(1,0)}$, respectively $\Mbar_{0,0;(1,0),(1,0)}$ is the identity and so the composition with it is simply the identity.
			\item\label{forget-interior-outgoing} The composition `with' $\Mbar_{0,0;(1,0),0}$
			is the map forgetting the respective outgoing interior marked point, while the composition with $\Mbar_{0,0;(0,1),0}$ forgets the respective incoming interior marked point.
			\item\label{forget-boundary-outgoing} The composition with $\Mbar_{0,1;(0,1),0}$ ($\Mbar_{0,1;(0,1),0}$) forgets an incoming (outgoing) boundary marked point.
			\item\label{self-clutching} For $\{i,j\} = \{0,2\}$, the composition with $\Mbar_{0,0;i,j}$ is given by self-clutching at incoming, respectively outgoing interior marked points. Meanwhile, composition with $\Mbar_{0,1;(0,i),(0,j)}$  is given by self-clutching at incoming, respectively outgoing boundary marked points.
			\item\label{collapsed-circle} For $\{j^-,j\} = \{0,1\}$, the composition $\Mbar_{0,1;(j^-,j),0}$	is the inclusion of the boundary stratum where a boundary circle collapses.
		\end{enumerate}
	\end{itemize}
\end{definition}

\noindent
 By construction, the category $\Mbar_{oc}$ admits a symmetric monoidal structure given by the disjoint union on objects and morphism spaces. Moreover, $\Mbar_{oc}$ has several important subcategories.

\begin{definition}\label{de:stable-subcategory} We define $\Mbar_{oc}^{\stb}$ to be the \emph{stable subcategory}, that is, the subcategory with the same set of objects but where each morphism space only comprises moduli spaces in the stable range, together with the `identity' morphisms, that is, the spaces $\Mbar_{0,0;(1,1),0}$ and $\Mbar_{0,1;0,(1,1)}$.
\end{definition}

\noindent
 In particular, $\Mbar_{oc}^{\stb}$ does not have the compositions \eqref{forget-interior-outgoing}-\eqref{collapsed-circle}.
 
\begin{definition}
	\label{de:operadic-subcategory}
	Let $\Mbar_{oc}^{\opr}$ be the \emph{operadic subcategory}, which is the subcategory with the same set of objects, but where the morphism spaces are generated (using the symmetric monoidal structure and compositions) by connected surfaces with at most one output marked point, which is a boundary marked point if the surface is open.
\end{definition}

\noindent
Morphisms in the operadic subcategory are surfaces for which every connected component has at most one output marked point.

\begin{definition}\label{de:closed-subcategory} The \emph{closed subcategory} $\Mbar_c$ of $\Mbar_{oc}$ has objects $y = (y_c,\emst)$ and morphisms given by closed surfaces. The \emph{open subcategory} $\Mbar_o$ is the full subcategory on the objects $y = (\emst,y_o)$.
\end{definition}

\noindent
While one could directly work with the topological category $\Mbar_{oc}$, we use the category of analytic singular chains on it. In genus zero, there is an obvious choice of chains; in higher genus, the morphism spaces are orbifolds. The standard definition of singular chains on orbifolds (taking the geometric realisation of the diagonal of the nerve) is too unwieldy for our purposes. Since we are only interested in chains with rational or real coefficients, we can use the following model. Note that we are not working with orbifold (co)homology as defined in \cite{CR04}.

\begin{definition}\label{de:chains-on-orbifolds}
	Let $X = [X_1 \rightrightarrows X_0]$ be a topological groupoid so that the source and target maps $s,t \cl X_1 \to X_0$ are topological submersions. We define its \emph{singular chain complex} to be 
	\[C_*(X) = C_*(X;\bR) \coloneqq \coker(t_*-s_* \cl C_*(X_1;\bR)\to C_*(X_0;\bR)).\] 
\end{definition}

\noindent
The definitions of the subcomplex of \emph{analytic} or \emph{smooth} singular chains on $X$, in the context where it is well-defined, are immediate.

\begin{lemma}\label{lem:chains-on-orbifolds-well-defined}
	This definition is homotopy equivalent to the definition of singular chains in \cite{Be04} for any \'etale proper topological groupoid. In particular, there exists a canonical isomorphism
	$$H_*(X)\; \cong\; H_*(|X|;\bR).$$
	Moreover, any refinement $X\to Y$ of compact Lie groupoids induces an isomorphism $C_*(X)\to C_*(Y)$ of smooth singular chains.
\end{lemma}

\noindent
For our purposes it is important that the chain complex itself is invariant up to equivalence.  

\begin{proof}
	\cite{Be04} defines the singular chains, denoted here $C^B_*(X)$, on a topological groupoid $X$ by taking $X_p = N_p(X)$ to be the $p$th level of the nerve of $X$ and letting $ C^B_*(X)$ be the totalisation of the bicomplex $C_*(X_*)$. Then, $C_*(X) = E^1_{*,0}$ is the first non-zero column of the first page of the spectral sequence associated to $C_*(X_*)$.  Since $X$ is \'etale proper, all isotropy groups are finite. Thus, $E^2_{p,q} = 0$ for $p \neq 0$, which shows that the spectral sequence collapses at the second page and we have $H^B_*(X)\cong H_*(X)$. The second claim then follows from \cite[Proposition~36]{Be04}. The last claim is a straightforward verification, using Ehresmann's Lemma.
\end{proof}

\noindent
We will also need the following property in \S\ref{subsec:def-of-dmft}.

\begin{lemma}\label{lem:fibre-product-well-defined}
	Suppose $f \cl X\to Y$ is a submersion of \'etale proper Lie groupoids. Then, for any simplex $[\sigma]\in C_k(X)$, the (orbifold) fibre product $\Delta^k{}_{\sigma}\times_f X$ does not depend on the choice of representation $\sigma\cl\Delta^k \to Y$.
\end{lemma}

\begin{proof}
	If $\sigma'$ is another representative of $\sigma$, then there exists $\rho \cl \Delta^k \to Y_1$ so that $t_*\rho= \sigma$ and $s_*\rho = \sigma'$. Thus, $\rho$ induces an isomorphism $\Delta^k{}_\sigma\times_f X \to\Delta^k{}_{\sigma'}\times_f X$ given by
	\[ (v,\varphi,x)\mapsto (v,\varphi\g \rho(v),x),\]
	on objects and the identity on morphisms.
\end{proof}

\begin{example}\label{ex:concrete-groupoids}
	The orbifold $\Mbar_{g,h;y,y'}$ can be represented as the almost free $\PU(N+1)$ (respectively $\text{PO}(N+1)$)-manifold $\Mbar^*_{g,y\sqcup y'}(\bC P^{N};d)$ if $h = 0$, respectively $\Mbar^*_{g,y\sqcup y'}(\bC P^{N},\bR P^N;d)$ if $h > 0$, where $N$ and $d$ only depend on $g,h,y,y'$. In the closed case this is explained in \cite{ACG-moduli}; in the open case it follows easily by applying \cite[\S 2]{HH25} to stable curves. We will use the associated \'etale proper Lie groupoids for the remainder of the paper.
\end{example}

\noindent
Clearly, the restriction of an analytic simplex to one of its faces is analytic as well. Thus, we can write $C_*^{an}(Y;\bR)$ for the subcomplex of the singular chains on $Y$ generated by analytic smooth simplices.

\begin{lemma}\label{lem:analytic-homology} $H_*(C_*^{an}(Y;\bR)) = H_*(Y;\bR)$ for any compact analytic orbifold $Y$ with corners.
\end{lemma}

\begin{proof} By \cite[Theorem~VII.2.25]{GMT86}, a function $f \cl \Delta^m \to Y$ can be arbitrarily well approximated by an analytic function $g_\epsilon$ so that, if $f$ is analytic on an analytic subvariety $K$ of $\Delta^m$, we can choose $g_\epsilon|_K = f|_K$.
	To be precise, the result of \cite{GMT86} concerns analytic manifolds without corners; to carry it over to the case with corners, we embed $Y$ into some Euclidean space and extend $f$ to a continuous function in a neighbourhood of $\Delta^m$ in $\{x\in \bR^{m+1}\mid \sum_{i = 0}^{m}x_i=1\}$. To extend to the case where $Y$ is an orbifold, we represent $Y$ as a Lie groupoid $[Y_1\rightrightarrows Y_0]$ and argue with $Y_0$.
\end{proof}

\begin{definition}\label{de:chains-for-dmft}
The symmetric monoidal dg category $\dmc$ has the same objects as $\Mbar_{oc}$ with morphisms given by the cochain complexes
	\[\dmc^*(y,y') =  (C_{-*}(\Mbar_{oc}(y,y')),\del),\]
	of analytic singular chains on the Lie groupoids of Example~\ref{ex:concrete-groupoids}. The symmetric monoidal structure is given by disjoint union on objects, and on morphisms by applying the Eilenberg-Zilber map to the cross product of chains. The braiding $S$ of the symmetric monoidal structure is given by taking a disjoint union of spheres and discs with one incoming and one outgoing marked point each and the appropriate labelling of the marked points.
\end{definition}

\noindent
We have to consider a twisted version of $\dmc$. To this end, define the real line bundle
\begin{equation}\label{eq:det} \text{det} = \text{det}_{y,y'}\coloneqq \det(\bC^{\times y_c'})\dul\otimes\det(\bR^{\times y'_o})\dul\otimes\det(\delbar_{(\bC,\bR)})[-\chi+2|y'_c|+|y'_o|].\end{equation}
of degree zero and let $\orl(\det)$ be the associated $\bZ/2$-torsor.

\begin{lemma}\label{lem:composition-and-det} Any composition map $\psi\cl \Mbar_{oc}(y',y'')\times\Mbar_{oc}(y,y')\to \Mbar_{oc}(y,y'')$
	admits a canonical isomorphism 
	\begin{equation*}\label{eq:iso-of-det} 
		\text{det}^{\otimes r}_{y',y''}\otimes\text{det}^{\otimes r}_{y,y'}\to \psi^*\text{det}^{\otimes r}_{y,y''}
	\end{equation*} 
of degree $0$.
\end{lemma}

\begin{proof} This can be seen using the isomorphism in \cite{CZ24} between the determinant line on the clutched surface and the one on its normalisation.
\end{proof}

The category $\dmc_r$ has the same objects as $\Mbar_{oc}$ with morphisms given by 
\begin{equation*}\label{} \dmc_r^*(y,y') \coloneqq (C_{-*}(\Mbar_{oc}(y,y');\text{det}^{\otimes r}),\del).
\end{equation*}
Let $\dmc_r^{\stb, *}\coloneqq C_{-*}(\Mbar^{\stb}_{oc};\text{det}^{\otimes r})$ be the category of chains on the stable subcategory. Define $\dmc_r^{\opr}$ similarly using chains on the operadic subcategory $\Mbar^{\opr}_{oc}$.

\begin{remark}\label{rem:trivialising-det} The point of twisting the coefficients by this sheaf is that it cancels the difference between the boundary orientation and the fibre product orientation of a boundary stratum of the moduli space of stable maps. Equivalently, we can use the canonical spin structure and the ordering of the boundary components to trivialise $\text{det}^{\otimes r}$ over each component of $\Mbar_{oc}(y,y')$. This has the effect that the composition maps are equipped with certain signs, depending on $d$ and the distribution of marked points. They can be explicitly determined using \cite[CROrient 7C and 7H3]{CZ24} and \cite[Theorem~3.5]{HH25}.\end{remark}

\noindent
Since we do not restrict ourselves to unobstructed Lagrangians, we need to consider curvature. To this end, we modify the category $\dmc_r$, using a few preliminary definitions. We thank Nick Sheridan for fruitful discussions about this.

\begin{construction}[Lift along marked point]\label{con:lifted-simplex}
	Given $y,y'\in \obj(\Mbar_{oc})$ and an analytic simplex $\sigma \cl \Delta^m \to \Mbar(y,y')$ in the stable range, we define the chain $\pi^*\sigma\in C_{k+1}(\Mbar_{oc}(y^+,y'))$ as follows. We have $y^+ := (y_c,y_o\sqcup \{*\})$ where $*$ is ordered as the last element of $y_o$ in the orientation ordering. Then, we can form the fibre product
	\begin{center}\begin{tikzcd}
			Z \arrow[r,""] \arrow[d,"\wt\sigma"]&\Mbar_{oc}(y^+,y') \arrow[d,"\pi"]\\ \Delta^k \arrow[r,"\sigma"] & \Mbar_{oc}(y,y') \end{tikzcd} \end{center}
	where $\pi$ forgets the boundary marked point labelled by $*$. Let $\wt Z$ be a real analytic resolution of singularities as discussed in \S\ref{subsec:singular-integration}. By \cite{Gra64}, we can find a triangulation of $\wt Z$ by real analytic simplices. We define $\pi^*\sigma = \wt\sigma$ to be the chain associated to this triangulation with trivialisation of $\orl(\pi^*\sigma)$ induced by the isomorphism 
	\begin{equation}\label{eq:orientation-lift}
		\orl(\pi^*\sigma) \cong \orl(\sigma)\orl(\bR).
	\end{equation}
	If $\sigma$ maps to an unstable moduli space, we set $\pi^*\sigma = 0$.
\end{construction}

\noindent
Recall that a chain is \emph{degenerate} if it is in the image of the degeneracy maps of the simplicial set $\text{Sing}(X)$. Clearly, a simplex is degenerate if it factors through a manifold of lower dimension; or equivalently, if it is smooth, its differential is nowhere injective.

\begin{lemma}\label{lem:lift-well-defined}
	Up to degenerate simplices, the chain $\pi^*\sigma$ is a well-defined analytic simplex that does not depend on the choices made in Construction~\ref{con:lifted-simplex}.
\end{lemma}

\begin{proof}
	That $\pi^*\sigma$ is a well-defined analytic chain is immediate from the construction. To see the other assertion, we show first that, up to a degenerate chain, it does not depend on the choice of real analytic triangulation. For this it suffices to show that, for any manifold $M$ with boundary, two fundamental chains differ by a degenerate chain. Indeed, if $\sigma_1,\sigma_2$ are fundamental chains for $M$, then their difference is exact in $H_*(M,\del M)$, so there exists $\tau\in C_{\dim(M)+1}(M)$ and $\rho \in C_{\dim(M)}(\del M)$ such that $\sigma_1 -\sigma_2 = d\tau + \rho$. Then $\tau$ and $\rho$ are degenerate as claimed. If $\wt Z'$ is another resolution of singularities, we can find $\wt Z''$ with blow-down maps $f\cl \wt Z'' \to \wt Z$ and $f'\cl \wt Z''\to \wt Z'$ and using the given triangulations of $\wt Z$ and $\wt Z'$ we can find triangulations of $\wt Z''$ so that the associated chains agree with $\wt\sigma$ respectively $\wt \sigma'$. By the previous step, we thus have that $\wt \sigma-\wt\sigma'$ is degenerate, whence the claim follows.
\end{proof}

\begin{definition}\label{def:OCbar} The category $\cdmc_r$ is the symmetric monoidal dg category generated under composition and the symmetric monoidal structure by the same objects as $\dmc_r$ and the chain complexes
	$\dmc_r(y,y')$ for $y\notin \{(0,0),(0,1)\}$ and $y' \neq (0,1)$ as well as the graded vector spaces
	\begin{align*}
		\cdmc_r((0,0),(0,1)) &\coloneqq \dmc_r^*((0,0),(0,1)) \oplus \langle D_0\rangle\\
		\cdmc_r((0,1),(0,1)) &\coloneqq \dmc_r^*((0,1),(0,1)) \oplus \langle D_1\rangle,
	\end{align*}
	where $D_0$ and $D_1$ are treated as formal generators with $|D_i| = i-2$. We define the differential $\del_{oc}$ by 
	\begin{equation}\label{eq:modified-differential-old-simplex}
		\partial_{oc} \sigma \coloneqq \del\sigma + (-1)^{|\sigma|}\pi^*\sigma \g_{*} D_0 - \sum_{i \in y'_o} D_1 \g_{i} \sigma + (-1)^{|\sigma|}\sum_{i \in y_o} \sigma \g_{i} D_1,
	\end{equation}
	for $\sigma \in \dmc_r(y,y')$, while $\partial_{oc} D_0 := -D_1 \g D_0$ and
	\begin{equation}\label{} 
		\partial_{oc} D_1 \coloneqq -D_1 \g D_1 + D_2 \g_1 D_0 - D_2 \g_2 D_0,
	\end{equation}
	and extending across symmetric monoidal structure and compositions.
\end{definition}

\begin{nota}
	It will be useful to define 
	\[\pi^*D_0 \coloneqq D_1 \qquad\quad \pi^*D_1 = D_2^+ + D_2^-\]
	where $D_2^+$ is the fundamental chain of $\Mbar_{0,1;(0,1),(0,2)}$ with the ordering of the incoming marked points on the boundary agreeing with the order of their labels and $D_2^-$ the fundamental chain of the (isomorphic) moduli space where the order of the incoming marked points opposite to their ordering on the boundary.
\end{nota}

\begin{remark}
	Pictorially, we can view $D_0$ as a disc (labelled by $D_0$) with one output, and $D_1$ as a disc with one input and one output. A composition $D_1 \g \sigma$ can be viewed as (formally) attaching a $D_1$ disc bubble to an outgoing boundary marked point of $\sigma$.
\end{remark} 

\begin{definition}
	\label{de:stable-operadic-subcategories}
	Let $\cdmc_r^{\stb}$ and $\cdmc_r^{\opr}$ be the smallest symmetric monoidal subcategories containing $\dmc_r^{\stb}$, respectively $\dmc_r^{\opr}$ as well as the (formal) morphisms $D_0$ and $D_1$. Equip them with the differential restricted from $\cdmc_r$.
\end{definition}

\begin{lemma}\label{lem:lifted-simplex-analytic} Suppose $\sigma \in C^{an}_*(\Mbar_{oc}(y,y'))$ as well as $\sigma_1$ and $\sigma_2$ are stable. Then the following holds up to degenerate simplices.
	\begin{enumerate}[leftmargin=20pt]
		\setlength\itemsep{2pt}
		\item\label{pushforward-degenerate} $\pi_*\pi^*\sigma$ is degenerate.
		\item\label{lift-product} $\pi^*(\sigma\times \rho) = (-1)^{|\rho|}\pi^*\sigma \times \rho + \sigma \times \pi^*\rho$.
		\item\label{lift-composition} $\pi^*(\sigma_2\g \sigma_1) = (-1)^{|\sigma_1|}\pi^*\sigma_2 \g \sigma_1 + \sigma_2\g \pi^*\sigma_1$.
		\item\label{lift-differential} $\pi^*$ commutes with $\del_{oc}$. 
	\end{enumerate}  
\end{lemma}

\begin{proof} Assertion~\eqref{pushforward-degenerate} is immediate from the construction. The second claim follows from the fact that the forgetful map on the product of moduli spaces is the product of the forgetful map on one factor and the identity on the other; the sign is determined by the isomorphism~\eqref{eq:orientation-lift}. Assertion~\eqref{lift-composition} is a consequence of~\eqref{lift-product} and the fact that the clutching maps commute with the forgetful maps. For the last claim, abbreviate $\Mbar_+ := \Mbar_{oc}(y^+,y')$ and $\Mbar \coloneqq \Mbar_{oc}(y,y')$. We have by \cite{Joy} that 
	\[\del \Delta^k\times_{\Mbar}\Mbar_+ = \del\Delta^k \times_{\Mbar}\Mbar_+ \sqcup\Delta^k \times_{\Mbar}\del^v\Mbar_+,\]
	where $\del^v$ indicates the boundary of the fibres of the forgetful map $\pi_\kappa \cl \Mbar_+\to \Mbar$, given by the restriction of $\pi$ to $\Mbar_+$, where the marked point labelled by $*$ is between the $\kappa$ and $(\kappa+1)$-labelled marked point on the boundary.
	This vertical boundary is given by 
	\[\del^v\Mbar_+ = \im(s_{\kappa,*})^{-1}\sqcup \im(s_{*,\kappa+1}),\]
	where $s_{\kappa,*} \cl \Mbar \to \Mbar_+$ is the canonical section of $\pi_\kappa$ inserting the additional marked point at the marked point labelled by $\kappa$, creating a ghost disc in the process on which $*$ is the marked point situated \emph{after} the marked point labelled by $\kappa$ in the counter-clockwise ordering, while in $s_{*,\kappa}$ the marked point is situated before the $\kappa$-marked point on the boundary. Note that $\kappa+1$ could be an outgoing marked point. Summing over all $\kappa\in y_o$, we obtain
	\begin{equation}\label{} 
		\del\pi^*\sigma= \pi^*(\del\sigma) + (-1)^{|\sigma|} D_2^-\gt\sigma\times\ide- (-1)^{|\sigma|}D_2^+\gt(\ide\times \sigma) + (-1)^{|\sigma|}\sigma\gt (D_2^+-D_2^-).
	\end{equation}
	The sign is due to the fact that we orient $\orl(\pi^*\sigma)$ via the isomorphism $\orl(\pi^*\sigma)\cong \orl(\sigma)\orl(\ker(d\pi))$ over the locus where $d\pi$ is surjective. 
	Thus,
	\begin{align*}\label{}
		 \del_{oc}\pi^*\sigma-\pi^*\del_{oc}\sigma&= (-1)^{|\sigma|} D_2^-\gt\sigma\times\idsimp- (-1)^{|\sigma|}D_2^+\gt(\idsimp\times \sigma) + (-1)^{|\sigma|}\sigma\gt (D_2^+-D_2^-)\\&\quad +(-1)^{|\sigma|+1} (\pi^*\sigma)\gt D_1  - D_1\gt \pi^*\sigma + (-1)^{|\sigma|+1}\pi^*(\pi^*\sigma)\g_{*'} D_0  \\&\quad- \pi^*\lbr{(-1)^{|\sigma|}\pi^*\sigma \g_{*} D_0 -  D_1 \gt \sigma + (-1)^{|\sigma|}\sigma \gt D_1}\\
		 &= \pi^*(D_1\gt\sigma) - D_1\gt\pi^*\sigma +(-1)^{|\sigma|}D_2^-\gt\sigma\times\idsimp - (-1)^{|\sigma|}D_2^+\gt\idsimp\times\sigma  \\
		 &\quad +(-1)^{|\sigma|+1}(\pi^*\sigma) \gt D_1 + (-1)^{|\sigma|+1}\pi^*(\sigma\gt D_1) + (-1)^{|\sigma|}\sigma\gt (D_2^+-D_2^-)\\
		 & \quad +(-1)^{|\sigma|+1}\pi^*(\pi^*\sigma)\g_{*'} D_0 -(-1)^{|\sigma|}\pi^*(\pi^*\sigma \g_{*} D_0).
		 \end{align*}
		 The first line of the right-hand side vanishes by Claim~\eqref{lift-composition}. Adding the terms of the second line of the right hand side, we obtain $(-1)^{|\sigma|}\pi^*\sigma \g_{*} D_1$ because $\pi^*\sigma$ has the additional incoming boundary marked point labelled by $*$, which is not accounted for in $\pi^*(\sigma \gt D_1)$. Applying Claim~\eqref{lift-composition} to the last line, we obtain that the right hand side vanishes. 
\end{proof}	

\noindent
The following is immediate from a short sign verification.
\begin{cor}
	$\cdmc_r$ is a symmetric monoidal dg category.
\end{cor}

\noindent
Recall from \cite[(1a)]{Sei08} that any dg category $\scC$ defines an $A_\infty$ category on the same objects. We use cohomological grading throughout the paper. We also introduce the following notation.

\begin{nota}\label{nota:sequences}
	Given a composable sequence $\lc{\sigma} = (\sigma_d,\dots,\sigma_1)$ of morphisms in an $A_\infty$ category $\scC$ and $i \leq d-1$, we define 
	\[\lc\sigma^{> i} \coloneqq (\sigma_d,\dots,\sigma_{i+1}) \qquad\quad \qquad\lc\sigma^{\le i} \coloneqq (\sigma_i,\dots,\sigma_1)\]
	and
	\begin{equation*}\label{}
		|\lc\sigma|' \coloneqq \sum_{i =1}^{d}(|\sigma_i|-1)
	\end{equation*}
\end{nota}

\begin{definition}\label{de:dg-to-a-infinity}
	Given a dg category $\scC$, we define the associated $A_\infty$ category, denoted $\scC_\infty$ to have the same objects and with the same underlying vector spaces as morphism spaces, equipped with the operations
	\begin{itemize}
		\item $\mu^1_{\scC}\cl \scC(x,y)\to \scC(x,y) : a\mapsto (-1)^{|\alpha|}d(a)$,
		\item $\mu^2_{\scC} \cl \scC(y,z)\otimes\scC(x,y) \to \scC(x,z) : (a_2,a_1)\mapsto (-1)^{|a_1|}a_2\g a_1$, 
		\item  $\mu^d_{\scC} = 0$ for $d \ge 3$.
	\end{itemize}
	An \emph{$A_\infty$ functor} between dg categories $\scC$ and $\scD$ is an $A_\infty$ functor between their associated $A_\infty$ categories, that is, a map $\scF\cl \obj(\scC)\to \obj(\scD)$ and for each $d\ge 1$ and any objects $x_0,\dots,x_d$ of $\scC$ a map 
	\begin{align*}\label{} 
		\scF^d\cl \scC_\infty^*(x_0,\dots,x_d) \coloneqq \bigotimes\limits_{i=0}^{d-1}\scC_\infty^*(x_{d-i-1},x_{d-i})\to \Hom_{\scD_\infty}^{*+1-d}(\scF(x_0),\scF(x_d)).
	\end{align*}
	which we will denote by $\scF^d_{\lc \sigma}$ or $\scF^d(\lc \sigma)$ depending on notational convenience,
	satisfying 
	\begin{multline}\label{eq:a-infinity-functor}
		\mu^1_{\scD}(\scF^d_{\lc\sigma}) \;+\; \sum_{i=1}^d\mu^2_{\scD}\lbr{\scF^{d-i}_{\lc\sigma^{> i}},\scF^i_{\lc\sigma^{\le i}}} 
		\;=\;\sum_{i=1}^{d}(-1)^{|\lc\sigma^{< i}|'}\scF^d_{\sigma_d,\dots,\mu^1_\scC(\sigma_i),\dots,\sigma_1}\;
		+\;\sum_{i=1}^{d-1}(-1)^{|\lc\sigma^{< i}|'}{\scF^{d-1}_{\sigma_d,\dots,\mu^2_{\scC}(\sigma_{i+1}, \sigma_{i}),\dots,\sigma_1}}.
	\end{multline}
	The $A_\infty$ functor $\scF$ is said to be \emph{unital} if $\scF^1_{\ide_{y}} = \ide_{\scF(y)}$ for any $y \in \obj \scC$ and $\scF^d_{\dots,\ide, \dots} = 0$ whenever $d >1$.
\end{definition}

\begin{example}\label{ex:chain-complexes-dg-category}
	Let $R$ be a commutative ring and $\text{Ch}_R$ be the category of chain complexes of $R$-modules. This is a dg category with 
	$\text{Ch}_R(C^*,D^*)_n = \{\text{linear maps } f \cl C^*\to D^{*+n}\}$ and $\delta(f) = d_{D}\g f - (-1)^{|f|}f\g d_C$. It admits the symmetric monoidal structure given by the tensor product over $R$ of chain complexes and on morphisms by 
	\begin{align}\label{} 
		f_1\otimes f_2\cl C^*_1\otimes_R C^*_2&\to D^*_1\otimes_R D^*_2 \notag\\
		c_1\otimes c_2 &\mapsto (-1)^{|f_2||c_1|}f_1(c_1)\otimes f_2(c_2).
	\end{align}
\end{example}

\vspace*{8pt}

\noindent
Let $\scF,\scG\cl \scC\to \scD$ be two $A_\infty$ functors between $A_\infty$ categories $\scC$ and $\scD$. Following \cite[Section~1h]{Sei08}, a \emph{pre-natural transformation} between $\scF$ and $\scG$ is a sequence $(T^d)_{d\ge 0}$ of linear maps
\begin{equation}
	T^d\cl {\scC}(X_{d-1}, X_d) \otimes \dots \otimes {\scC}(X_0, X_1) \rightarrow {\scD}(\scF(X_0),\scG(X_d))[g - d].
\end{equation}
where $|T| = g$ is the degree of $T$ and we will usually omit the superscript. 
The vector space of pre-natural transformations $\scQ := \text{Fun}(\scF,\scG)$ can be equipped with an $A_\infty$ structure. For the rest of this section, we take $\scC$ to be any dg category and $\scD = \text{Ch}_\Lambda$. In this case, using the conventions of Definition~\ref{de:dg-to-a-infinity}, the differential on $\scQ$ is given by
\begin{multline}
	\label{eq:fun-differential}
	\hspace{-8pt}\mu_\scQ^1(T)(\lc \sigma) \,\coloneqq\, (-1)^{|T(\lc \sigma)|}[\del, T(\lc \sigma)]\,+\,\sum_{i = 0}^d \lbr{(-1)^{|T| + |\lc \sigma^{\leq i}|'}\scG(\lc \sigma^{>i}) \g T(\lc \sigma^{\leq i}) -(-1)^{|T||\lc \sigma^{\leq i}|'}T(\lc \sigma^{> i}) \g \scF(\lc \sigma^{\leq i})}\\ 
	- \sum_{i = 1}^d (-1)^{|T| + |\lc \sigma^{\leq i}|'} T(\partial_i\lc\sigma) - \sum_{i = 1}^{d-1} (-1)^{|T| + |\lc \sigma^{\leq i}|'} T(\concat_i\lc\sigma).
\end{multline}

\begin{definition}\label{de:natural-transformation}
	A pre-natural transformation $T$ is a \emph{natural transformation} if $\mu_\scQ^1(T) = 0$.
\end{definition}

\noindent The composition of two natural transformations $T$ from $\scF$ to $\scG$ and $S$ from $\scG$ to $\scH$ is given by 
\begin{equation}
	\label{eq:fun-composition}
	\hspace{-12pt}\mu_\scQ^2(S,T)(\lc \sigma) \,=\,\sum_{j = 0}^{d}(-1)^{|T|+|S||\lc\sigma^{\le j}|'}S(\lc\sigma^{> j})\g T(\lc\sigma^{\le j}).
\end{equation}

\subsection{Symmetric monoidal $A_\infty$ functors}\label{subsec:symmetric-monoidal}
In general, defining what it means for an $A_\infty$ category $\scC$ to be (symmetric) monoidal is difficult. The $A_\infty$ categories $\cdmc$ and $\text{Ch}_\Lambda$ that are important in this paper are dg categories, in which case the being symmetric monoidal is a well-known property. However, an open-closed DMFT is an $A_\infty$ functor $\scF\cl \cdmc \ra \text{Ch}_\Lambda$, and so it is not immediate what it means for $\scF$ to be symmetric monoidal. Based on \cite{amorim2016}, we will give a definition here in the case of $A_\infty$ functors defined on dg categories. As before, we assume $\Lambda$ is filtered and that all $A_\infty$ functors discussed below are filtered as well.

\begin{nota}\label{nota-for-symm-monoidal}
	For a sequence $\lc p = (p_0 = 0 < p_1 < \dots < p_r< d = p_{r+1})$, let 
	\begin{equation}\label{} 
		\concat_{\lc p}\lc\sigma := (\sigma_d\g \dots \g \sigma_{p_r +1},\dots,\sigma_{p_1}\g \dots \g\sigma_1)
	\end{equation}
	and \begin{equation}
		\scF^{\g \lc p}(\lc \sigma) := \scF^{d-p_r}(\sigma_d, \dots, \sigma_{p_r+1}) \g \dots \g \scF^{p_1}(\sigma_{p_1},\dots, \sigma_1),
	\end{equation}
	setting $\scF^{-1} = 0$ to have uniform notation.
\end{nota}

\noindent Let $\scC$ and $\scD$ be symmetric monoidal dg category and $\scF\cl \scC \ra \scD$ be an $A_\infty$ functor. The following definitions are inspired by \cite[Definition~3.9]{amorim2016}. 
We define the dg category $\scC\boxtimes\scC$ to have 
\begin{itemize}[leftmargin=20pt]
	\item objects: pairs $(x,y)$ with $x,y\in\text{obj}(\scC)$,
	\item morphism spaces 
	\[\scC\boxtimes\scC((x,y),(x',y')) \,:=\, \scC(x,x')\otimes\scC(y,y')\] with differential $d(\sigma \otimes \tau) = (d\sigma)\otimes\tau  + (-1)^{|\sigma|}\sigma\otimes (d\tau)$ and
	\item composition given by $(\sigma'\otimes\tau')\g (\sigma\g \tau) = (-1)^{|\tau'||\sigma|}(\sigma'\g \sigma) \otimes (\tau' \g \tau)$.
\end{itemize}
Given a composable sequence $\lc\sigma$ of morphisms in $\scC$, set 
\[|\lc \sigma + \lc p|\coloneqq(|\sigma_d|, \dots ,|\sigma_{p_r +1}|, |\sigma_{p_r}|', |\sigma_{p_r -1}|, \dots, |\sigma_{p_1}|', \dots, |\sigma_1|),\]
that is, adding $-1$ to the entries labelled by $p_r, \dots, p_1$.
Similarly let $|\lc \sigma + \lc p^c|$ be the sequence obtained by adding $-1$ to all entries not labelled by $p_r, \dots, p_1$. 
Then, define
\begin{align*}\label{} 
	\star_1 &:= \dagger(|\lc \sigma + \lc p|, |\lc \tau + \lc p^c|)\\
	\star_2 &:= \dagger(|\lc \sigma + \lc p^c|, |\lc \tau + \lc p|),
\end{align*}
Also define
\[\dagger(\lc a, \lc b) \;\coloneqq\; \sum_{i<j} a_ib_j\]
for sequences $\lc a$ and $\lc b$ of integers of the same length.


\begin{lemma}\label{}
	The following formulas define $A_\infty$ functors $\scC\boxtimes \scC\to \scD$
	\begin{itemize}[leftmargin=15pt]
		\setlength\itemsep{6pt}
		\item $\scF \g \otimes_{\scC}$ by $\scF \g \otimes_{\scC}(X, Y) = \scF(X\otimes_\scC Y)$ on objects and
		\begin{equation}
			\label{eq:F-tensor}
			\quad(\scF \g \otimes_{\scC})^d (\lc \sigma\otimes \lc \tau) := \scF^d(\sigma_d \otimes_{\scC} \tau_d, \dots, \sigma_1 \otimes_{\scC} \tau_1)
		\end{equation}
		\item  $ \text{} \otimes_{\scD} \g (\scF \otimes \scF)$ by $ \text{} \otimes_{\scD} \g (\scF \otimes \scF)(X,Y) = \scF(X)\otimes_\scD\scF(Y)$ on objects and
		\begin{equation}
			\label{eq:tensor-F}
			\quad(\otimes_{\scD} \g (\scF \otimes \scF))^d(\lc \sigma\otimes \lc \tau) \coloneqq \sum_{1 \leq p_1 < \dots < p_{r} < d} (-1)^{\dagger(|\lc \sigma + \lc p|, |\lc \tau + \lc p^c|)} \scF(\concat_{\lc p}\lc\sigma) \otimes_{\scD} \scF^{\g\lc p}(\lc\tau),
		\end{equation}
		\item $\otimes_{\scD} \g (\scF\, \ov\otimes\, \scF)$ by $\otimes_{\scD} \g (\scF \,\ov\otimes\, \scF)(X, Y) = \scF(X)\otimes_\scD \scF(Y)$ on objects and
		\begin{equation}
			\quad(\otimes_{\scD} \g (\scF\, \ov\otimes\, \scF))^d(\lc \sigma\otimes \lc \tau)\coloneqq \sum_{1 \leq p_1 < \dots < p_{r} < d} (-1)^{\dagger(|\lc \sigma + \lc p^c|, |\lc \tau + \lc p|)} \scF^{\g\lc p}(\lc\sigma) \otimes_{\scD} \scF(\concat_{\lc p}\lc\tau),
		\end{equation}
	\end{itemize}
	where we have used $\g$ to denote the composition on both $\scC$ and $\scD$. 
\end{lemma}

\begin{proof}
	The first definition is a composition of $A_\infty$ functors, thus trivially an $A_\infty$ functor. For the other two, it remains to show that $\lc\sigma\otimes \lc\tau\mapsto (-1)^{\star_1} \scF(\concat_{\lc p}\lc\sigma) \otimes \scF^{\g\lc p}(\lc\tau)$ is an $A_\infty$ functor $\scC\boxtimes \scC\to \scD\boxtimes\scD$.
	This is exactly \cite[Definition~3.9]{amorim2016}, whence the claim follows from Proposition~3.8 op. cit..
\end{proof}

\begin{definition}\label{def:symmetric-monoidal}
	An $A_\infty$ functor $\scF\cl \scC \ra \scD$ between monoidal dg categories is \emph{monoidal} if it is equipped with the data of natural transformations $R$ from $ \text{} \otimes_{\scD} \g (\scF \otimes \scF)$ to $\scF \g \otimes_{\scC}$, and $L$ from $\text{} \otimes_{\scD} \g (\scF \ov\otimes \scF)$ to $\scF \g \otimes_{\scC}$ with  $R^0 = L^0$.  If $\scC$ and $\scD$ are symmetric monoidal with braiding $S$, then $\scF$ is \emph{symmetric monoidal} if it is monoidal,
	\[\scF(S_{x_d,y_d})\g R(\lc \sigma\otimes \lc \tau) = L(\lc \tau\otimes\lc \sigma) \g S_{\scF(x_0),\scF(y_0)}\] 
	and the diagram
	\begin{equation}
		\begin{tikzcd}
			\scF(x) \otimes \scF(y) \arrow[r, "R^0_{x,y}"] \arrow[d, "S_{\scF(x), \scF(y)}"'] & \scF(x \otimes y) \arrow[d, "\scF(S_{x,y})"] \\
			\scF(y) \otimes \scF(x) \arrow[r, "R^0_{y,x}"]   & \scF(y \otimes x)
		\end{tikzcd}
	\end{equation}
	commutes.
	If $\scF$ is a unital $A_\infty$ functor, $\scF$ is \emph{unital symmetric monoidal} if moreover for all $d \geq 1$ \begin{equation}
		\label{eq:unital-symm-monoidal}
		R^d(\lc \sigma\otimes \lc \tau) = L^d(\lc \tau\otimes\lc \sigma) = 0 
	\end{equation}
	whenever for each $i$ either $\sigma_i$ or $\tau_i$ is the identity morphism.
\end{definition}

The following definition is due to Costello, \cite{Cos07}, in the case of dg functors.

\begin{definition}\label{de:h-split}
	A monoidal $A_\infty$ functor $\scF$ is \emph{h-split} if $R^0$ and $L^0$ are quasi-isomorphisms.
\end{definition}

\begin{remark}
	As noted below \cite[Definition~3.9]{amorim2016}, the definition of $\otimes_{\scD} \g (\scF \otimes \scF)$ is not functorial, but is expected to be so up to homotopy.
\end{remark}

\noindent We have the following direct consequence of unitality.

\begin{lemma}
	If $\scF$ is a unital symmetric monoidal $A_\infty$ functor, then
	\[
	\scF(\lc \sigma\otimes_\scC (\idsimp^y, \dots, \idsimp^y)) \g R^0_{x_0,y} \;=\; R^0_{x_d,y}\g \left(\scF(\lc \sigma) \otimes \ide_{\scF(y)}\right),
	\]
	and	
	\[
	\scF((\idsimp^x, \dots, \idsimp^x)\otimes_\scC \lc \tau) \g L^0_{x,y_0} \;=\; L^0_{x,y_d}\g \left(\ide_{\scF(x)}\otimes \scF(\lc \tau)\right).
	\]
\end{lemma}

\begin{proof}
	By the unitality of $\scF$ and the natural transformations $R$ and $L$, all other terms in Equation~\eqref{eq:fun-differential} vanish. Thus, the claim follows from a sign computation.
\end{proof}

\subsection{Curved open-closed Deligne--Mumford field theories}\label{subsec:def-of-dmft}
Fix a valuation ring $(\Lambda,\nu)$ equipped with an $\bR$-algebra structure. Later, we will take $\Lambda$ to be a Novikov ring defined in \eqref{eq:novikov-ring}. Let $\text{Ch}_\Lambda$ be the dg category of cochain complexes over $\Lambda$ equipped with the symmetric monoidal product $\otimes_\Lambda$.

\begin{definition}\label{de:curved-dmft}
	A \emph{curved open-closed $r$-dimensional Deligne--Mumford field theory} over a filtered ring $\Lambda$ is an h-split unital symmetric monoidal $A_\infty$ functor $\scF \cl \cdmc_r \rightarrow \text{Ch}_{\Lambda}$ such that
	\begin{itemize}[leftmargin =20pt]
		\item (nondegeneracy) $\scF_{\sigma_d, \dots, \sigma_1} = 0$ whenever some $\sigma_i$ is a degenerate chain,
		\item (positivity) $\scF_{D_0}$ and $\scF_{D_1}$ have positive filtration.
	\end{itemize}
\end{definition}

\begin{remark}\label{} One could use any other model of chains on $\Mbar_{oc}$ for this definition. In fact, we will construct a specific model in \S\ref{subsec:our-chains} that will allow us to achieve better properties for the (curved) open-closed DMFT we define.
\end{remark}

\begin{remark}\label{rem:curved-dmft} Explicitly, a curved $r$-dimensional open-closed Deligne-Mumford field theory consists of a filtered graded chain complex $\scF^*(y) = (V_y, d_y)$ for each $y \in \text{obj}(\cdmc_r)$, and a sequence $(\scF^d)_{d\ge 1}$ of linear\label{key} maps
	\begin{equation*}\label{}
		\scF^d\cl \cdmc^*_r(y^{d-1},y^d) \otimes \dots\otimes \cdmc^*_r(y^0,y^1) \to \Hom^{*+1-d}_{\text{Ch}_\Lambda}(\scF(y^0),\scF(y^d))
	\end{equation*}
	such that $\scF_{D_0}$ and $\scF_{D_1}$ have positive filtration, and 
	\begin{multline}\label{eq:warped-relations} 
		(-1)^{|\scF^d_{\lc \sigma}|}(d_{y_d}\g\scF^d_{\lc \sigma}-(-1)^{|\scF^d_{\lc\sigma}|}\scF^d_{\lc\sigma}\g d_{y_0}) + 
		\sum_{i=1 }^{d-1}{(-1)^{|\lc\sigma^{\leq i}|'+1}\scF^{d-i}_{\lc\sigma^{> i}}\g\scF^i_{\lc\sigma^{\le i}}}\\ 
		= \sum_{i=1}^{d}(-1)^{|\lc\sigma^{\leq i}|' + 1}\scF^d_{(\sigma_1,\dots,\partial_{oc}\sigma_{i},\dots,\sigma_d)} +\sum_{i=1}^{d-1}{(-1)^{|\lc\sigma^{\leq i}|' + 1}\scF^{d-1}_{(\sigma_d,\dots,\sigma_{i+1}\g \sigma_{i},\dots,\sigma_1)}},
	\end{multline}
	such that $\scF$ is unital symmetric monoidal.
\end{remark}

\begin{definition}
	\label{de:uncurved-dmft}
	An \emph{(uncurved) open-closed Deligne--Mumford field theory} is an h-split unital symmetric monoidal $A_\infty$ functor $\wt \scF \cl \dmc_r \rightarrow \text{Ch}_{\Lambda}$ satisfying analogous conditions.
\end{definition}

\begin{definition}\label{}
	The \emph{open DMFT} obtained from $\scF$ is the restriction to the full subcategory on objects $y = (\emst,y_o)$, while the \emph{(closed) DMFT} underlying $\scF$ is the restriction to the subcategory with objects $y = (y_c,\emst)$ and morphism given by chains on moduli space of closed curves.
\end{definition}

\begin{lemma}\label{lem:cohft-from-dmft} If $\scF$ is a curved open-closed DMFT with closed sector $V^*$, then $\scF$ defines the structure of a cohomological field theory on $H^*(V^*)$, i.e. a symmetric monoidal functor $H_*(\Mbar_c) \rightarrow Vect_\Lambda$.
\end{lemma}

\begin{proof}
	The $A_\infty $ equation for $\scF^1$ and $\sigma \in C_*(\Mbar_c)$ yields that $d\scF^1_\sigma(\alpha) =(-1)^{|\sigma|}\scF^1_\sigma(d\alpha)+ \scF_{d\sigma}^1(\alpha)$, whence $\scF_\sigma^1$ is a chain map for such $\sigma$ if $\partial\sigma = 0$. Meanwhile, the $A_\infty $ relation for $\scF^2$, applied to two cycles $\sigma_1,\sigma_2$ yields that the resulting map on cohomology is a functor, using that $|\scF_{\sigma_i}^1| = |\sigma_i|$.
\end{proof}

\begin{definition}\label{de:curved-dmft-properties} 
	An open-closed DMFT $\scF\cl \cdmc_r\to \text{Ch}_\Lambda$ is
	\begin{itemize}[leftmargin=17pt]
		\setlength\itemsep{1.5pt}
		\item \emph{strict} if $\scF^d = 0$ for $d \ge 2$.
		\item \emph{semi-strict} if $\scF^d$ vanishes on $({\cdmc_r}^{\stb,*})^{\otimes d}$ whenever $d \geq 2$.
		\item \emph{operadic} if $\scF^d$ vanishes on $({\cdmc_r}^{\opr,*})^{\otimes d}$ whenever $d \geq 2$.
		\item \emph{weakly curved} if it is operadic and $\scF^1_{D_0} = c\,e_\scF$ for some $c \in \Lambda$, where $e_\scF$ denotes the (open) unit.
	\end{itemize}
\end{definition}

\begin{definition}\label{de:category-without-curvature}
	Let $\cdmc^{unc}_r$ be the symmetric monoidal dg category obtained from $\cdmc_r$ by setting $D_0$ equal to $0$.
\end{definition}

\begin{definition}
	If $\scF^d_{\lc \sigma} = 0$ whenever $D_0$ appears in $\lc \sigma$ (either as an element of the tuple or composed with another chain), we say that $\scF$ is a curved open-closed DMFT \emph{with vanishing curvature}. Equivalently, one can view this as a functor $\scF\cl \cdmc^{unc}_r \rightarrow \text{Ch}_{\Lambda}$.
\end{definition}

\noindent A semi-strict curved open-closed DMFT gives rise to a curved $A_\infty$ algebra.

\begin{lemma}\label{lem:algebra-from-dmft}
	\label{lem:curved-algebra-from-dmft} Suppose $\scF$ is a semi-strict open-closed DMFT with $(\scA,d)\coloneqq \scF(0,1)$. Then $\scF$ defines the structure of a curved $A_\infty$ algebra on $\scA$. If $\scF$ is operadic, then $\scA$ is a curved cyclic unital $A_\infty$ algebra.
\end{lemma}

\begin{proof}
	 To make the sign computations more manageable, we introduce the following notation. For $\alpha \in \scF(x)$, $\beta \in \scF(y)$, let 
	 \[
	\alpha \times \beta := (-1)^{|\alpha||\beta|}R^0(\alpha \otimes \beta).
	\]
	where $R = (R^d)_{d\ge 0}$ is one natural transformation of the right symmetric monoidal structure on $\scF$. 
	We then find that 
	\begin{equation}
		\label{eq:symm-monoid-times-sign}
		\scF^1_{\sigma \times \idsimp}(\alpha \times \beta) = (-1)^{{|\sigma||\beta|}}\scF^1_{\sigma}(\alpha) \times \beta.
	\end{equation}
	Let $\mu_\scA^1(\alpha) \coloneqq  (-1)^{|\alpha|}(d_{0,1}(\alpha) + \scF_{D_1}(\alpha))$ be the $A_\infty$ differential of $\scA$. Additionally define 
	\begin{equation*}
		\mu_{\scA}^0 := - \scF_{D_0}.
	\end{equation*}
	To define the higher operations, let $D_k \in \cdmc_r^{k-2}((\emst,\{1,\dots,k\}),(\emst,1))$ be an inductive choice of analytic fundamental chain so that its boundary is given by
	\begin{equation}\label{eq:compatible-chains}
		\partial D_k \;= \sum_{\substack{k_0 + k_1 = k + 1\\ k_0,k_1 \geq 2}}\sum_{1\le i \le k_0}^{} (-1)^{k_1(k_0-i) + i} D_{k_0} \g (\idsimp^{i-1} \times D_{k_1} \times \idsimp^{ k_0 - i}),
	\end{equation}
	where the map $\psi$ clutches the output of $D_{k_1}$ to the $i$'th marked point of $D_{k_0}$. 
	Then, for $k \geq 2$, let 
	\begin{equation*}
		\mu^k(\alpha_1, \dots, \alpha_k) = (-1)^{\epsilon(\alpha)}\scF^1_{D_k}(\alpha_1 \times \dots \times \alpha_k).
	\end{equation*} 
	where 
	\begin{equation}\label{eq:st-sign}
		\epsilon(\alpha) = 1+ \sum_{j = 1}^{k} j(|\alpha_j| + 1)
	\end{equation}
	Note that this agrees with the signs in \cite{ST16}.
	As before, we use the notation $D_1 \gt D_k$ instead of $\sum_{i \in y'_o} D_1 \g_i D_k$, and similarly for the reverse composition.
	Note that $|\scF_{D_k}| = |D_k|\equiv k\mod{2}$. Thus, Equation~\eqref{eq:warped-relations} specialises to 
	\begin{equation}
		\label{eq:A-infinity-chain-map}
		d \g \scF_{D_k} - (-1)^{k}\scF_{D_k} \g d \;=\;  \scF_{\partial D_k} - \scF_{D_1 \gt D_k} + (-1)^k\scF_{D_k \gt D_1} + (-1)^{k}\scF_{\pi^*D_k \g_* D_0}
	\end{equation}
	To see that this yields the $A_\infty$ equations, we write out the terms one by one.
	Given $1\leq i \leq k_0$ and $\alpha_1,\dots,\alpha_k\in \scA$ we use the short hand $$\alpha^1 = (\alpha_1, \dots, \alpha_{i-1}),\qquad \alpha^2 = (\alpha_i, \dots, \alpha_{i+k_1 - 1}),\qquad \alpha^3 = (\alpha_{i+k_1}, \dots, \alpha_k)$$\noindent
	whenever $i$ is clear from context. 
	Moreover, we write $\alpha$ for $\alpha_1\times \dots\times \alpha_k$ as well. 
	First observe that Equation \eqref{eq:symm-monoid-times-sign} yields
	\begin{equation*}
		\scF_{\ide^{ i-1} \times D_{k_1} \times \ide^{ k_0 - i}}(\alpha) = (-1)^{k_1|\alpha^3|}\alpha^1 \times \scF_{D_{k_1}}(\alpha^2) \times \alpha^3,
	\end{equation*}
	whence
	\begin{equation*}
		\scF_{\partial D_k} = (-1)^{|\alpha| + k + \epsilon(\alpha) + 1}\sum_{\substack{k_0 + k_1 = k\\i = 1, \dots, k_0\\ k_0,k_1 \geq 2}} (-1)^{|\alpha^1|'} \mu^{k_0}_\scA(\alpha^1, \mu^{k_1}_\scA(\alpha^2), \alpha^3).
	\end{equation*}
	The sign here comes from the fact that 
	\[
	k_1(k_0 - i)+ k_1|\alpha^3| + \epsilon(\alpha^2) + \epsilon(\alpha^1,\mu^{k_1}_\scA(\alpha^2), \alpha^3) \equiv |\alpha|+ k + \epsilon(\alpha) + |\alpha^1| + i, 
	\]
	where we have used \cite[Lemma~2.9]{ST16} to compute $\epsilon(\alpha^2) + \epsilon(\alpha^1,\mu^{k_1}_\scA(\alpha^2), \alpha^3)$.
	Since $\scF$ is semi-strict,  
	\begin{equation*}
		\scF_{D_1 \gt D_k} = \scF_{D_1} \gt \scF_{D_k} \text{ and } \scF_{D_k \gt D_1} = \scF_{D_k} \gt \scF_{D_1}.
	\end{equation*}
	We thus find that 
	\[
	(d \g \scF_{D_k} + \scF_{D_1 \gt D_k})(\alpha) = (-1)^{|\alpha| + k}\mu^1_\scA(\scF_{D_k}(\alpha)) = (-1)^{|\alpha| + k + \epsilon(\alpha)}\mu^1_{\scA}(\mu^k_\scA(\alpha)).
	\]
	Similarly,
	\[
	(\scF_{D_k \gt D_1} + \scF_{D_k} \circ d)(\alpha) = (-1)^{|\alpha| + 1 + \epsilon(\alpha)}\sum_{i = 1}^{k} (-1)^{|\alpha^1| + i - 1} \mu^k_\scA(\alpha^1, \mu^1_\scA(\alpha_i), \alpha^3),
	\]
	For the curvature terms, first observe that because $\pi^*D_k$ has the extra marked point ordered as the last incoming marked point, we have \[
	\scF_{\pi^* D_k \g D_0}(\alpha) = -\scF_{\pi^*D_k}(\alpha \times \mu^0_\scA) = \sum_{i=0}^{k}(-1)^{k-i+1}\scF_{D_{k+1}}(\alpha_1 \times \dots \times \alpha_i \times \mu^0_\scA \times  \alpha_{i+1} \times \dots \times \alpha_k),
	\]
	where the sign $(-1)^{k-i}$ comes from reordering the marked points.
	Then,
	\begin{align*}
		\sum_{i}&(-1)^{|\alpha_1|+ \dots+ |\alpha_i| - i}\mu^{k+1}_\scA(\alpha_1, \dots, \alpha_i, \mu^0_\scA, \alpha_{i+1}, \dots, \alpha_k) \\
		&= \sum_i (-1)^{\epsilon(\alpha) + |\alpha| + k-i}\scF_{D_{k+1}}(\alpha_1 \times \dots \times \alpha_i \times \mu^0_\scA \times \alpha_{i+1}\times  \dots\times \alpha_k) \\
		&= (-1)^{|\alpha| + \epsilon(\alpha) + 1}\scF_{\pi^* D_k \g_* D_0}(\alpha),
	\end{align*}
	where the second equality follows from the fact that $\epsilon(\alpha) + \epsilon(\alpha^1,\mu^0_\scA, \alpha_2) = |\alpha^2|' + i = |\alpha^2| + k$.
	
	When $\scF$ is operadic, we have that $\scF_{D_k}(e_\scF, \dots) = \scF_{(\pi_1)_*D_k}(\dots) = 0$, where $\pi_1$ denotes the map forgetting the first incoming boundary marked point. Being operadic also implies that $\scF_{D_2}(e_\scF,\alpha) = \scF_{D_2}(\alpha,e_\scF) = \alpha$, whence unitality of $\scA$ follows.
	
	Define the pairing $\langle\cdot,\cdot \rangle_{\scA}: \scA^{\otimes 2} \rightarrow \Lambda$ by $\langle \alpha_1, \alpha_2 \rangle_\scA = (-1)^{|\alpha_2|}\scF_{D_{2,0}}(\alpha_1 \times \alpha_2)$, where $D_{2,0}$ denotes the surface giving rise to the pairing, that is, the disc with $2$ incoming marked points, and no outgoing marked points. Let $P_{k+1}$ denote the chain obtained by the composition $D_{2,0} \g (D_k \times \idsimp)$. This represents the moduli space of discs with $k+1$ incoming marked points. It admits a $\ZZ/(k+1)$ action by cyclically relabelling the marked points. The action of the generator $\gamma$ induces an isomorphism with sign $(-1)^{k}$ (obtained by moving the $k+1$st marked point all the way past the others). 
	Thus,
	\[
	\scF_{P_{k+1}}(\alpha \times \alpha_0) = (-1)^{k+|\alpha_0||\alpha|}\scF_{P_{k+1}}(\alpha_0 \times \alpha),
	\]
	for $\alpha = \alpha_1\times \dots\times \alpha_k$. If $\scF$ is operadic, Equation \eqref{eq:symm-monoid-times-sign} implies
	\[
	\scF_{P_{k+1}}(\alpha \times \alpha_0) = \scF_{D_{2,0}} \g \scF_{(D_k \times id)}(\alpha \times \alpha_0) =  (-1)^{(k-1)|\alpha_0| + \epsilon(\alpha)}\langle \mu^k_\scA(\alpha),\alpha_0\rangle_\scA,
	\]
	whereas
	\[
	\scF_{P_{k+1}}(\alpha_0 \times \alpha) = (-1)^{(k-1)|\alpha_k| + \epsilon(\alpha_0, \dots, \alpha_{k-1})}\langle \mu^k_\scA(\alpha_0, \dots, \alpha_{k-1}),\alpha_k\rangle_\scA.
	\]
	A short sign computation then shows the result.
\end{proof}

\begin{remark}
	Similar to our construction of the modified chain complex $C_*(\Mbar_{oc})$ in \ref{subsec:our-chains}, one can use the higher homotopies to associate an $A_\infty$ algebra to an open-closed DMFT which is not necessarily semi-strict.
\end{remark}

\noindent
One can view a weakly curved $A_\infty$ algebra as an uncurved $A_\infty$ algebra. The same is true for open-closed DMFTs, provided it has the following properties.

\begin{property}\label{assum:weakly-curved-to-uncurved}
	Let $\scF$ be an operadic open-closed DMFT with open unit $e_\scF$ and let $\lc \sigma \in \cdmc(x, \dots, x')$ be any composable sequence. Then, we have the properties
	\begin{enumerate}
		\item \label{assum:weakly-curved-unit}$\scF^d_{(\sigma_d, \dots, \pi^*\sigma_i, \dots, \sigma_1 \times \idsimp)}(R^0(\alpha \otimes e_\scF)) = 0$.
		\item \label{assum:weakly-curved-Dk} for any $k \geq 0$, 
		\[\scF_{(\sigma_d, \dots, \pi^*\sigma_i\g_* D_k, \sigma_{i-1},\dots, \sigma_{1})}\g R^0\;=\;(-1)^{k|\lc\sigma^{< i}|'}\scF_{(\sigma_d, \dots, \pi^*\sigma_i, \sigma_{i-1},\dots, \sigma_{1})}\g (\ide\otimes\scF_{D_k}) \g R^0\]
		on $\scF^*(y^0)\otimes \scF(0,k)$.
	\end{enumerate}
\end{property}

\noindent
Property \eqref{assum:weakly-curved-unit} can be seen as a stronger version of operadic, as it says that `the unit is compatible with any chain of the form $\pi^*\sigma$'.
For our open-closed DMFT, these properties are shown in Lemma \ref{lem:additional-properties}.

\begin{lemma}\label{lem:weakly-curved-dmft} If $\scF$ is a weakly curved $r$-dimensional open-closed DMFT, and Property~\ref{assum:weakly-curved-to-uncurved} holds, then $\scF$ restricts to a functor $\scF\cl \cdmc^{unc}_r \rightarrow \text{Ch}_{\Lambda}$.
\end{lemma}

\begin{proof}
	The only way that $D_0$ shows up in the structure equations is through $\scF_{(\sigma_d, \dots, \pi^*\sigma_i \g_* D_0, \dots, \sigma_1)}$. We thus need to show that this vanishes whenever $\scF$ is weakly curved. 
	Applying both Properties~\ref{assum:weakly-curved-to-uncurved}, we obtain
	\begin{align*}
		\scF_{(\sigma_d, \dots, \pi^*\sigma_i \g_* D_0, \dots, \sigma_1)}(\alpha) &\; =\; \pm\, \scF_{(\sigma_d, \dots, \pi^*\sigma_i, \dots, \sigma_1 \times \idsimp)}(R^0( \alpha \otimes \scF_{D_0})) \\
		&\; =\; \pm\, c\,\scF_{(\sigma_d, \dots, \pi^*\sigma_i, \dots, \sigma_1\times \idsimp)}(R^0( \alpha \otimes e_\scF)) \\
		&\;=\; 0
	\end{align*}
	as required.
\end{proof}
\subsection{Deformations and bounding cochains} \label{subsec:bounding-cochains-and-deformations}
Recall from \cite[Definition~3.6.29]{FOOO2} that a {weak bounding cochain} for a filtered, unital curved $A_\infty$ algebra $\scA$ is an element $\bc \in \scA$ of degree $1$ and positive valuation such that \begin{equation}
	\sum_{k \geq 0} \mu^k_\scA(\bc^{\otimes k}) = c \cdot 1,
\end{equation}
for some $c \in \Lambda_{> 0}$. When $c = 0$, the element  $\bc$ is called a {(strong) bounding cochain}. \cite{FOOO2} show that given any element $\bc$ of degree $1$ with positive valuation, one obtains the deformed $A_\infty$ algebra $(\scA,\mu_\bc)$ with operations \begin{equation}
	\label{eq:deformed-Ainfty-operations}
\mu^k_\bc(\alpha_1 \otimes \dots \otimes \alpha_k) := \sum_{i_0, \dots, i_k \geq 0} \mu^{i_0 + \dots + i_k + k}(\bc^{\otimes i_0} \otimes \alpha_1 \otimes \dots \otimes \alpha_k \otimes \bc^{\otimes i_k}). 
\end{equation}
If $\bc$ is a weak bounding cochain, then the operations $\mu_\bc^k$ define an (uncurved) $A_\infty$ structure.

The main result of this section, Corollary~\ref{cor:actual-dmft}, asserts that we can similarly deform a curved open-closed DMFT $\scF$ by an element $\bc$ of degree $1$ with positive valuation. If $\bc$ is a weak bounding cochain and $\scF$ satisfies certain properties, this defines an (uncurved) open-closed DMFT.

\begin{definition}\label{de:bounding-cochain}
	A \emph{(weak) bounding cochain} for a semi-strict (or operadic) curved open-closed DMFT $\scF$ is a (weak) bounding cochain for the associated $A_\infty$ algebra of Lemma \ref{lem:curved-algebra-from-dmft}.
\end{definition}

\noindent We will need the following preliminary definitions.

\begin{definition}\label{de:lift-sequence}
	For a tuple $\lc{\sigma} = (\sigma_d, \dots, \sigma_1)$, define the pullback 
	\begin{equation}
		\label{eq:lift-of-tuple}
		\pi^*\lc{\sigma} = \sum_{i = 1}^d (-1)^{|\lc\sigma^{<i}|'}(\sigma_d, \dots,\sigma_{i+1}, \pi^*\sigma_i, \sigma_{i-1} \times \idsimp, \dots, \sigma_1\times\idsimp).
	\end{equation}
\end{definition}

\begin{remark}\label{}
	It follows that the iterated lift of the sequence $\lc\sigma$ satisfies
	\begin{equation}
		\label{eq:iterated-lift}
		\pi^{[k]}\lc \sigma = \sum_{k_d + \dots k_1 = k}(-1)^{\dagger(\lc k, \lc \sigma)}(\pi^{[k_d]}\sigma_d, \dots, \pi^{[k_1]}\sigma_1 \times \idsimp^{k-k_1}),
	\end{equation}
	for $k \ge 1$, where $\dagger(\lc k, \lc \sigma) = \sum_{i = 1}^d k_i(|\lc\sigma^{<i}|' + |\lc k^{<i}|)$. 
\end{remark}

\begin{lemma}\label{lem:lifted-sequence-properties} Suppose $\lc\sigma$ is a composable sequence of chains and $\lc k = (k_1,\dots,k_d)$. Set 
	\[\pi^{[\lc k]}\lc\sigma \coloneqq (-1)^{\dagger(\lc k, \lc \sigma)}(\pi^{[k_d]}\sigma_d, \dots, \pi^{[k_1]}\sigma_1 \times \idsimp^{k-k_1}).\] 
	Then, we have for any $i \le d$ that
	\begin{itemize}
		\item $\pi^{[\lc k]}((\partial_{oc})_i\lc\sigma) = (-1)^{|\lc k^{> i}|}(\partial_{oc})_i(\pi^{[\lc k]}\lc\sigma)$, where $(\partial_{oc})_i\lc\sigma = (\sigma_d,\dots,\partial_{oc}\sigma_i,\dots,\sigma_1)$.
		\item $\concat_i\pi^{[\lc k]}\lc\sigma = (-1)^{|\lc k^{> i}|}\pi^{[\lc k]}(\concat_i\lc\sigma)$;
		\item $((\pi^{\lc k}\lc\sigma)^{ > i},(\pi^{\lc k}\lc\sigma)^{ \le i}) = (-1)^{(|\lc\sigma^{\le i}|'+|\lc k^{\le i}|)|\lc k^{> i}|}(\pi^{\lc k^{> i}}(\lc\sigma^{> i}),\pi^{\lc k^{\le i}}(\lc\sigma^{\le i}))$.
	\end{itemize} 
\end{lemma}

\begin{proof}
	The first two claims follow from Lemma~\ref{lem:lifted-simplex-analytic} and the fact 
	\[\dagger(\lc k, \partial_i\lc \sigma) -\dagger(\lc k, \lc \sigma) = |\lc k^{>i}|= \dagger(\lc k, \lc \sigma) -\dagger(\lc k, \concat_i\lc \sigma).
	\]
	The third claim follows from a direct computation of the sign difference.
\end{proof}	

More generally, for $\lc\sigma\otimes\lc\tau = (\sigma_d\otimes\tau_d,\dots,\sigma_1\otimes\tau_1)$ and $k = k_0 + k_1$, we define
\begin{equation*}\label{}
	\pi^{[k_0,k_1]}(\lc\sigma\otimes\lc\tau) = \s{k^i_1+\dots +k^i_d = k_i}{(-1)^{\dagger(\lc k^0,\lc k^1,\lc\sigma,\lc\tau) } (\pi^{[k^0_d]}\sigma_d\otimes \pi^{[k^1_d]}\tau_d,\dots,\pi^{[k^0_1]}\sigma_1\otimes \pi^{[k^1_1]}\tau_1)},
\end{equation*}
with 
\[\dagger(\lc k^0,\lc k^1,\lc\sigma,\lc\tau) := \sum_{i = 1}^{d}k^1_i(|\lc\sigma\otimes\lc\tau^{< i}|'+ |\lc k^{< i}|)+k^0_i(|\lc\sigma\otimes\lc\tau^{< i}|'+ |\lc k^{< i}|+|\sigma_i|+ |k^0_i|).\]

\begin{definition}\label{de:deformed-dmft} Given an element $\bc \in \scA$, with $|\bc| = 1$ and positive valuation, the associated \emph{deformed curved open-closed DMFT} $\scfb$ is defined by $ \scfb = \scF$ on objects and by
	\begin{equation}\label{eq:deformed-dmft-morphisms}
		\scfb_{\lc\sigma}(\alpha)\, \coloneqq\,\sum_{k \geq 0} \frac{(-1)^{k|\alpha|}}{k!}\scF_{\pi^{[k]} \lc \sigma}(R^0(\alpha\otimes \bc^{\otimes k})),
	\end{equation}
	on composable sequences $\lc\sigma$ of simplices in $\cdmc_r(y,y')$ such that $\sigma_i \neq \sigma'_i \g_\kappa D_0$ for some $\sigma'_i$ for any $i$. For sequences of simplices containing such compositions, we define $\scfb_{\lc\sigma}$ inductively by
	\begin{equation}\label{eq:deformed-dmft-morphisms-2}
		\scfb_{\lc\sigma}(\alpha)\, \coloneqq\,\sum_{k \geq 0} \frac{(-1)^{k|\alpha|}}{k!}\scF_{\pi^{[ k]} \lc \sigma}(R^0(\alpha\otimes \bc^{\otimes k}))
		+ \sum_{k \geq 1} \frac{(-1)^{k|\alpha|}}{k!} \scF_{\pi^{[k-1]}\lc \sigma'}(R^0(\alpha\otimes d(\bc)\otimes \bc^{\otimes k-1})),
	\end{equation}
	where $\sigma'_j = \sigma_j$ for $i \neq j$ and $\sigma'_j = \sigma_i'$ for $j = i$.
	Finally, extend linearly to composable sequences of chains. 
	The associated symmetric monoidal transformations $R$ and $L$ are defined by $(R^{\bc})^0 = R^0$ and
	\begin{equation}\label{} 
		R^{\bc}_{\lc\sigma\otimes\lc\tau}(\alpha\otimes \beta) \;=\; \s{k\ge 0}{\sum_{k_0+k_1 = k}\frac{(-1)^{k_0|\alpha|+k_1|\beta|}}{k!}R_{\pi^{[k_0,k_1]}(\lc\sigma\otimes\lc\tau)}\lbr{(\alpha\otimes \bc^{\otimes k_0})\otimes (\beta\otimes \bc^{\otimes k_1})}}
	\end{equation}
	for sequences in $\dmc_r$
	 and a sum as in~\eqref{eq:deformed-dmft-morphisms-2} involving compositions with $D_0$ or $D_1$. We define $L^{\bc}$ analogously.
\end{definition}

\noindent In particular, 
	\begin{equation}\label{eq:deformed-differential}
			\scfb_{D_1}(\alpha) = \sum_{k \geq 0} \frac{(-1)^{k|\alpha|} }{k!}\scF_{\pi^{[k]}D_1}(R^0(\alpha\otimes \bc^{\otimes k}))
		\end{equation}
	and
	\begin{equation}
			\scfb_{D_0} = d(\bc) + \sum_{k \geq 0} \frac{1}{k!}\scF_{\pi^{[k]}D_0}(R^0(\bc^{\otimes k})).
		\end{equation}
		
\begin{remark}
	Note that the operations $\scfb_\sigma$ are only deformed for stable chains and $D_0$ and $D_1$ because the lift constructed in~\eqref{eq:iterated-lift} vanishes on unstable chains.
\end{remark}

\begin{remark}\label{rem:compare-differentials}
	Recalling the definition $\alpha_1 \times \alpha_2 := (-1)^{|\alpha_1||\alpha_2|}R^0(\alpha_1 \otimes \alpha_2)$, which agrees in our geometric setting with cross-product of differential forms, the deformed DMFT is given by \[
		\scfb_{\lc\sigma}(\alpha)\, \coloneqq\,\sum_{k \geq 0} \frac{1}{k!}\scF_{\pi^{[k]} \lc \sigma}(\alpha\times \bc^{\times k}).
	\]
	In particular, if $\scF$ is semi-strict or operadic, the deformation of the underlying $A_\infty$ algebra (as defined in \cite[Definition~13.39]{FOOO}) associated to $\scF$ agrees with the $A_\infty$ algebra associated to $\scfb$. This follows from a sign computation and the observation that we require a factor of $\frac{1}{k!}$ because we sum over $k!$ many fundamental chains since each insertion of $\bc$ can be along any of the boundary segments, as opposed to the definition of the deformation in the $A_\infty$ algebra setting, where each insertion of $\bc$ is fixed on a specific boundary segment.
\end{remark}

\noindent The following result shows that (weak) bounding cochains suffice to turn a curved open-closed DMFT into an open-closed DMFT with vanishing curvature. In the next subsection we will show how that gives rise to an open-closed DMFT. 

\begin{proposition}\label{prop:deformed-dmft}
	If $\scF$ has Property~\ref{assum:weakly-curved-to-uncurved}\eqref{assum:weakly-curved-Dk}, then $\scfb$ is a curved open-closed DMFT. If $\bc$ is a weak bounding cochain, $\scfb$ is weakly curved.
\end{proposition}

\begin{proof}
	This is a lengthy but, given Lemma~\ref{lem:lifted-sequence-properties} straightforward, computation, which reduces the $A_\infty$ functor equations for $\scfb_{\lc\sigma}$ to (the sum of) the $A_\infty$ functor equations for $\scF_{\pi^{[k]}\lc\sigma}$. 
\end{proof}

\noindent 
From the definition, one verifies directly that if $\scF$ has Property~\ref{assum:weakly-curved-to-uncurved}, then so does $\scfb$. Then, Lemma \ref{lem:weakly-curved-dmft} implies

\begin{cor}
	 If $\bc$ is a weak bounding cochain, and $\scF$ has Property~\ref{assum:weakly-curved-to-uncurved}, then the deformed open-closed DMFT $\scfb$ induces an curved open-closed DMFT with vanishing curvature.
\end{cor}

\subsection{Deforming the differential}\label{subsec:compatible-with-d1}
In this section we show that a curved open-closed DMFT $\scF\cl  \cdmc^{unc}_r \rightarrow \text{Ch}_\Lambda$ with vanishing curvature induces an open-closed DMFT.

To this end, recall that a \emph{Maurer--Cartan element} $b$ for an $A_\infty$ category $\scC$ is the data of a Maurer--Cartan element $b_y$ for each $A_\infty$ algebra $\scC(y,y)$. Maurer--Cartan elements are the same as bounding cochains, but we will use different terminology here to avoid confusing the Maurer--Cartan elements employed here with the bounding cochains for a Lagrangian employed in the previous section.

Given a Maurer--Cartan element $b$, the deformed category $\scC^b$ has the same objects and morphisms as $\scC$, but a deformed $A_\infty$ structure defined by the analogous formula to Equation \eqref{eq:deformed-Ainfty-operations}, where one replaces $b$ by $b_y$ for appropriate $y$. We will only require the definition in the case where $\scC = \scC_\infty$ is the $A_\infty$ category associated to a dg category.

\begin{definition}
	\label{de:deformation-of-dg-cat}
	Let $\scC$ be a dg category, and $b$ a Maurer--Cartan element for $\scC_\infty$. Then $\scC^b$ is the $dg$ category associated to the $A_\infty$ category $\scC_{\infty}^b$. It has the same objects and morphisms and compositions as $\scC$ but is equipped with the deformed differential 
	\[
	d^b(\alpha) = d(\alpha) + b_{y'}\g \alpha - (-1)^{|\alpha|}\alpha \g b_y.
	\]
	for $\alpha \in \scC(y,y')$.
\end{definition}

\noindent 
We will from now on drop the subscript when passing from dg categories to $A_\infty$ categories and simply assume that all $A_\infty$ categories in this subsection have $\mu^k = 0$ for $k \ge 3$.

\begin{lemma}\label{lem:pushforward-bounding-cochain}
	If $\scF\cl \scC\to \text{Ch}_\Lambda$ is an $A_\infty$ functor and $b$ a Maurer--Cartan element for $\scC$, then for each $y$, the element
	\begin{equation}\label{} 
		b^\scF_y \; \coloneqq\; \sum_{d\ge 1} \,\scF^d_{(b_y, \dots, b_y)}
	\end{equation} 
	is a Maurer--Cartan element for $\normalfont\text{Ch}_\Lambda(\scF(y),\scF(y))$ and the maps
	\begin{equation}
		\label{eq:deformed-functor}
		\scF^b_{\sigma_d, \dots, \sigma_1} \;\coloneqq\; \sum \scF_{(b, \dots, b, \sigma_d, b, \dots, b, \sigma_1, b, \dots, b)},
	\end{equation}
	where we have just written $b$ instead of $b_y$ for appropriate $y$, define an $A_\infty$ functor $\scC^b \to \text{Ch}_\Lambda$, where $\scF^b(y)$ carries the differential deformed by $b^\scF_y$. If $\scF$ is unital, so is $\scF^b$.
\end{lemma}

\begin{proof} Since $\scF$ is an $A_\infty$ functor, we have $\mu^1_{\text{Ch}_\Lambda}(b^{\scF}_y) +\mu^2_{\text{Ch}_\Lambda}(b^{\scF}_y,b^{\scF}_y)\;=\;  \scF_{\mu^1_\scC(b_y)+ \mu^2_{\scC}(b_y,b_y)} \;=\; 0$.
\end{proof}

\begin{remark}
	A functor $\scF\cl \scC \rightarrow \text{Ch}_\Lambda$ is also known as a left $A_\infty$ module. We have chosen to state everything in terms of functors since it simplifies the characterisation of certain properties and is closer to Costello's set-up. The above example shows that given a left $A_\infty$ module over $\scC$, a Maurer--Cartan element $b$ for $\scC$ induces a left $A_\infty$ module over $\scC^b$.
\end{remark}

\noindent 
In general, if $\scC$ is symmetric monoidal, $\scC^b$ might no longer be, since the differential is changed. To ensure this, we can only deform by certain Maurer--Cartan elements of $\scC$.

\begin{lemma}\label{lem:symmetric-monoidal-bounding-cochain}
	We call a Maurer--Cartan element $b$ for $\scC$ \emph{symmetric monoidal} if 
	\begin{equation}\label{} b_{x\otimes y}\;=\; b_x\otimes_\scC \ide_y +\ide_x \otimes_{\scC} b_y\end{equation}
	for any $x,y \in \obj(\scC)$. In this case the symmetric monoidal structure on $\scC$ induces one on $\scC^b$.
\end{lemma}

\begin{proof}
	Since 
	\[(\alpha'\otimes \beta')\g (\alpha\otimes \beta) = (-1)^{|\alpha'||\beta|}(\alpha\g \alpha')\otimes (\beta'\g \beta)\]
	the claim follows from a straightforward computation.
\end{proof}

\noindent 
Suppose now $T$ is a pre-natural transformation between two $A_\infty$ functors $\scF$, $\scG\cl \scC \rightarrow \text{Ch}_{\Lambda}$. If $b$ is a Maurer--Cartan element for $\scC$, then the maps 
\begin{equation}\label{eq:deformed-pre-transformation-mc-element}
	T^b(\lc \sigma) := \sum T(b, \dots, b, \sigma_d, b, \dots, b,\sigma_{1}, b, \dots, b)
\end{equation}
define a pre-natural transformation between $\scF^b$ and $\scG^b$, where we have omitted subscripts on the Maurer--Cartan elements for simplicity.

\begin{lemma}
	\label{lem:deformation-of-natural-transformation}
	If $T$ is a natural transformation from $\scF$ to $\scG$, then $T^b$ is one from $\scF^b$ to $\scG^b$.
\end{lemma}

\begin{proof}
	Applying the equation $\mu_{\scQ}^1(T) = 0$ to tuples of the form $\lc \sigma^b = (b, \dots, b, \sigma_d, b, \dots, b,\sigma_{1}, b, \dots, b)$. The proof is then an easy verification of signs. We will sketch it here, omitting the signs. First observe that \[
	\sum_i \scG((\lc \sigma^b)^{>i}) \g T((\lc \sigma^b)^{\leq i}) = \scG(b, \dots, b) \g T^b(\lc \sigma) + \sum_{i = 0}^{d-1} \scG^b(\lc \sigma^{>i}) \g T^b(\lc \sigma^{\leq i}).
	\]
	Similarly\[
	\sum_i T((\lc \sigma^b)^{>i}) \g \scF((\lc \sigma^b)^{\leq i}) = T^b(\lc \sigma) \g \scF(b, \dots, b) + \sum_{i = 1}^d T^b(\lc \sigma^{>i}) \g \scF^b(\lc \sigma^{\leq i}). 
	\]
	Next observe that \[[\partial, T(\lc \sigma^b)] = [\partial^b, T^b(\lc \sigma)] - \scG(b, \dots, b) \g T^b(\lc \sigma) + (-1)^{|T|} T^b(\lc \sigma) \g \scF(b, \dots, b).\]
	The terms involving $\scF(b, \dots, b)$ and $\scG(b,\dots, b)$ thus cancel.
	Next, we find
	\[
	\sum_i T(\partial_i \lc \sigma^b) = \sum_i T^b(\partial_i \lc \sigma) + \sum T(b, \dots, b, \sigma_d, \dots, \partial b, \dots, \sigma_1, b, \dots, b),
	\]
	where the $\partial b$ could be anywhere (including to the left and right of $\sigma_d$, $\sigma_1$). Similarly
	\[\sum_i T(\star_i \lc \sigma^b) = \sum_i T^b(\star_i \lc \sigma) + \sum T(\dots \sigma_d, \dots, b \g b, \dots, \sigma_1,\dots) + \sum T(\dots, \sigma_d, \dots, [b, \sigma_i], \dots, \sigma_1,\dots).
	\] 
	The Maurer--Cartan equation for $b$ shows that \[
	\sum T(b, \dots, b, \sigma_d, \dots, \partial b, \dots, \sigma_1, b, \dots, b) + \sum T(\dots \sigma_d, \dots, b \g b, \dots, \sigma_1,\dots) = 0.
	\]
	Finally,
	\[
	T^b(\partial^b_i \lc \sigma) = T^b(\partial_i \lc \sigma) + \sum T(b,\dots, \sigma_d, \dots, [b, \sigma_i], \dots, \sigma_1,\dots,b)
	\]
	by definition of the differential on $\scC^b$.
\end{proof}

\begin{definition}\label{de:compatible-transformation}
	We say a pre-natural transformation $T$ is \emph{compatible} with the Maurer--Cartan element $b$ if $(T^b)^0 = T^0$.
\end{definition}

\begin{lemma}
	\label{lem:deformation-of-SM-functor}
	Let $\scF\cl \scC \rightarrow \text{Ch}_{\Lambda}$, be a monoidal $A_\infty$ functor and  $b$ be a symmetric monoidal Maurer--Cartan element for $\scC$. Then, the induced functor $\scF^b\cl \scC^b \rightarrow \text{Ch}_{\Lambda}$ is a monoidal $A_\infty$ functor. If $\scF$ is unital symmetric monoidal, then $\scF^b$ is unital symmetric monoidal. If $\scF$ is h-split, so is $\scF^b$.
\end{lemma}

\begin{proof}
	Let $R$ and $L$ be the two natural transformations giving the symmetric monoidal structure of $\scF$. We only discuss the case of $R$ as the other case is analogous. 
	Abbreviate $\scG\coloneqq \bigotimes_{\text{Ch}_\Lambda}\g (\scF\otimes\scF)$, defined in~\eqref{eq:tensor-F}, and $\scH:= \scF\g \otimes_\scC$.
	Recalling the definition of $\scC\boxtimes\scC$ from \S\ref{subsec:symmetric-monoidal}, we define the Maurer--Cartan element $\wt b$ by $\wt b_{(x,y)} \coloneqq b_x\otimes\ide_y + \ide_x \otimes b_y$. Then,
	$$\otimes_{\scC}(\wt b_{(x,y)})\; =\; b_{x\otimes y},$$\noindent
	which implies $\scH^{\wt b} = \scF^b\g\otimes_{\scC}$. Due to the unitality of $\scF$, we have $\scG^{\wt b} = \otimes_{\text{Ch}_\Lambda}\g (\scF^b\otimes\scF^b)$ and $R^{\wt b}$ defines the desired first part of the symmetric monoidal structure on $\scF^b$ by Lemma~\ref{lem:deformation-of-natural-transformation}. The same argument shows that $L^{\wt b}$ yields the other part of the symmetric monoidal structures. Since $\wt b$ is symmetric, $\scF^b$ is symmetric monoidal.
	Unitality follows from the fact that $\wt b_{(x,y)}$ is a sum of tensor products containing the identity. Finally, since $b$ and thus also $\wt b$ have positive valuation, one can use spectral sequences to prove that $(R^b)^0$ and $(L^b)^0$ are still quasi-isomorphisms if $R^0$ and $L^0$ are.
\end{proof}

\noindent 
We now specialise the discussion to our geometric setup of the previous subsections. Thus, $\scC = \cdmc^{unc}$ is the symmetric monoidal dg category of Definition~\ref{de:chains-for-dmft}.
 For any $y \in \cdmc$ we write $\wh D_1 \in \cdmc(y,y)$ for the morphism of the form \[\wh D_1 = \sum_{y_1,y_2} \idsimp^{y_1} \times D_1 \times \idsimp^{y_2},\]
 summing over all possible partitions of $y \sm \{\kappa\}$ for arbitrary $\kappa \in y_o$ with open points in $y_1$ ordered before those in $y_2$.

\begin{lemma}\label{}
	 $\wh D_1$ is a symmetric monoidal Maurer--Cartan element for the category $\cdmc^{unc}$ of Definition~\ref{de:category-without-curvature} and there exists a canonical inclusion
	 \[
	 \dmc_r \;\hkra\; (\cdmc^{unc}_r)^{\wh D_1}
	 \]
	 of symmetric monoidal dg categories. 
\end{lemma}

\begin{proof}
	The first claim from the fact that $\mu^1_{oc}(\wh D_1) = \wh D_1 \g \wh D_1 = - \mu^2_{oc}(\wh D_1, \wh D_1)$. For the second we observe that by Definition~\ref{de:deformation-of-dg-cat}, the differential of $(\cdmc^{unc}_r)^{\wh D_1}$ is given by  
	\begin{align*}
		\partial_{oc}^{\wh{D_1}}(\sigma) \;&=\; \partial_{oc}(\sigma) + \wh D_1 \g \sigma - (-1)^{|\sigma|}\sigma \g \wh D_1 \\&=\; \partial \sigma,
	\end{align*}
	where $\partial \sigma$ denotes the ordinary boundary operator on the singular chain $\sigma$.
\end{proof}

\noindent 
We can thus consider the deformed category $(\cdmc^{unc}_r)^{\wh{D_1}}$, and give a curved open-closed DMFT $\scF\cl \cdmc^{unc}_r \to \text{Ch}_\Lambda$ with vanishing curvature, the deformed $A_\infty$ functor $\scF^{\wh{D_1}}\cl (\cdmc^{unc}_r)^{\wh{D_1}} \rightarrow \text{Ch}_\Lambda$. Since $\wh D_1$ is symmetric monidal, Lemma~\ref{lem:deformation-of-SM-functor} implies

\begin{cor}\label{cor:zero-curvature-to-uncurved}
	If $\scF$ is an $r$-dimensional curved  open-closed DMFT with vanishing curvature, the $A_\infty$ functor $\wt \scF\cl \dmc_r \rightarrow \text{Ch}_{\Lambda}$ defined by the composite
	 \[
	 \dmc_r \;\hkra\;(\cdmc^{unc}_r)^{\wh{D_1}} \;\xrightarrow{\scF^{\wh D_1}} \;\text{Ch}_{\Lambda}
	\]
	and equipped with the deformed symmetric monoidal structure is an open-closed DMFT.\qed
\end{cor}

\noindent 
Note that $\wt \scF$ and $\scF$ do not agree on objects. For example, if $\scF(0,1) =  (V, d)$ is the open sector, then $\wt \scF(0,1) = (V, \wt d)$ where the deformed differential is given by $\wt d = d + \sum_{i \geq 1} \scF^i_{\wh{D_1}, \dots, \wh{D_1}}.$

\begin{cor}
	\label{cor:bounding-cochain-induces-uncurved}
	 Let $\scF$ be a curved open-closed DMFT with Property~\ref{assum:weakly-curved-to-uncurved}. If $\scF$ admits a weak bounding cochain $\bc$, then the deformed open-closed DMFT $\scfb$ induces an open-closed DMFT $\wt \scF^b\cl\dmc_r \rightarrow \text{Ch}_{\Lambda}$. If $\scF$ is strict, semi-strict or operadic, then so is $\wt \scF^b$.
\end{cor}

\begin{proof}
 The first claim follows from Corollary~\ref{cor:zero-curvature-to-uncurved} and Lemma~\ref{lem:weakly-curved-dmft}, while the last claim follows directly from the definition of strict, semi-strict or operadic because $\scF^b_{\lc\sigma}$ is defined in terms of sequences $\lc\sigma'$ that have length at least the length of $\lc\sigma$ and we are only adding $\wh{D_1}$.
\end{proof}

\subsection{Homotopies}\label{subsec:homotopies}
A natural transformation $T$ between $A_\infty$ functors $\scF$ and $\scG$ is called a \emph{homotopy} if $\scF$ and $\scG$ agree on objects, $T^0 = 0$ and $\mu_\scQ^1(T)^d = \scF^d - \scG^d$ for $d \ge 1$.
In particular, a homotopy necessarily has degree $0$.

\begin{definition}\label{de:homotopy-dmfts}
	A \emph{homotopy} $T$ between curved open-closed DMFTs $\scF^0$ and $\scF^1$ is a homotopy $T$ of $A_\infty$ functors so that two compositions of the square
	\begin{center}\begin{tikzcd}
			\otimes_{\text{Ch}_\Lambda}\g (\scF^0\otimes\scF^0) \arrow[r,"R^0"] \arrow[d,"\otimes_{\text{Ch}_\Lambda}\g (T\otimes T)"]&\scF^0\g\otimes_{\cdmc_r} \arrow[d,"T\g\otimes_{\cdmc_r}"]\\ 
			\otimes_{\text{Ch}_\Lambda}\g (\scF^1\otimes\scF^1) \arrow[r,"R^1"] & \scF^1\g \otimes_{\cdmc_r} \end{tikzcd} \end{center}
	differ by the boundary $\mu^1(P)$ of natural transformation $P$ from $\otimes_{\text{Ch}_\Lambda}\g (\scF^0\otimes\scF^0) $ to $\scF^1\g \otimes_{\cdmc_r}$, and similarly with $R^{(i)}$ replaced by $L^{(i)}$ and $P$ replaced by $P\g S$, where $S$ is the braiding of $\cdmc_r$.
\end{definition}

\noindent 
However, if $T$ is a homotopy, the deformed natural transformation $T^b$ need not satisfy $(T^b)^0 = 0$. Thus, $T^b$ might not be a homotopy, which is the motivation behind Definition~\ref{de:compatible-transformation}. In particular, we have the following observation.

\begin{lemma}\label{lem:homotopy-to-uncurved}
	Suppose $T$ is a homotopy between curved open-closed DMFTs $\scF^0$ and $\scF^1$ with vanishing curvature. If $T$ is compatbile with (the Maurer--Cartan element) $\wh{D_1}$, then $T^{\wh{D_1}}$ is a homotopy between the open-closed DMFTs $\scF^{0,\wh{D_1}}$ and $\scF^{1,\wh{D_1}}$.\qed
\end{lemma}

\noindent 
Recall that any semi-strict curved open-closed DMFT defines a curved $A_\infty$ algebra by Lemma~\ref{lem:curved-algebra-from-dmft}.

\begin{lemma}
	\label{lem:homotopies-induce-homs}
	A homotopy between two semi-strict curved open-closed DMFTs induces an isomorphism of the induced curved $A_\infty$ algebras.
\end{lemma}

\noindent 
The proof is similar to Fukaya's proof that pseudo-isotopies of $A_\infty$ algebras induce $A_\infty$ isomorphisms, \cite[Theorem~8.2]{Fu10}. Homotopies of DMFTs can thus be seen as a weakening of the notion of pseudo-isotopy.

\begin{proof}
	Following Fukaya \cite[Section~9]{Fu10}, define for any $k \in \bN$ the set $\text{Trees}(k)$ of (isomorphic classes of) all rooted trees with the following properties: the root vertex $v_0$ has exactly one output, and the tree has $k$ open inputs. Every vertex $v$ is labeled with energy $\beta(v)\in [0,\infty)$ with $\beta(v) = 0$ only if $v$ has at least three edges. We additionally equip our trees with a choice of total order on the vertices which respects the partial order defined by the tree, where $v' < v$ if $v'$ lies on the shortest path from $v$ to $v_0$. To any $\Gamma \in \text{Trees}(k)$ with $d$ vertices associate the tuple $\lc \sigma(\Gamma) = (\sigma_d, \dots, \sigma_1)$, where $\sigma_d$ is the chain $D_{|v_0|-1}$, and the other $\sigma_i$ consists of a disjoint union of $D_{|v_i|-1}$ and $\idsimp$ as required, so that the sequence $\lc \sigma(\Gamma)$ is a composable tuple from $k$ open inputs to one open output. 
	
	Given a homotopy $T$ between $\scF^0$ and $\scF^1$, denote the associated $A_\infty$ algebras by $\scA_i$, noting that their underlying $\Lambda$-modules are the same. Then define the maps $f^k\cl \scA_0^{\otimes k} \rightarrow \scA_1$
	by 
	\begin{equation}
		f^k(\alpha_1, \dots, \alpha_k) \;=\; \sum_{\substack{\Gamma \in Trees(k)}} (-1)^{\epsilon(\alpha)}T_{\lc \sigma(\Gamma)}(\alpha_1\times \dots\times \alpha_k),
	\end{equation}
	where $\epsilon(\alpha)$ was defined in~\eqref{eq:st-sign} and $\alpha\times \beta = (-1)^{|\alpha||\beta|}R^0(\alpha\otimes \beta)$.
	
	We claim that $f =(f^k)_{k\ge 1}$ defines an $A_\infty$ algebra homomorphism. First observe that the terms from the homotopy equation~\eqref{eq:fun-differential} with $T_{\partial_i \lc \sigma( \Gamma)}$ and $T_{\concat_i \lc \sigma(\Gamma)}$ cancel with each other after summing over all trees. We then separate out the cases where the length of $\lc \sigma(\Gamma)$ is $1$ and where it is $>1$.
	
	If its length is $1$, we have $\sigma(\Gamma) = D_k$. First, observe that $f^1$ in energy zero is given by $T_{\idsimp} = \ide$. Ignoring the terms that cancel, Equation~\eqref{eq:fun-differential} implies (omitting signs)
	\begin{equation*}
		[d, T_{D_k}] \,+\, T_{\del_{oc}D_k}\;=\; \scF^0_{D_k} \,-\, \scF^1_{D_k}.
	\end{equation*}
	This gives rise to the terms $\mu_{\scA_1}^{1,0} \g f^k$, $f^k \g \mu_{\scA_0}^{1,0}$, $\mu_{\scA_1}^{k} \g (f^{1,0} \otimes \dots \otimes f^{1,0})$ and $f^{1,0} \g \mu_{\scA_0}^k$ in the $A_\infty$ homomorphism equation.
	For sequences $\lc\sigma(\Gamma)$ of length $> 1$, note that as $\scF^i$ is semi-strict, $\scF^i_{\lc \sigma(\Gamma)} = 0$. Thus, ignoring signs and cancelling terms, \begin{equation*}
		[d, T_{\lc \sigma(\Gamma)}] \;=\; \scF^1_{\sigma_d} \g T_{\lc \sigma(\Gamma)^{<d}} \,+\, T_{\lc \sigma(\Gamma)^{>1}} \g \scF^0_{\sigma_1}.
	\end{equation*}
	The term on the left-hand side contributes to $\mu_{\scA_1}^{1,0} \g f^k$ and $f^k \g \mu_{\scA_0}^{1,0}$. For the term on the right-hand side, note that $\sigma_d = D_{|v_0|-1}$, so that $\lc \sigma(\Gamma)^{<d}$ consists of $|v_0| - 1$ trees, and \begin{equation*}
			\scF^1_{\sigma_d} \g T_{\lc \sigma(\Gamma)^{<d}}\; =\; \mu_{\scA_1}^{|v_0|-1}(f^{k_1}(\dots),\dots, f^{k_{|v_0|-1}}(\dots)).
		\end{equation*} 
		For the other term, note that $\sigma_1$ is a product of $D_{r}$ and identity simplices $\idsimp$, so that 
		\begin{equation*}
			T_{\lc \sigma(\Gamma)^{>1}} \g \scF^0_{\sigma_1}\; =\; f^{k-r}(\dots, \mu_{\scA_0}^r(\dots), \dots),
		\end{equation*}
		as required.
\end{proof}

%
%

\section{A Deligne--Mumford field theory from moduli spaces of stable maps}\label{subsec:lagrangian-dmft}
\noindent 
Suppose $(X,\omega)$ is a closed symplectic manifold and $L \sub X$ an embedded Lagrangian submanifold equipped with a relative spin structure $(V,\fs)$. Define the Novikov ring
\begin{align}\label{eq:novikov-ring}
	\Lambda &\coloneqq \set{  \sum_{i = 0}^{\infty}a_iQ^{\beta_i}\,\big|\, a_i \in \mathbb{R},\; \beta_i \in H_2(X, L), \; \lim_{i\to \infty } \omega(\beta_i) = \infty }
\end{align}
This subsection is devoted to proving Theorem~\ref{thm:OCDMFT associated to L}, using the coefficient ring $\Lambda$ endowed with the valuation $\nu \cl \Lambda\to \bR$ given by 
\begin{equation*}\label{eq:valuation}
	\nu\lbr{\s{j}{a_j\,Q^{\beta_j}}} = \infi{\substack{j\\a_j \neq 0}}{\omega(\beta_j)}.
\end{equation*}
We first construct a curved open-closed Deligne--Mumford field theory associated to a relatively spin Lagrangian submanifold.

\begin{theorem}\label{thm:dmft-exists} For any choice of $\omega$-tame almost complex structure $J$, of cubical cobordisms $\{\cKc_{\Gamma}\}_\Gamma$ and compatible system of Thom forms $\eta_{\Gamma}$, there exists a semi-strict, operadic curved open-closed Deligne--Mumford field theory
	\[\scF_L\cl \cdmc_{n} \to \normalfont\text{Ch}_{\Lambda}\]
	defined using the open-closed Deligne-Mumford moduli spaces of $(X,L,J)$.
\end{theorem}

\noindent 
Cubical cobordisms and Thom systems were defined in \cite{HH25} and are recalled in the next subsection. 
Applying the discussion of weak bounding cochains from \S\ref{subsec:bounding-cochains-and-deformations} in the context of curved open-closed DMFTs yields the second part of Theorem~\ref{thm:OCDMFT associated to L}.

\begin{cor}\label{cor:actual-dmft} Suppose $L$ is equipped with a weak bounding cochain $\bc$ for the unital $A_\infty$ algebra $\scA_\scF$. Then, $L$ admits a semi-strict, operadic $n$-dimensional open-closed DMFT whose associated cohomological field theory recovers the Gromov--Witten theory of $(X,\omega)$.
\end{cor}

\begin{proof}
	By Corollary~\ref{cor:actual-dmft} and Lemma~\ref{lem:additional-properties}, which shows that $\scF_L$ has Property~\ref{assum:weakly-curved-to-uncurved}, the curved open-closed DMFT $\scfb_L$ has vanishing curvature. Thus, the claim follows from Corollary~\ref{cor:bounding-cochain-induces-uncurved}.
\end{proof}

\begin{remark}
Somewhat surprisingly, all operations of $\scF_L$ are compatible with $D_1$ except for the open co-pairing. Concretely, $\scF_{\lc \sigma} = 0$ whenever $d\ge 2$, $\sigma_i = D_1$ for some $i$ and no other $\sigma_j$ represents the co-pairing. Given a weak bounding cochain $\bc$, it follows from Equation \eqref{eq:deformed-functor} that the functor $\wt \scfb$ only differs from $\scfb$ for sequences $\lc \sigma$ involving the co-pairing. The higher homotopies $\scfb_{D_1, \dots,  D_1}$ also vanish, so the differential on $\wt \scfb(0,1)$ is given by $d + \scF_{D_1}$, which is exactly the Floer differential.
\end{remark}

\subsection{A suitable chain model}\label{subsec:our-chains}
In order to achieve semi-strictness, we construct a special chain model for the real-valued singular homology on the moduli spaces $\Mbar_{\mathsf{a}}$. To this end, we consider `fattened' versions of the moduli spaces $\Mbar_{g,h;y,y'}$, requiring some preliminary definitions.
Recall from \cite[\S3]{HH25} that by a \emph{stable graph} we mean the combinatorial data underlying a possibly disconnected and nodal Riemann surface with boundary and marked points. In this paper, this comes with the additional data of an `orientation' of each half-edge, denoting whether they are incoming or outgoing. 
\emph{We fix for each isomorphism class of stable graph a representative and we will work with that throughout the remainder of the paper.}
Each stable graph $\Gamma$ determines the product moduli space
\begin{equation}
	\Mbar_{\Gamma} =  \prod_{v \in V^s(\Gamma)} \Mbar_{\mfa_v}.
\end{equation}
We have the relation \[
\del_\Gamma\Mbar_{\mathsf{a}} \cong \frac{\Mbar_{\Gamma}}{\Aut(\Gamma)}.
\]
We define the auxiliary moduli spaces \[\cMc_\Gamma := I^{E(\Gamma)}\times \Mbar_\Gamma.\] They will be used to define a suitable chain model for the singular homology of the moduli spaces $\Mbar_{\mathsf{a}}$. Define 
\begin{equation}\label{} 
	\Mbar_{\mathsf{a}}^+\coloneqq \djun{\Gamma\to \mathsf{a}}\cMc_\Gamma.
\end{equation}
While, we do not glue the moduli spaces together, this construction is similar to the outer-collaring of $\Mbar_{\mathsf{a}}$ as used in \cite[Chapter~17]{FOOO20} and \cite[\S3.3.1]{BX22}, considering the moduli spaces $\cMc_\Gamma$ as spikes on $\Mbar_{\mathsf{a}}$. Note that we have additional `spikes' because we consider our half-edges to be oriented. Readers who will find Construction~\ref{con:chains-for-dmft} too algebraic can equally well think of the moduli spaces $\Mbar_{\mathsf{a}}^+$ as actually glued along their common boundary strata and of the chains we take as stratified chains on that `spiky moduli space'. To help that intuition, we include the following picture.

\begin{remark}
	\vspace{-8pt}
	\begin{figure}[H]
		\includegraphics[scale = 0.22]{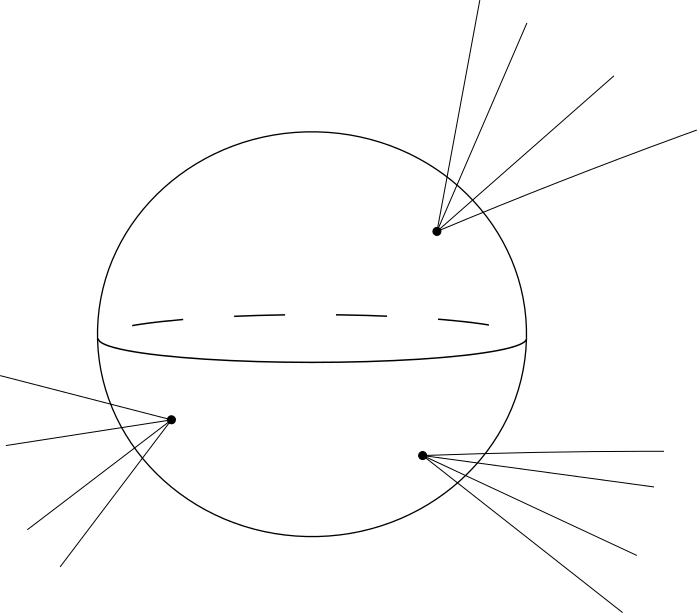}
		\caption{$\Mbar^+_{0,0;(0,0),(4,0)}$}
		\label{fig:spiky-moduli-space}
	\end{figure}
	The usual moduli space of spheres with $4$ marked points $\Mbar_{g= 0, k = 4} \cong \mathbb{CP}^1$ has 3 `special points' where the configuration can be viewed as two spheres clutched at a single node. In the usual outer collaring, one would add one spike (intervals) to each special point. Since we have to keep track of the orientations of the half-edges, we instead have to add $4$ spikes to each special point, making a total of $12$ spikes for the moduli space $\Mbar^+_{0,0;(0,0),(4,0)}$ as drawn above.
\end{remark}

\noindent 
We extend the forgetful and clutching maps on the moduli spaces $\Mbar_{\mathsf{a}}$ to $\Mbar_{\mathsf{a}}^+$ as follows.
\begin{enumerate}[label=(E\hspace{0.2pt}\arabic*),leftmargin=30pt,ref=E\hspace{0.2pt}\arabic*,]
	\item\label{def:clutching-map-on-fattened-moduli} The clutching $\Gamma$ of two graphs $\Gamma_1$ and $\Gamma_2$ induces an embedding
	\begin{equation}
		\label{eq:lift-of-clutching} \cMc_{\Gamma_1} \times \cMc_{\Gamma_2} \rightarrow \cMc_{\Gamma}.
	\end{equation} 
	given by the identity on the moduli spaes $\Mbar_{\Gamma_i}$ and by the canonical inclusion map 
	\begin{equation}
		\label{eq:inclusion-cubes}I^{E(\Gamma_1)} \times I^{E(\Gamma_2)} \hookrightarrow I^{E(\Gamma)}
	\end{equation}
	induced by the embedding $E(\Gamma_1)\sqcup E(\Gamma_2)\hkra E(\Gamma)$	onto the boundary stratum $\{t_e = 1\}_{e \in E(\psi)}$.
	\item\label{def:boundary-contraction-fattened-moduli} The inclusion $\rho \cl \Mbar_{g,h;k,\ell} \rightarrow \Mbar_{g,h+1;k-1,\ell}$ of the stratum with a boundary node of type (E) defines an embedding $\Mbar_\Gamma\hkra \Mbar_{\rho(\Gamma)}$. Moreover $E(\Gamma)$ embeds canonically into $E(\rho(\Gamma))$ so we obtain a smooth real analytic map 
	$$\cMc_\Gamma\hkra\cMc_{\rho(\Gamma)}$$ 
	by including into $\{t_e = 1\}$, where $e$ is the (interior) edge connecting the `collapsed boundary' vertex to the rest of the graph.
	\item\label{def:forgetful-map-on-fattened-moduli} For a dual graph $\Gamma$, let $\pi(\Gamma)$ denote the (possibly stabilised) dual graph obtained by forgetting the respective marked point. Then, the map 
	$$\cMc_\Gamma\,\to\, \cMc_{\pi(\Gamma)}$$
	is given by the natural map $\Mbar_\Gamma\rightarrow \Mbar_{\pi(\Gamma)}$ in the second factor, and by the following map $I^{E(\Gamma)} \rightarrow I^{E(\pi(\Gamma))}$ between cubes: if $T\sub \Gamma$ is a maximal subgraph that is contracted, then $T$ is either attached by a single edge $e$ to the rest of $\Gamma$ or by two edges $e_1,e_2$. In the first case, we forget $I^{E(T)}\times I^{\{e\}}$ and in the second case we map $t \in I^{E(T)}\times I^{\{e_1,e_2\}}$ to 
	\begin{equation*}\label{}
		t_{\bar{e}} := 1- \p{e \in E(T)\cup\{e_1,e_2\}}{(1-t_e)}
	\end{equation*}
	where $\bar{e}$ is the edge onto which $T$ is contracted as described pictorially by
	\begin{figure}[H]
		\includegraphics[scale = 0.5]{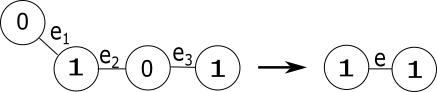}
		\label{fig:contraction}
	\end{figure}
	\noindent Compare this map with the usual local expression of the forgetful map given by $(z_1,z_2)\mapsto z_1z_2$. In the map above, we have to compose with $1-t$ because we use outer-collarings and not inner-collarings.
\end{enumerate}


We can now construct the chain model we use for the definition of $\dmc$ and ultimately in \textsection\ref{subsec:lagrangian-dmft}.

\begin{construction}\label{con:chains-for-dmft} Fix a representative $\Gamma$ for each isomorphism class $[\Gamma]$ of open-closed stable graph as in \cite[Definition~4.13]{HH25}. Define the category $\scS$ to have 
	\begin{itemize}[leftmargin=20pt]
		\item objects $(\Gamma,E,f_\Gamma)$, where $\Gamma$ is a stable graph and $E\sub E(\Gamma)$ is a subset of edges,
		\item morphisms from $(\Gamma,E)\to (\Gamma',E')$ given by (orientation-preserving) morphisms $f \cl \Gamma\to \Gamma'$ in the sense of \cite[Definition~4.14]{HH25} so that any edge contracted by $f$ lies in $E$ and $E'$ consists of the images of those edges in $E$ that are not contracted by $f$.  
	\end{itemize}
	Given $\Gamma$, we let $\mathsf{a}_\Gamma$ be the topological data of the unique corolla (i.e. graph without edges) $\Gamma$ contracts to. We denote said by corolla by $\mathsf{a}$ as well and let $\scS_{\mathsf{a}}$ be the full subcategory on objects $(\Gamma,E)$ with $\mathsf{a}_\Gamma = \mathsf{a}$.
	Define the functor $\scC\cl \scS\to \normalfont\text{Ch}_\bR$ by setting 
	\begin{equation*}\label{}
		\scC(\Gamma,E) := C^{an}_*(\cMc_\Gamma|_{\{t_E = 0\}})
	\end{equation*}
	and letting $\scC(f) = (\ide_{I^{E'}}\times \psi^f)_*$ be the pushforward by the identity and the associated clutching map. 
	
	\begin{lemma}\label{lem:right-homology} For any $\mathsf{a}$ we have $H_*(\normalfont\text{colim}_{\scS_{\mathsf{a}}}\scC) \cong H_*(\Mbar_{\mathsf{a}};\bR)$. 
	\end{lemma}
	
	\begin{proof}
		Define the `reduced' functor $\ov\scC\cl \scS_{\mathsf{a}} \to \normalfont\text{Ch}_\bR$ by $\ov\scC(\Gamma,E) = \scC(\Gamma,E(\Gamma))$ and $\ov\scC(f) = \psi^f_*$. Then, there exist canonical natural transformations $\eta \cl \ov\scC\to \scC|_{\scS_{\mathsf{a}}}$ and $\xi \cl \scC|_{\scS_{\mathsf{a}}}\to \ov\scC$, induced by the inclusions $\cMc_\Gamma|_{\{t_{E(\Gamma)}= 0\}}\hkra \cMc_\Gamma|_{\{t_{E}= 0\}}$ and the canonical retraction $\cMc_\Gamma|_{\{t_{E}= 0\}}\hkra \cMc_\Gamma|_{\{t_{E(\Gamma)}= 0\}}$ respectively. These are natural homotopy inverses in the sense that the homotopies are compatible with the morphisms of $\scS_{\mathsf{a}}$. In particular, this shows that 
		\begin{equation}\label{} 
			H_*(\normalfont\text{colim}_{\scS_{\mathsf{a}}}\scC) \cong H_*(\normalfont\text{colim}\,\ov\scC).
		\end{equation}
		Let $\ov\scS_{\mathsf{a}} \sub \scS_{\mathsf{a}}$ be the full subcategory on objects $(\Gamma,E(\Gamma))$. Note that $\ov\scS_{\mathsf{a}}$ has a terminal object given by $\mathsf{a}$ itself. Moreover, there exists a right inverse $R\cl\scS_{\mathsf{a}} \to \ov\scS_{\mathsf{a}}$ the inclusion $\iota$, given by $R(\Gamma,E) = (\Gamma,E(\Gamma))$. Since $\ov\cC\g \iota \g R = \ov\cC$, it follows that 
		\begin{equation*}\label{} 
			H_*(\normalfont\text{colim}\,\ov\scC)\cong H_*(\normalfont\text{colim}_{\ov\scS_{\mathsf{a}}}\,\ov\scC)\cong H_*(\Mbar_{\mathsf{a}};\bR),
		\end{equation*}
		where the last isomorphism is due to the fact that $\mathsf{a}$ is terminal in $\ov\scS_{\mathsf{a}}$ and by Lemma~\ref{lem:analytic-homology}.
	\end{proof}
\end{construction}

\begin{definition}\label{de:chain-model}
	We define the symmetric monoidal category $\dmc$ enriched in cochain complexes to have the same objects as $\Mbar_{oc}$ with morphisms given by the cochain complexes
	\[\dmc^*(y,y') =  \bigoplus\limits_{\mathsf{a} \to y,y'}\normalfont\text{colim}_{\scS_{\mathsf{a}}}\scC_{-*},\]
	where $\scC$ is the functor defined in Construction~\ref{con:chains-for-dmft}. The composition maps 
	\begin{gather*}\label{}\dmc(y',y'') \otimes \dmc(y,y')\to \dmc(y,y'')\end{gather*}
	are induced by the composition maps~\eqref{def:clutching-map-on-fattened-moduli}-~\eqref{def:forgetful-map-on-fattened-moduli} and the Alexander--Whitney choice of Eilenberg--Zilber morphism (if in the stable range).\footnote{Given the explicit description of this choice of Eilenberg--Zilber morphism, one sees easily that analyticity is preserved.}
	This is well-defined as the clutching maps are local embeddings. Our choice of Eilenberg--Zilber morphism ensures that the compositions are strictly associative.
\end{definition}

\noindent 
The category $\dmc_r$ has the same objects as $\Mbar_{oc}$ with morphisms given by 
\begin{equation}\label{de:twisted-chains} \dmc_r^*(y,y') \coloneqq \bigoplus\limits_{\mathsf{a} \to y,y'}\normalfont\text{colim}_{\scS_{\mathsf{a}}}\scC_r^*\end{equation}
for $r \in \bZ$, where 
\begin{equation}\label{}
	\scC_r^*(\Gamma,E) \coloneqq C^{an}_{-*}(\cMc_\Gamma;{\det}^{\otimes r}) \end{equation}
with the `signed' differential as above. The argument of Lemma~\ref{lem:right-homology} carries over verbatim, showing that this recovers (cohomologically graded) singular homology with the local coefficients in ${\text{det}}^{\otimes r}$. Thus, a simplex $\wt\sigma\in \dmc_r$ is given by a simplex $\sigma\cl \Delta^m\to \Mbar_{oc}(y,y')$ together with a trivialisation of 
\begin{equation}\label{} 
	\det(\sigma) \coloneqq \sigma^*{\det}^{\otimes r}[r\chi + rd|y'_c|+r|y'_o|].
\end{equation}
We now define the pull-back $\pi^*: \dmc_r^*(y,y') \rightarrow \dmc_r^*(y \sqcup *,y')$ on this modified chain complex. Suppose $\sigma \cl \Delta^k \to \cMc_{\Gamma}$ is a simplex and let $\Gamma^+$ be a stable graph which contracts to $\Gamma$ after forgetting a last incoming boundary marked point. Then, by taking an analytic triangulation of the resolution of the fibre product $\Delta^k\times_{\cMc_{\Gamma}}\cMc_{\Gamma^+}$, one obtains a chain $\pi^*_{\Gamma^+} \sigma$. The proof of  Lemma~\ref{lem:lift-well-defined} carries over verbatim to show this is well-defined.
\begin{definition}
\label{de:lift-of-chain-modified-chains}
Define the chain $\pi^*\sigma$ by $\pi^*\sigma := \sum_{\Gamma^+ \to \Gamma} \pi^*_{\Gamma^+} \sigma,$ where we sum over all stable graphs $\Gamma^+$ that map to $\Gamma$ when forgetting the last incoming boundary marked point.
\end{definition}

\noindent
The proof of Lemma~\ref{lem:lifted-simplex-analytic} shows that $\pi^*\sigma$ has Properties~\eqref{pushforward-degenerate}-\eqref{lift-differential}. Thus, we can define the symmetric monoidal dg category $\cdmc_r$ as in Definition~\ref{de:chains-for-dmft} using the dg category $\dmc_r$ of Definition~\ref{de:chain-model} instead of $C_*^{an}(\Mbar_{oc})$.
\begin{lemma}
	There exist a symmetric monoidal $A_\infty$ quasi-equivalence $C_{-*}^{an}(\Mbar_{oc}) \rightarrow \dmc$.
\end{lemma}

\begin{proof}
	Define the functor $\scI: C_{-*}^{an}(\Mbar_{oc}) \rightarrow \dmc$ by $\scI^1(\sigma) = \sigma$, while for $d \geq 2$, if $\lc \sigma$ contains any unstable chains, set $\scI^d(\lc \sigma) = 0$. Otherwise, if $\lc \sigma$ only contains stable chains, define \[
	\scI^d(\lc \sigma) := \sigma_d \times \dots \times \sigma_1 \times I^{d-1} \in C^{an}_{-*}(\cMc_\Gamma;{\det}^{\otimes r})
	\]
	where $\Gamma$ denotes the stable graph obtained by concatenating all $\sigma$ and $I^{d-1} \hookrightarrow I^{E(\Gamma)}$ is the (diagonal) inclusion of the edges separating the $\sigma_i$.
\end{proof}
\subsection{Correspondences}\label{subsec:correspondences}
We introduce the notion of a pull-push operation obtained from a suitable correspondence and define the correspondences required for our curved open-closed DMFT.

\begin{definition}\label{de:operation-from-correspondence}
	To a correspondence $\fC := X\xleftarrow{f} Y \xra{g} Z$ with $g$ a submersion and $Y$ and $Z$ oriented (respectively $g$ co-oriented) orbifolds, we associate the linear operation 
	\begin{equation}\label{eq:operation-from-correspondence}
		\fq_\fC \cl \Omega^*_c(X)\to \Omega^*(Z) : \alpha \mapsto g_*f^*\alpha,
	\end{equation}
	If we have instead $Y$ an oriented global Kuranishi chart $\cK= (\cT,\cE,\obs)$ with corners equipped with a Thom form $\tau$, we can define $\fq_\fC$ by 
	\begin{equation*}\label{} 
		\fq_\fC(\alpha) = g_*(f^*(\alpha)\wedge \obs^*\tau)
	\end{equation*}
\end{definition}

\noindent
In Theorem~\ref{thm:cubical-cobordisms-enhanced}, Proposition~\ref{prop:thom systems exist} and Lemma~\ref{lem:gkc-for-higher-homotopies}, we construct the correspondences required for $\scF_L$. To show compatibility, we require the following notion of a morphism of correspondences.

\begin{definition}\label{de:morphism-correspondence} A \emph{morphism} $\pi \cl \fC\to \fC'$ \emph{of correspondences} is a morphism of global Kuranishi charts $\pi \cl \cK\to \cK'$ so that $\cE = \pi^*\cE'$ and $\tau = \wt\pi^*\tau'$, while $f'\g \pi = f$ and $g'\g f = g$.
\end{definition}

\begin{definition}\label{de:degenerate-correspondence} We call a correspondence $X\leftarrow (\cK,\tau)\to Y$ \emph{degenerate} if it admits a morphism to a correspondence of strictly lower virtual dimension.
\end{definition}

\begin{definition}\label{} Suppose $\fC = X\xla{f}(\cK,\tau)\xra{g} Z$.
	If $\vartheta\cl Z \to Z'$ is a submersion, we define the \emph{pushforward} $\vartheta_*\fC$ to be the correspondence $X\xla{f}(\cK,\tau)\xra{\vartheta g} Z'$. Similarly, if $\theta\cl X'\to X$ is a submersion,  we define the pullback $\theta^*\fC$ to be $X'\xla{f}(X'\times_X\cK,\tau)\xra{g} Z$.
\end{definition}

\noindent
We observe the following properties of these operations. They make the proofs of the functoriality of $\scF$ cleaner and more geometric.

\begin{proposition}\label{prop:correspondences-operations} The operations associated to correspondences satisfy the following properties.
	\begin{enumerate}[leftmargin=20pt]
		\item\label{pushforward-and-pullback} We have $\fq_{\vartheta_*\fC} = \vartheta_*\g\fq_\fC$ and $\fq_{\theta^*\fC} = \fq_\fC\g \theta^*$ for any correspondence $\fC$ and any maps $\vartheta$ and $\theta$ as above.
		\item\label{composing-correspondences} Suppose $\fC_i := X_i\xleftarrow{f_i} (\cK_i,\tau_i)\xra{g_i} Z_i$ is a correspondence for $i = \{0,1\}$ and $Z_0 = X_1$. Then, the composition $\fq_{\fC_1}\g \fq_{\fC_0}$ is the operation associated to the fibre product correspondence
		\begin{equation*}\label{} 
			\fC_1\times_{Z_0}\fC_0\; \coloneqq\; X_0\xla{f_0 p_0}\, (\cK_1\times_{Z_0} \cK_0,\tau_1\times\tau_0)\, \xra{g_1p_1} Z_1,
		\end{equation*}
		where $p_0$ and $p_1$ are the canonical projections of the fibre product.
		\item\label{reversing-orientation} Suppose $\fC := X\xleftarrow{f} (\cK,\tau)\xra{g} Z$ is a correspondence and $\phi$ an orientation-reversing diffeomorphism of $Y$ so that $g\g \phi= g$, $f\g \phi = f$ and $\wt\phi^*\tau = \tau$. Then $\fq_\fC = 0$. 
		\item\label{products-associative} Given three correspondences $\fC_0,\fC_1$ and $\fC_2$ we have $(\fq_{\fC_0}\times\fq_{\fC_1}) \times \fq_{\fC_2} = \fq_{\fC_0}\times (\fq_{\fC_1}\times\fq_{\fC_2})$.
		\item\label{degenerate-vanishes} If a correspondence $\fC$ is degenerate, then $\fq_\fC$ vanishes.
	\end{enumerate}
\end{proposition}

\begin{proof} These properties follow from a straightforward verification using the properties of the pushforward, \cite[Proposition~2.1]{ST16}.
\end{proof}

\noindent
We recall here the cubical cobordisms constructed in \cite{HH25} that are the main ingredient for achieving strict compatibility of (some of) the chain-level operations. A variation of the construction yields the correspondences that will be used to construct the higher homotopies of the curved open-closed DMFT.

\begin{theorem}[{\cite[Theorem~3.11]{HH25}}]\label{thm:cubical-cobordisms-enhanced} Given a choice of unobstructed auxiliary datum $\alpha_{\mathsf{a}}$ for each $\mathsf{a}\in \cA$, there exists for each stable map graph $\Gamma$ a \emph{smooth} global Kuranishi chart $\wt{\cKc_\Gamma}$ for $I^{E(\Gamma)}\times\Mbar_{\Gamma}(X,L)$ so that the following holds
	\begin{enumerate}[\normalfont 1),leftmargin=18pt,ref=\arabic*]
		\item $\wt{\cKc_{\mathsf{a}}} = \wt\cK_{\mathsf{a}}$ for any $\mathsf{a} \in \cA$.
		\item $\wt{\cKc_\Gamma}$ is orientable.
		\item\label{cubical-1-boundary-enhanced} for any edge $e$ of $\Gamma$ we have a local smooth embedding
		\begin{equation}\label{eq:restricted-e-0-enhanced} 
			\wt{\cKc_{\Gamma}}|_{\{t_e = 0\}} \;\stackrel{\sim\,}{\longrightarrow}\; (\pm 1)\,\del_\Gamma\wt{\cKc_{\Gamma_e}}
		\end{equation}
		of degree \[d_1(\Gamma,\Gamma_e) = |\{\phi\in \Aut(\Gamma)\mid f\g\phi = f\},\] where $f\cl \Gamma\to \Gamma_e$ is the underlying contraction of graphs.
		\item\label{cubical-0-boundary-enhanced} The restriction to the other boundary face induces a smooth embedding
		\begin{equation}\label{eq:restricted-e-1-enhanced} \wt{\cKc_{\Gamma}}|_{\{t_e = 1\}} \;\stackrel{\sim\,}{\longrightarrow}\; (\pm 1)\,\wt{\cKc_{\Gamma_1}}\times_{Y}\wt{\cKc_{\Gamma_2}}
		\end{equation}
		where $\Gamma_1*_e \Gamma_2 = \Gamma$ and $Y = X$ or $L$, depending on whether $e$ is an interior or a boundary edge.
		\item\label{evaluation-cubical} $\wt{\cKc_\Gamma}$ admits a smooth submersive lift $\eva_\Gamma\cl \wt{\cKc_\Gamma}\to X^k\times L^\ell$ of the evaluation map on $\Mbar_\Gamma(X,L)$ and the (local) embeddings of~\eqref{eq:restricted-e-0-enhanced} and~\eqref{eq:restricted-e-1-enhanced} intertwine these lifts.
	\end{enumerate}
\end{theorem}

\begin{proposition}\label{prop:thom systems exist} Given a system of (enhanced) cubical cobordisms $\{\cKc_{\Gamma}\}_{\Gamma}$ and a choice of Thom form $\eta_{\mathsf{a}}$ for $\cK_{\mathsf{a}}$ so that $\obs_{\mathsf{a}}^*\eta_{\mathsf{a}}$ is compactly supported, there exists a system of Thom forms $\eta_{\Gamma}$ for $\cE_{\Gamma}$ with the following properties.
	\begin{enumerate}[label=\normalfont\arabic*),leftmargin=15pt,ref=\roman*]
		\item The pullback $\obs_{\Gamma}^*\eta_{\Gamma}$ is compactly supported in $\cTc_{\Gamma}$.
		\item If $\Gamma = \mathsf{a}$, then $\eta_{\Gamma} = \eta_{\mathsf{a}}$.
		\item It is compatible with boundary restrictions. More precisely, for any edge $e$ of $\Gamma$, we have that
		\begin{equation}
			\label{eq:thom-contracting-edge}
			\eta_\Gamma|_{\{t_e = 0\}} \sim \eta_{\Gamma_e}|_{\del_\Gamma\cKc_{\Gamma_e}}, 
		\end{equation} 
		and
		\begin{equation}
			\label{eq:thom-separating-edge}
			\eta_\Gamma|_{\{t_e = 1\}} \sim \eta_{\Gamma_1}\times \eta_{\Gamma_2}|_{\cTc_{\Gamma_1}\times_Y \cTc_{\Gamma_2}} 
		\end{equation}
		if $e $ separates $\Gamma$ into $\Gamma_1$ and $\Gamma_2$ with $Y = X$ or $L$ depending on $e$, and 
		\begin{equation}
			\label{eq:thom-nonseparating-edge}
			\eta_\Gamma|_{\{t_e = 1\}} \sim \eta_{\Gamma\sm e}|_{Y\times_{Y^2}\cKc_{\Gamma\sm e}} 
		\end{equation}
		if $e$ is a non-separating node and $\Gamma\sm e$ is the graph given by cutting the edge $e$. Finally, if the edge corresponds to a collapsed boundary circle, then
		\begin{equation}
			\label{eq:thom-boundary-edge-type-E}
			\eta_\Gamma\mid_{\{t_e = 1\}} \sim \eta_{\Gamma\sm e}|_{L\times_{X}\cKc_{\Gamma\sm e}}.
		\end{equation} 
	\end{enumerate}
\end{proposition}

\noindent
While these results holds for all stable map graphs, we only consider graphs $\Gamma$ whose unstable vertices (after forgetting the degree) are of type $(0,1)$, i.e. encode discs. In this case, we can construct relative orientations for $\cKc_\Gamma$ over $\Mbar_{\Gamma^{\stb}}$.

\begin{construction}[Orientation]\label{con:orientation}Suppose first $\Gamma$ is a stable map graph whose complete contraction $\mfa_\Gamma$ is stable after forgetting the degree. In this case, the forgetful map $\cTc_\Gamma\to I^{E(\Gamma^{\stb})}\times\Mbar_{\Gamma^{\stb}}$ has a canonical relative orientation defined as follows. This map factors through the base space $\cBc_\Gamma$, and since $\cTc_\Gamma\to \cBc_\Gamma$ has a canonically relative orientation by \cite[Lemma~2.29]{HH25}, it suffices to consider 
	$$\ff_{\,\Gamma}\cl \cBc_\Gamma\to I^{E(\Gamma^{\stb})}\times\Mbar_{\Gamma^{\stb}}$$\noindent 
	instead. Using that a disc bubble creates a boundary of codimension $1$, we obtain an open $\cBc_\Gamma\hkra \cBc_{\Gamma'}$ onto a collar neighbourhood of $\cBc_{\Gamma'}$, through which $\ff_\Gamma$ factors. 
	For chains ending in leaves, there is a unique choice of $\Gamma'$.
	For interior chains of unstable vertices, there are two choices of graphs $\Gamma'$ and $\Gamma''$, which only differ in their degree distribution and which both yield the same graph $\wh\Gamma$ after contracting a single boundary edge. 
	Thus we have a homotopy commutative square of open immersions
	\begin{equation*}\begin{tikzcd}
			\cBc_\Gamma \arrow[r,hook,""] \arrow[d,hook,""]&\cBc_{\Gamma'}\arrow[d,hook,""]\\
			\cBc_{\Gamma''} \arrow[r,hook,""] & \cBc_{\wh\Gamma} \end{tikzcd} \end{equation*}
	which covers
	\begin{equation*}\begin{tikzcd}
			\cMc_{\Gamma^{\stb}} \arrow[r,hook,""] \arrow[d,hook,""]&	 \cMc_{{\Gamma'}^{\stb}}\arrow[d,hook,""]\\
			\cMc_{{\Gamma''}^{\stb}} \arrow[r,hook,""] &	\cMc_{\Gamma^{\stb}} \end{tikzcd} 
	\end{equation*}
	of open embeddings, where the forgetful maps of the lower triangle have the same fibres. Thus, a relative orientation of $\ff_{\,\wh\Gamma}$ induces canonically a relative orientation of $\ff_\Gamma$. Proceeding in this fashion, we end up with a graph $\wt\Gamma$ obtained by contractions from $\Gamma$ and which is stable, whence $\ff_{\,\wt\Gamma}$ has a canonical relative orientation, inducing therefore one on $\ff_{\,\Gamma}$.
	
	If $\mfa_\Gamma$ is not unstable, it is a disc with one outgoing marked point and at most one incoming marked point. In this case, we can use the same argument as in the previous case to reduce the problem to orienting $\cK_{\mfa_\Gamma}$, i.e., choosing a trivialisation of 
	\[\orl(D)\otimes \bigotimes_{i}\orl(T_{z_i}\del \bD)\orl(\Aut(\bD))\dul.\]  
	We use the orientation of $\orl(D)$ induced by the relative spin structure; see \cite[\S B]{HH25} or \cite{CZ24}. The remaining parts are canonically oriented.
\end{construction}

\noindent
These will be the correspondences used for the operations associated to stable simplices. For the operations associated to unstable simplices, note that by the nondegeneracy assumption in Definition~\ref{de:curved-dmft}, we only have to specify a correspondence for the unique $0$-dimensional simplex. These correspondences are described in the following definition.

\begin{definition}\label{de:unstable-operations}
	The \emph{unstable operation} $\scF_{\ov{\mathsf{a}}}$, associated to an unstable tuple $\ov{\mathsf{a}}$, is defined using the following correspondences. 
	\begin{enumerate}[label=U\arabic*),leftmargin=22pt,ref=U\arabic*]
		\item\label{identity-correspondence} (Identity) If $\ov{\mathsf{a}}$ represents spheres/discs with one incoming and one outgoing interior marked point,  the correspondences are 
		\begin{equation*}\label{} 
			X\;\xleftarrow{\ide_X} \;X\;\xra{\ide_X}\; X\qquad\text{and} \qquad 	L\;\xleftarrow{\ide_L} \;L\;\xra{\ide_L}\; L.
		\end{equation*}
		\item\label{forgetful-correspondence} (Co-unit) If $\ov{\mathsf{a}}$ represents spheres with one incoming marked point, respectively discs with only one incoming boundary marked point, the correspondences are
		\begin{equation*}\label{} 
			X\;\xleftarrow{\ide_X} \;X\;\to\; \pt,\qquad \text{respectively} \qquad 	L\;\xleftarrow{\ide_L} \;L\;\to\; \pt.
		\end{equation*}
		\item\label{unit-correspondence} (Unit) If $\ov{\mathsf{a}}$ corresponds to a sphere/disc with one outgoing interior/boundary marked point, the correspondences are 
		\begin{equation*}\label{} 
			\pt\;\xleftarrow{} \;X\;\xra{\ide_X}\; X\qquad \text{and} \qquad 	\pt\;\xleftarrow{} \;L\;\xra{\ide_L}\; L.
		\end{equation*}
		\item\label{pairing-correspondence} (Pairing) If $\ov{\mathsf{a}}$ represents a sphere with two incoming marked points or a disc with two incoming boundary marked points, then the correspondences are
		\begin{equation*}\label{} 
			X\times X\;\xleftarrow{\Delta_X} \;X\;\xra{}\; \pt\qquad \text{and} \qquad 		L\times L\;\xleftarrow{\Delta_L} \;L\;\xra{}\; \pt.
		\end{equation*}
		\item\label{copairing-correspondence} (Co-pairing) If $\ov{\mathsf{a}}$ represents a sphere with two outgoing marked points the correspondence is
		\begin{equation*}\label{} 
			\pt\;\xleftarrow{} \;(TX^{\oplus 2}_\epsilon,\pi_X^*TX^{\oplus 2},\Delta_{TX^{\oplus2} },\tau_X^{\oplus 2})\;\xra{\exp\times\exp}\; X\times X,
		\end{equation*}
		where $ \epsilon > 0$ is smaller than the injectivity radius of the chosen Riemannian metric on $X$, $\pi_X$ is the projection map and $\tau_X$ is a Thom form of $TX$ with support in the $\epsilon$-disc bundle.
		
		If $\ov{\mathsf{a}}$ represents a disc with two outgoing boundary marked points, the correspondence $\fC^0_{0,2}$ is given by
		\begin{equation*}
			\pt\;\xleftarrow{} (TL^{\oplus 2}_\epsilon,\pi_L^*TL^{\oplus 2},\Delta_{TL^{\oplus2}},\tau_L^{\oplus 2})\;\xra{\exp\times\exp}\; L\times L,
		\end{equation*}
		with $\epsilon,\pi_L$ and $\tau_L$ as in the closed case.
		\item\label{restrict-correspondence} (Restriction) if $\ov{\mathsf{a}}$ represents a disc with one incoming interior marked point, the associated correspondence is 
		\begin{equation*}\label{} X\;\xleftarrow{j}\; L \; \to \;\pt,\end{equation*}
		where $j$ is the inclusion of $L$.
		\item\label{corestrict-correspondence} (Inclusion) if $\ov{\mathsf{a}}$ represents a disc with one outgoing interior marked point, the associated correspondence is 
		\begin{equation*}\label{} \pt\;\leftarrow\; (TX_\epsilon|_L,\pi_X^*(TX|_L),\Delta,\tau_X|_{TX|_L}) \; \xra{\exp} \;X.\end{equation*}
		\item If the curve that $\ov{\mathsf{a}}$ represents has no marked point, then we associate the empty correspondence $\pt \leftarrow\emst\to \pt$ to it.
	\end{enumerate}
	We define $\scF_\sigma$ to be zero if $\sigma$ is a degenerate simplex in an unstable moduli space.
\end{definition}

\noindent
We can now turn to the correspondences required for the higher homotopies. The proof is a variation of the proof of Theorem~\ref{thm:cubical-cobordisms-enhanced}.

\begin{lemma}\label{lem:gkc-for-higher-homotopies} Let $\{\cKc_\Gamma\}_{\Gamma}$ be a system of cubical cobordisms as in Theorem~\ref{thm:cubical-cobordisms-enhanced}. Then, given any sequence $\lc{\Gamma} = (\Gamma_d,\dots,\Gamma_1)$ of stable map graphs and truly unstable corollas, there exists smooth global Kuranishi charts $\cKc_{\lc{\Gamma}}$ with corners with a submersion to $I^{d-1}\times \prod_{j=0}^{d-1}I^{E(\Gamma_{d-j})}$ with the following properties.
	\begin{enumerate}[label=\arabic*),leftmargin=20pt,ref=\arabic*]
		\item\label{higher-base} $\cKc_{\lc{\Gamma}}$ admits a smooth submersion to the base space 
		$\cBc_{\lc{\Gamma}} \;\coloneqq\; \cBc_{\Gamma_d}\times \dots \times \cBc_{\Gamma_1}$.
		\item\label{higher-ev} $\cKc_{\lc{\Gamma}}$ admits smooth invariant submersions $\eva^+_{\lc\Gamma}\cl\cTc_{\lc{\Gamma}}\to Y^+_{\Gamma_d}$ and $\eva^-_{\lc\Gamma}\cl\cTc_{\lc{\Gamma}}\to Y^-_{\Gamma_1}$.
		\item\label{higher-side-one} $\cKc_{\lc{\Gamma}}|_{\{s_i = 0\}} \simeq \cKc_{\Gamma_d,\dots,\Gamma_{i+1}\g \Gamma_i,\dots,\Gamma_1}$ as correspondences over $Y^-_{\Gamma_1}$ and $Y^+_{\Gamma_d}$.
		\item\label{higher-side-zero} $\cKc_{\lc{\Gamma}}|_{\{s_i = 1\}} \simeq \cKc_{ \Gamma_d,\dots,\Gamma_{i+1}}\times_{Y^+_{\Gamma_i}}\cKc_{\Gamma_i,\dots,\Gamma_{1}}$ as (unoriented ) correspondences over $Y^-_{\Gamma_1}$ and $Y^+_{\Gamma_d}$.
		\item\label{higher-edge-zero} $\cKc_{\lc{\Gamma}}|_{\{t_e = 0\}} \simeq \cKc_{ \Gamma_d,\dots,(\Gamma_i)_e,\dots,\Gamma_1}$ for any edge $e \in E(\Gamma_i)$.
		\item\label{higher-edge-one} $\cKc_{\lc{\Gamma}}|_{\{t_e = 1\}} \simeq \cKc_{ \Gamma_d,\dots,\Gamma_i'',\Gamma'_i,\dots,\Gamma_1}|_{\{s_{i}=0\}}$ for any edge $e \in E(\Gamma_i)$ separating $\Gamma_i$ into $\Gamma'_i$ and $\Gamma''_i$ (where we add identity corollas as necessary).
		\item\label{higher-case-one} $\cKc_{\lc{\Gamma}} = \cKc_{\Gamma_1}$ if $d = 1$.
		\item\label{trivial-if-stable} If $\Gamma_d,\dots,\Gamma_1$ are all stable, then $\cKc_{\lc{\Gamma}} = I^{d-1}\times\cKc_{\Gamma_{d}}\times_{Y_{d-1}^+} \dots\times_{Y_1^+}\cKc_{\Gamma_1}$.
	\end{enumerate}
	Here we use the coordinates $s_1,\dots,s_{d-1}$ on $I^{d-1}$.
\end{lemma}

\noindent
Note that~\eqref{higher-base} implies that we have a smooth map $\stb \cl \cKc_{\lc{\Gamma}}\to \prod_{i =1}^{d}\Mbar_{\Gamma^{\stb}_i}$, where we take $\Mbar_{\Gamma^{\stb}_i}$ to be a point if $\Gamma_i$ does not have a well-defined stabilisation.

\begin{proof} In this proof a graph will mean a stable map graph or a disjoint union of truly unstable corollas. The upshot of the construction is that different stable graphs can have different \emph{stabilisation patterns}, that is, which evaluation maps are stabilised can change when composing two graphs. The easiest example is the coproduct in energy zero: the two outgoing marked points are stabilised; however, when forgetting one (and recovering the identity) the remaining marked point is no longer stabilised in our definition. This mismatch has to be reconciled by a cobordism. We will construct them here compatibly.
	
	We first have to define some notation.
	To a graph $\Gamma$ and a marked point $j$ of $\Gamma$, we associate the number 
	\begin{equation}\label{stabilisation-pattern}\delta_{\Gamma}(j)\,\coloneqq\,
		\begin{cases}
			1 \quad &\text{if }\eva_{\Gamma,j} \text{ is stabilised}\\
			0 \quad &\text{otherwise}.
	\end{cases}\end{equation}
	Here an evaluation map is stabilised according to Theorem~\ref{thm:cubical-cobordisms-enhanced} or Definition~\ref{de:unstable-operations}.
	Then, given a sequence $\lc{\Gamma}$ as in the statement with outgoing marked points labelled by $y^d_c \sqcup y^d_o$ and incoming marked points labelled by $y_c^0\sqcup y_o^0$, we define the functions $\chi^\pm_j\cl I^{d-1}\to I$ by 
	\begin{equation}\label{} 
		\chi_{\lc\Gamma,j}^-(s) \;=\; \sum_{k=1}^{d}\delta_{\Gamma_k\g \dots\g \Gamma_1}(j)\,s_k\prod_{i=1}^{k-1}(1-s_i)
	\end{equation}
	for $j \in y_c^0\sqcup y_o^0$, where we set $s_d \equiv 0$, and
	\begin{equation}\label{} 
		\chi_{\lc\Gamma,j}^+(s) \;=\; \sum_{k=0}^{d-1}\delta_{\Gamma_d\g \dots\g \Gamma_{d-k}}(j)\,s_{d-k-1}\prod_{i=1}^{k}(1-s_{d-i})
	\end{equation}
	for $j \in y^d_c \sqcup y^d_o$ with $s_0 \equiv 1$. The key property that the maps $\chi^\pm_{\lc\Gamma,j}$ satisfy is that 
	\begin{equation}\label{} 
		\chi^-_{\lc\Gamma,j}(s) \; = \; \delta_{\Gamma_k\g \dots\g \Gamma_1}(j)\qquad \qquad \text{if }s_1,\dots,s_{k-1} = 0,\, s_k = 1,
	\end{equation}
	and
	\begin{equation}\label{} 
		\chi^+_{\lc\Gamma,j}(s) \; = \; \delta_{\Gamma_d\g \dots\g \Gamma_{d-k}}(j)\qquad \qquad \text{if }s_{d-1},\dots,s_{d-k} = 0,\, s_{d-k-1} = 1.
	\end{equation}
	More generally, letting $\chi^\pm_{\lc\Gamma}$ be the product over all respective marked points, we have for $s_k = 1$, that 
	\begin{equation}\label{eq:splitting-property} 
		\chi^-_{\lc{\Gamma}}(s) = \chi^-_{\Gamma_k,\dots,\Gamma_1}(s_{k-1},\dots,s_1)\qquad\text{and}\qquad\chi^+_{\lc{\Gamma}}(s) = \chi^+_{(\Gamma_d,\dots,\Gamma_k)}(s_{d-1},\dots,s_{k}).
	\end{equation}
	We now define for $1 \leq i \leq d$ the functions 
	\begin{equation}\label{eq:intermediate-parametrisations} 
		\chi^{+,i}_{\lc\Gamma}(s) \coloneqq \chi^+_{\Gamma_i,\dots,\Gamma_1}(s_{i-1},\dots,s_1)\qquad \qquad \chi^{-,i}_{\lc\Gamma}(s) \coloneqq \chi^-_{\Gamma_d,\dots,\Gamma_{i}}(s_{d-1},\dots,s_i).\end{equation}
	Given this, we define the thickening 
	\[\cTc_{\lc\Gamma}\sub I^{d-1}\times \prod_{i=0}^{d-1}\wh{\cKc_{\Gamma_{d-i}}}\] 
	to be the parametrised fibre product of tuples $(s,(u_d,\dots,u_1))$ where the global Kuranishi chart $\wh{\cKc_{\Gamma_{d-i}}}$ is the maximal stabilisation possible (although we only stabilise its evaluation maps according to the rules of Theorem~\ref{thm:cubical-cobordisms-enhanced}, respectively Definition~\ref{de:unstable-operations}) and for any edge $e$ between $\Gamma_i$ and $\Gamma_{i+1}$ we have 
	\begin{equation}\label{} \exp_{u_i(z_{i,e})}\lbr{\chi^{+,i}_{\lc\Gamma,e}(s)\xi_{i,e}} \;=\;  \exp_{u_{i+1}(z_{i+1,e})}\lbr{\chi^{-,i+1}_{\lc\Gamma,e}(s)\xi_{i+1,e}},\end{equation}
	where $z_{i,e}$ and $z_{i+1,e}$ are the marked points on the respective graph that are clutched to $e$ and $\xi_{j,e}$ is the tangent vector associated to that marked point, coming from the stabilisation of the respective global Kuranishi chart. We let $\cEc_{\lc\Gamma}$ and $\obs_{\lc\Gamma}$ be the pullback of $\bigoplus\limits_{i= 0}^{d-1}\cEc_{\Gamma_{d-i}}$ and the product of the obstruction sections, and we let the covering group be $G_{\lc\Gamma} = \prod_{i=0}^{d-1}G_{d-1}$. Since the outgoing evaluation map is always a submersion, the resulting global Kuranishi chart $\cKc_{\lc\Gamma}$ is a smooth global Kuranishi chart and it satisfies the first two as well as the penultimate property of the assertion. The remaining two properties follow from the fact that the functions $\chi_{\lc\Gamma}^{\pm,i}$ have the property~\eqref{eq:splitting-property} as well as the fact that $\chi^\pm_{\lc{\Gamma},j} \equiv \delta_{\Gamma_d\g \dots \g \Gamma_1}(j)$ if $\delta_{\Gamma_d\g \dots \g \Gamma_k}$ does not depend on $k = 1,\dots,d$.
\end{proof}
\subsection{Construction}\label{subsec:construction} We define $\scF_L$ on objects to be the completed tensor product
\begin{equation}\label{eq:dmft-on-objects} 
	\scF_L^*(y) \;\coloneqq\; \Omega^*(X^{y_c}\times L^{y_o};\Lambda)\; =\; \Omega^*(X^{y_c}\times L^{y_o})\,\wh{\otimes}\, \Lambda
\end{equation}
of differential forms with the Novikov ring. The differential is given by 
\begin{equation}\label{eq:differential} 
	d_y(\alpha T^\beta) \;\coloneqq\; (-1)^{|\alpha|}d_{dR}(\alpha)T^\beta
\end{equation}
for $\alpha \in \Omega^*(X^{y_c}\times L^{y_o};\Lambda)$. 
To define $\scF$ on morphisms we will use the correspondences constructed in the previous subsection. However, they are only relatively oriented over $\cMc_{\ov\Gamma}$ and do not yet take any information about chains on $\Mbar_{oc}$ into account. The point of the next lemma is that we can modify these global Kuranishi charts using morphisms of $\cdmc_n$ so that we obtain oriented correspondences. The definition of these operations is a considerable generalisation of the $\fq$-operations of \cite{ST16} and requires the resolution of singularities discussed in \S\ref{subsec:singular-integration}. Throughout the remainder of the section, we write $\beta\in H_2$ if we do not want to specify whether $\beta$ is a relative homology class in $H_2(X,L;\bZ)$ or lies in $H_2(X;\bZ)$.

\begin{lemma}\label{lem:evaluation-pushforward} To a stable simplex $\sigma\cl \Delta^{\norm{\sigma}}\to \cMc_{\cc{\Gamma}_\sigma}$ and $\beta\in H_2$ we can associate a well-defined map
	\begin{gather}\label{eq:generalised-q-operation} \fq^\beta_{\sigma} \cl \Omega^*(X^{y_c}\times L^{y_o})\to \Omega^*(X^{y'_c}\times L^{y'_o})\\
		\notag \alpha\mapsto (-1)^{\delta_{0,h}\,w_\fs(\beta)}
		\s{\substack{\beta_\Gamma = \beta\\\Gamma^{\stb} = \ov\Gamma_\sigma}} {c_\Gamma\,(\eva^+_{\Gamma,\sigma})_*\lbr{(\eva^-_{\Gamma,\sigma})^*\alpha \wedge p_\sigma^*\obs_\Gamma^*\eta_{\Gamma}}},
		\end{gather}
	of degree 
	\begin{equation}
		\label{eq:degree-q-operation} \deg(\fq^\beta_{\sigma}) = -\mu_L(\beta)+|\sigma|
		\end{equation} 
	where the sum is over all marked prestable graphs of total degree and type $\mfa$ and the factor is given by
	\begin{equation}\label{eq:q-coefficient}
		c_\Gamma \coloneqq \frac{|\Aut^+(\Gamma^{\stb})|}{|\Aut^+(\Gamma^{ps})|},
	\end{equation}
	where $\Aut^+(\Gamma)$ is the automorphism group of $\Gamma$ preserving the distribution into incoming and outgoing marked points. We call $\fq^\beta_\sigma$ the \emph{$\fq$-operation associated to $\sigma$ of energy $\beta$}.
\end{lemma}

\begin{proof} 
	It will be useful to consider the pull-push operation without the sign and separately for each graph $\Gamma$, so we write
	\begin{equation}\label{eq:q-operation-graph} 
		\fq_{\Gamma,\sigma}(\alpha) \coloneqq (\eva^+_{\Gamma,\sigma})_*\lbr{(\eva^-_{\Gamma,\sigma})^*\alpha \wedge p_\sigma^*\obs^*\eta_{\Gamma}}
	\end{equation}
	for the (not yet defined) part of the sum. For a stable map graph $\Gamma$ of total degree $\beta$ and with stabilisation $\Gamma^{\stb} = \ov{\Gamma}_\sigma$, define
	\begin{equation}\label{eq:thickening-over-simplex}
		 \cT_{\Gamma,\sigma} \,\coloneqq\, \cTc_\Gamma \times_{\cMc_{\cc{\Gamma}_\sigma}}\Delta^{\norm{\sigma}}, 
	\end{equation}
	where the map
	\begin{equation}\label{eq:forgetful-map-to-spikes}
		 \cTc_\Gamma\to  \cMc_{\ov\Gamma}
	\end{equation}
	is given by the stabilisation map $\cTc_\Gamma\to \Mbar_{\ov{\Gamma}_\sigma}$ as well as the composition 
	$\cTc_\Gamma\to I^{E(\Gamma)}\to I^{E(\ov{\Gamma}_\sigma)}$.
	 Then, $\cT_{\Gamma,\sigma}$ is canonically a smooth orbifold over 
	\[ \cB_{\Gamma,\sigma} \coloneqq \cBc_{\Gamma} \times_{\cMc_{\cc{\Gamma}_\sigma}}\Delta^{\norm{\sigma}}\]
	 with a smooth structure away from the singular points, which lie in codimension $2$. Since any unstable vertices of $\Gamma$ are discs with only boundary marked points, there exists a canonical isomorphism 
	 
 	\begin{align}\label{eq:orienting-simplex-chart} 
 	\notag\orl(\cK_{\Gamma,\sigma}) \ &\cong \ \orl(\cKc_{\Gamma})\orl(\cMc_{\ov\Gamma_\sigma})\dul\orl(\Delta^{\norm{\sigma}}) 
 	\\\notag&\cong \ \orl(\delbar_\Gamma) \orl(\norm{\sigma})
 	\\& \cong \ \orl(Y_\sigma^+)\orl(\det(\sigma))\orl(\mu_L(\beta_\Gamma)) \orl(|\sigma|)
 	\end{align}
 	 over the locus of regular points, where $Y_\sigma^+ = X^{y'_c}\times L^{y_c}$. Note that $\mu_L(\beta_\Gamma)$ is even, so its position in the tensor product of orientation lines is irrelevant. 
 	By the definition of local system coefficient, $\sigma$ carries the data of a trivialisation of $\orl(\det(\sigma))$. Thus, the isomorphism~\eqref{eq:orienting-simplex-chart} yields a trivialisation of the orientation bundle of the smooth locus of $\cK_{\Gamma,\sigma}$.\par
	 Let $p_\sigma \cl \cT_{\Gamma,\sigma}\to \cTc_{\Gamma}$ be the canonical projection. Set 
	\begin{equation}\label{eq:chart-restricted-to-simplex} 
	\cK_{\Gamma,\sigma} \,\coloneqq\, \lbr{\cT_{\Gamma,\sigma},p_\sigma^*\cEc_{\Gamma},p_\sigma^*\obs_{\Gamma}}.
	\end{equation}	
	Since the evaluations maps $\eva^+_\Gamma$ are jointly submersive with $\ff$, the pullbacks $\eva_{\Gamma,\sigma}^\pm$ of the evaluation maps $\eva^\pm_{\Gamma}$ are still smooth and $\eva^+_{\Gamma,\sigma}$ is a relative submersion. 
	By Theorem~\ref{thm:pushforward along chains} with $\cW = \Delta^{\norm{\sigma}}$, $\cV = \cMc_{\ov\Gamma_\sigma}$ and $\cX = \cTc_\Gamma$ and the orientation determined by the order of the factors in \eqref{eq:thickening-over-simplex}, the map $\fq_{\Gamma,\sigma}$, defined by~\eqref{eq:q-operation-graph}, is well-defined. Hence, so is $\fq^\beta _\sigma$.
	 
	Finally, to compute the degree of $\fq_{\Gamma,\sigma}$, recall that, due to the definition of $\text{det}^{\otimes n}$, the map $\sigma$ is defined on $\Delta^{\norm{\sigma}}$, where 
	\begin{equation}\label{eq:shift-degree} \norm{\sigma} = |\sigma| -n\chi + 2n|y'_c| + n|y_o'|\end{equation}
	with $\chi$ being the Euler characteristic of the curves whose moduli $\sigma$ maps to. Thus,
	\begin{align*}
		|\fq_{\Gamma,\sigma}(\alpha)| &= |\alpha| -\lbr{\vdim(\Mbar_{\mathsf{a}}(X,L)\times_{\Mbar_{oc}(y,y')}\Delta^{\norm{\sigma}})-2n|y_c'|-n|y'_o|}\\& = |\alpha| -\lbr{\norm{\sigma}+n\chi+\mu_L(\beta_{\mathsf{a}})-2n|y_c'|-n|y'_o|}\\& = |\alpha|-|\sigma|-\mu_L(\beta_{\mathsf{a}})
		\end{align*}
	as claimed.
\end{proof}

\begin{remark}\label{} Note that the definition of \eqref{eq:generalised-q-operation} is also well-defined if the domain of $\sigma$ is an analytic map on an analytic manifold with corners. \end{remark}

\noindent
The operations of Lemma~\ref{lem:evaluation-pushforward} are not quite sufficient as they are not defined on elements of the morphism spaces of $\dmc_n$, cf. Definition~\ref{de:chain-model}. Thus, we have to show that these operations are independent under pushforwards by contractions in a suitable sense. The next lemma is essential to deal with the fact that the stabilisation map contracts unstable components, whence certain canonical squares are only Cartesian away from strata of positive codimension.

\begin{lemma}\label{lem:multiplication-up-to-codim} Let $\psi \cl X'\to X$ be a proper smooth map between smooth \'etale Lie groupoids with corners of the same dimension. If $\deg(\psi) = 1$ and $h \cl X\to Z$ is a submersion to a smooth manifold so that $h' \coloneqq h\psi$ is a submersion as well, then 
	\begin{equation*}\label{} h'_*\psi^* = h_* 
	\end{equation*} 
	as maps $\Omega^*_c(X)\to \Omega^{*-k}(Z)$.
\end{lemma}

\begin{proof} 
	By Sard's theorem for orbifolds, \cite[Theorem~4.1]{BB12}, the restriction $\psi_z: h'^{-1}(z) \rightarrow h^{-1}(z)$ is a degree-$1$ map for an open dense set $U\sub Z$. By the degree formula for orbifolds, \cite[Theorem~1.1]{BZ24}, we have $(h_*\alpha)_z = (h'_*\psi^*\alpha)_z$ for $z\in U$ and $\alpha \in \Omega_c^*(X)$. By continuity, the two differential forms agree everywhere.
\end{proof}

\begin{lemma}\label{lem:independence-of-spike} Suppose $\mathsf{a} \in \cA$ is stable. Then, the map 
	\begin{equation*}\label{} 
		\fq^\beta\, \cl\, \Omega^*(X^{y_c}\times L^{y_0})\times  \bigoplus\limits_{\ov\Gamma}C^{an}_*(\cMc_{\ov\Gamma})\;\to\;  \Omega^*(X^{y'_c}\times L^{y'_0}) \;:\; (\sigma,\alpha)\mapsto \fq^{\mfa}_{\sigma}(\alpha)
	\end{equation*}
of Lemma~\ref{lem:evaluation-pushforward} descends to a well-defined map 
\begin{equation}\label{} 
	\scF^\beta \cl \Omega^*(X^{y_c}\times L^{y_0}) \times \dmc_n^*(y,y')\to  \Omega^*(X^{y'_c}\times L^{y'_0})
\end{equation}
of degree $-\mu_L(\beta_{\mathsf{a}})$.
\end{lemma}

\begin{proof} 
	It suffices to treat the case where $\mathsf{a}$ consists of a single element. The general case is analogous. Thus the claim reduces to showing that if $\sigma\cl \Delta^m \to \cMc_{\cc{\Gamma}_\sigma}$ has image contained in $\cMc_{\cc{\Gamma}_\sigma}|_{t_e = 0}$, then $q_{\Gamma,\sigma} = 0$ if in the stabilisation of the underlying prestable graph $\Gamma^{ps}$ several edges are combined to form $e$.  But in this case the map $\cTc_\Gamma\to I^{E(\Gamma^{\stb})}$ has a corner singularity (Definition \ref{de:corner-singularity}) over $\{t_e = 0\}$, whence the associated operation vanishes by construction. Therefore, the claim follows up to sign from \cite[Lemma~4.22]{HH25}, combined with Lemma~\ref{lem:pushforward-if-group-acts}. 
	To see that the sign is $1$, we observe that the outer rectangle 
	\begin{center}\begin{tikzcd}
			\orl(\cK_{\Gamma,\sigma}) \arrow[r,""] \arrow[d,""]&\orl(\cKc_{\Gamma})\orl(\cMc_{\ov\Gamma_\sigma})\dul\orl(\Delta^{\norm{\sigma}}) \arrow[d,""]\arrow[r]& \orl(\delbar_\Gamma)\orl(\Delta^{\norm{\sigma}})\arrow[d,"="]\\ 
			\orl\lbr{\cKc_\Gamma|_{\{t_e = 0\}}\times_{\Mbar_{\ov\Gamma_\sigma}|_{\{t_e = 0\}}}\Delta^{\norm{\sigma}}} \arrow[r,""] & \orl(\cKc_{\Gamma}|_{\{t_e = 0\}})\orl(\cMc_{\ov\Gamma_\sigma}|_{\{t_e = 0\}})\dul\orl(\Delta^{\norm{\sigma}})\arrow[r]& \orl(\delbar_\Gamma)\orl(\Delta^{\norm{\sigma}}) \end{tikzcd} \end{center}
	commutes because both squares commute up to sign $(-1)^{\ind(\delbar_\Gamma)}$. It now follows from a straightforward verification that
	\begin{center}\begin{tikzcd}
			\orl(\cK_{\Gamma,\sigma}) \arrow[r,""] \arrow[d,""]&\orl(Y_\sigma^+)\orl(\det(\sigma))\orl(|\sigma|) \arrow[d,"\ide\otimes \eqref{eq:iso-of-det} \otimes\ide"]\\ 
			\orl(\cK_{\psi(\Gamma),\psi_*\sigma}) \arrow[r," "] & \orl(Y_{\psi_*\sigma}^+)\orl(\det(\psi_*\sigma))\orl(|\psi_*\sigma|) \end{tikzcd} \end{center}
	commutes, yielding the desired result.
\end{proof}

\begin{lemma}\label{lem:pushforward-if-group-acts} Suppose $X$ is a Lie groupoid admitting a rel--$C^\infty$ action by a finite group $G$. If $f^\pm \cl X\to Y^\pm$ are two rel--$C^\infty$ $G$-invariant maps so that $f^+$ is a submersions, let $\lc{f}^\pm \cl [X/G] = [G\times X_1\rightrightarrows X_0]\to Y^\pm$ be the induced maps. Then, 
	\begin{equation}\label{}\lc{f}^+_*((\lc{f}^-)^*\alpha) = \frac{1}{|G|}\,f^+_*({f^-}^*\alpha) \end{equation}
	for any $\alpha\in \Omega^*(Y^-)$.
\end{lemma}

\begin{proof} Note first that ${f^-}^*\alpha$ is $G$-invariant for any $\alpha\in \Omega^*(Y^-)$. Thus the only difference between the two pushforward is given by the choice of partition of unity, cf. \cite[Example~4.16]{ST23b}. Given a partition of unity $\rho$ for $X$, we may first replace it by its average over $G$. Due to the fact that $G$ intertwines source and target map this yields a well-defined $G$-invariant partition of unity $\varrho$ on $X$. Then, set $\lc{\rho} := \frac{1}{|G|}\varrho$. This is a well-defined partition of unity because $\lc{s}^*\varrho$ is $G$-invariant. Hence, $\lc{t}_*\lc{s}^*\varrho = |G|t_*s^*\varrho = |G|$. 
\end{proof}

\noindent
For the unstable operations, we write $\ov{\mathsf{a}}$ for the unique $0$-dimensional simplex mapping to the coarse moduli space $\Mbar_{\ov{\mathsf{a}}}$ and let 
\begin{equation}\label{} 
	 \scF_{\ov{\mathsf{a}}}^0 \coloneqq \fq_{\fC_{\ov{\mathsf{a}}}}\cl \scF(y)\to \scF(y')
\end{equation}
be the operation associated to the respective correspondence in Definition~\ref{de:unstable-operations}. We set $\scF_{\ov{\mathsf{a}}}^\beta$ for $\beta$ with $\omega(\beta) > 0$.
We then extend $\scF$ linearly to chains, that is, for $\sigma = \s{i}{t_i\,\sigma_i}$ and $\beta\in H_2(X,L;\bZ)$ we set
\begin{equation}\label{eq:linear-extension} 
	\scF_{\sigma}^\beta \coloneqq \s{i}{t_i\, \scF^\beta_{\sigma_i}}.
\end{equation}

\begin{definition}\label{de:differential-and-curvature}
	For $\beta \in H_2(X,L;\bZ)$ with $\omega(\beta) > 0$, let
	\begin{equation*}\label{}
		\scF^\beta_{D_1} \coloneqq \s{\Gamma}{\fq_{\Gamma}} \cl \Omega^*(L)\to \Omega^*(L) 
	\end{equation*}
	where we are summing over stable map graphs of type $(\beta,0,1)$ with one incoming and one outgoing boundary marked point and $\fq_{\Gamma}$ is the operation associated to the correspondence $\cKc_\Gamma$ of Theorem~\ref{thm:cubical-cobordisms-enhanced} with the canonical orientation induced by the relative spin structure on $L$. Similarly, \[\scF^\beta_{D_0}\coloneqq \s{\Gamma}{\fq_\Gamma}\in \Omega^*(L)\] 
	is given by a sum over stable map graphs of type $(\beta,0,1)$ with one outgoing boundary marked point and $\fq_\Gamma$ being the operation associated to the correspondence $\cKc_\Gamma$.
\end{definition}

\begin{remark}\label{}These operations do not require resolutions of singularities since we need not take the fibre product with a chain.
\end{remark}

\noindent
We now define the higher homotopies of the $A_\infty$ functor $\scF_L$. Recall that these take the form of linear maps 
\begin{equation}\label{eq:higher-homotopies} 
	\scF_{\lc{\sigma}}^d\cl \scF(y^0)\,\to\, \scF(y^d).
\end{equation}
for composable sequences $\lc\sigma = (\sigma_d,\dots,\sigma_1)$ in $\cdmc_n$.
As in the case of the previous operations, the operations~\eqref{eq:higher-homotopies} are induced by correspondences. To make the notation uniform for stable and unstable operations, we define \emph{truly unstable} corollas, that is, graphs without edges whose vertex has energy. In this case we let $\cK_{\ov{\mathsf{a}}}$ be the corresponding global Kuranishi chart correspondence of Definition~\ref{de:unstable-operations}. Given a disjoint union $T$ of truly unstable corollas and a stable (map) graph $\Gamma$ so that the incoming/outgoing marked points of $T$ correspond to the outgoing/incoming marked points of $\Gamma$, we define the composition $T\g \Gamma$ (respectively $\Gamma\g T$) to be the stable (map) graph obtained by applying the respective operation associated to the truly unstable vertices of $T$ described in Definition~\ref{de:open-closed-curve-category}, e.g., forgetting a marked point our clutching two marked points in $\Gamma$ to form an edge. Given this, it is clear what a stable (map) graph with truly unstable vertices is. In contrast to a composition of stable (map) graphs the composition $T\g \Gamma$ does not always create new edges. If $T$ and $\Gamma = T'$ is a disjoint union of truly unstable corollas, then we can similarly define $T\g T'$. Note that this is again a truly unstable corolla. To make the notation more streamlined, we define for a graph $\Gamma$ with $k^\pm$ interior and $\ell^\pm$ boundary outgoing/incoming marked points the manifolds 
\[Y^\pm_\Gamma \coloneqq X^{k^\pm}\times L^{\ell^\pm}\]
and similarly if the marked points are labelled by finite ordered sets.

\begin{lemma}\label{lem:construction-of-higher-homotopies} For any $d \ge 1$ and any objects $y^0,\dots,y^d$ of $\cdmc_n$, there exists a well-defined linear map  
	\begin{equation*}\label{} 
		\scF^d = \s{\beta}{\scF^{d,\beta}Q^\beta} \;\cl\; \scF^*(y^0)\otimes \bigotimes\limits_{i=1}^d{\cdmc_n^*(y^{i-1},y^i)}\; \longrightarrow\; \scF(y^d) 
	\end{equation*}
	of degree $0$. 
\end{lemma}

\begin{proof} Suppose $\sigma_i\cl \Delta^{\norm{\sigma_i}}\to \cMc_{\ov\Gamma_i}$ is an analytic simplex for each $1 \leq i \leq d$, where the incoming marked points of $\ov\Gamma_i$ are labelled by $y^{i-1}$ and its outgoing marked points by $y^i$. Note that $\ov\Gamma_i$ may be a collection of truly unstable corolla, in which case we can assume $\norm{\sigma_i} = 0$. Fix also $\beta \in H_2(X,L;\bZ)$, respectively $H_2(X;\bZ)$ if all simplices map to moduli spaces of closed curves. Let $\Gamma_d,\dots,\Gamma_1$ be stable map graphs, respectively disjoint unions of truly unstable corolla if $\ov\Gamma_i$ is such, so that $\s{i}{\beta_{\Gamma_i}} = \beta$. Then, $\lc{\sigma}$ induces an analytic map $\Delta^{\norm{\lc{\sigma}}} := \p{i \le d}{\Delta^{m_{d-i}}}\to  \p{i \le d}{\cMc_{\ov\Gamma_i}}$ and we can form the fibre product 
	\begin{equation*}\label{} 
		\cK_{\lc\Gamma,\lc\sigma}\;\coloneqq\; \cKc_{\lc\Gamma}\times_{\p{i \le d}{\cMc_{\ov\Gamma_i}}} \Delta^{\norm{\lc{\sigma}}}.
	\end{equation*}
	Define the orientation line
	\begin{equation}\label{eq:simplex-line} 
		\orl(\sigma) \coloneqq \orl(\det(\sigma))\orl(|\sigma|) = \orl(\Delta^{\norm{\sigma}})\orl(\det(\sigma))\orl(-n\chi+n|y'_o|+2n|y_c'|).
	\end{equation}
	Over the smooth locus of this fibre product, the fibre product $\cK_{\lc\Gamma,\lc\sigma}$ has a canonical orientation, induced by the isomorphism
	\begin{equation}\label{eq:orientation-higher} 
		\orl(\cK_{\lc\Gamma,\lc\sigma}) \;\cong\; \orl(Y_{\sigma_{d}}^+)\orl(\sigma_d)\orl(I^{\{d-1\}})\orl(\sigma_{d-1})\cdots\orl(I^{\{1\}})\orl(\sigma_1)
	\end{equation}
	and the given trivialisations that are part of the data of the simplices (i.e., coming from the local coefficients). Note that this is not the standard fibre product orientation, even in the case where $\cK_{\lc\Gamma,\lc\sigma}$ is an (unparametrised) fibre product.
	
	Using the resolution of singularities, Theorem~\ref{thm:pushforward along chains}, on $\cK_{\lc\Gamma,\lc\sigma}$ as in the proof of Lemma~\ref{lem:evaluation-pushforward}, we obtain a well-defined linear map 
	\begin{equation}\label{} 
		\fq_{\lc\Gamma,\lc\sigma}\;\cl\; \Omega^*(X^{y_c^0}\times L^{y^0_o};\Lambda)\;\to\; \Omega^*(X^{y_c^d}\times L^{y^d_o};\Lambda)
	\end{equation}
	which (up to replacing $\cK_{\lc\Gamma,\lc\sigma}$ by a non-singular space) is given by 
	\[\alpha\;\mapsto\; (\eva_{\lc\Gamma,\lc\sigma}^+)_*((\eva_{\lc\Gamma,\lc\sigma}^+)*\alpha\wedge p_{\lc\sigma}^*\obs_{\lc\Gamma}^*\eta_{\lc\Gamma})\]
	where $\eta_{\lc\Gamma}$ is the pullback of the product of Thom forms. Then,
	\begin{align*}\label{} 
		|\fq_{\lc\Gamma,\lc\sigma}(\alpha)| = |\alpha|- \vdim(\cK_{\lc\Gamma,\lc\sigma})+2n|y_c^d|+n|y_o^d| = |\alpha|+ |\lc\sigma|+1-d,
	\end{align*} 
	recalling that $|\sigma_i| = -m -2n\chi + 2n|y^i_c| + n|y^i_o|$. Thus, 
	\begin{equation*}\label{} 
		\scF^{d,\beta}_{\lc{\sigma}}\;\coloneqq\;\fq^\beta_{\lc\sigma} \;\coloneqq \;(-1)^{\epsilon}\s{\lc\Gamma}{c_{\lc\Gamma}\fq_{\lc\Gamma,\lc\sigma}},
	\end{equation*} 
	is a linear maps of degree $|\lc\sigma|$, where 
	\begin{equation}\label{eq:coefficient-higher-homotopies} 
		\epsilon \coloneqq \s{i}{\delta_{0,h_i}w_\fs(\beta_i)}\qquad \qquad c_{\lc\Gamma}\coloneqq \p{i}{c_{\Gamma_i}},
	\end{equation} 
	with $c_{\Gamma_i}= 1$ if $\Gamma_i$ is truly unstable and given by the formula in~\eqref{eq:q-coefficient} otherwise. Clearly, the map $\scF^{d,\beta}_{\lc\sigma}$ is linear. We extend it linearly to chains on $\cMc_{\ov\Gamma_d}\times \dots\times\cMc_{\ov\Gamma_1}$. By the same argument as in Lemma~\ref{lem:independence-of-spike}, applied to each $\sigma_i$ separately, it descends to the claimed linear map.
\end{proof}

\noindent
We show in \S\ref{subsec:warped-equations-proof} that $\scF$ satisfies the $A_\infty$ functor equations \eqref{eq:warped-relations} and that it is semi-strict and operadic. Here, we show first that $\scF$ interacts well with the symmetric monoidal structures, is unital and has Property~\ref{assum:weakly-curved-to-uncurved}, which is important for being able to deform it to an uncurved open-closed DMFT.

\begin{lemma}\label{lem:unital-functor} $\scF$ is unital.
\end{lemma}

\begin{proof}
	By definition, $\scF_{\idsimp^{\times y}} = \ide_{\scF(y)}$. Suppose then $\lc\sigma =(\sigma_d,\dots,\sigma_1)$ is a composable sequence in $\cdmc_n$ with $d > 1$ and $\sigma_j = \idsimp$ for some $j$.
	Given graphs $\Gamma_d,\dots,\Gamma_1$ lying over the simplices, the gluing $\Gamma_d\g \dots\g \Gamma_j$ has the same stabilisation pattern as $\Gamma_d\g \dots \g \Gamma_{j+1}$ and $\Gamma_j \g \dots \g \Gamma_1$ has the same stabilisation pattern as $\Gamma_{j-1}\g \dots \g \Gamma_1$.
	We first show that the map $\chi^-_{\lc\Gamma}$ is degenerate. By definition, 
	\begin{multline*}\label{} 
		\chi^-_{\lc\Gamma}(s)\; =\; \sum_{i =1}^{j-1}\delta_{\Gamma_i\g \dots \g\Gamma_1}\,s_i\prod_{\ell=1}^{i-1}(1-s_\ell) \;+ \;\delta_{\Gamma_i\g \dots \g\Gamma_1}(s_j + s_{j+1}(1-s_j))\prod_{\ell=1}^{j-1}(1-s_\ell) \\+\; (1-s_j)(1-s_{j+1})\sum_{i =1}^{j-1}\delta_{\Gamma_i\g \dots \g\Gamma_1}\,s_i\prod_{\ell=j+2}^{i-1}(1-s_\ell),
	\end{multline*}
	where $\delta_{\Gamma}$ is defined by~\eqref{stabilisation-pattern} and the first sum does not depend on $s_j$ and $s_{j+1}$. Since 
	$$(1-s_j)(1-s_{j+1}) = 1 -(s_j+s_{j+1}(1-s_j)),$$\noindent 
	it follows that  $\chi^-_{\lc\Gamma}(s)$ can be written as the sum 
	\[\chi^-_{\lc\Gamma}(s) = f(s') + (1-s_j)(1-s_{j+1}) g(s')\]
	where $s' = (s_1,\dots,s_{j-1},s_{j+2},\dots,s_{d-1})$. Thus, $\chi^-_{\lc\Gamma}$ factors through the map $I^{d-1}\to I^{d-2} : s \mapsto (s',s_js_{j+1})$, which is degenerate. It follows from a similar computation that $\chi^+_{\lc\Gamma}$ factors through this map as well. While the maps $\chi^-_{\lc\Gamma,i}$ ($\chi^+_{\lc\Gamma,i}$) do not depend on $s_j$ or $s_{j+1}$ for $i <j$ (respectively $i > j+1$), the map $\chi^-_{\lc\Gamma,j}$ does. However, since we are not stabilising $\cK_{T_{\ide}}$, this does not matter and the above argument shows that $\chi^-_{\lc\Gamma,i}$ factors through $s\mapsto (s',(1-s_j)(1-s_{j+1}))$ as well. The same is true for $\chi^+_{\lc\Gamma,j}$ and $\chi^+_{\lc\Gamma,j+1}$. This shows that $\cKc_{\lc\Gamma,\lc\sigma}$ is degenerate, whence $\scF_{\lc\sigma} = 0$ by the degeneracy property in Definition~\ref{de:curved-dmft}.
\end{proof}

\subsection{$A_\infty$ functor relations}\label{subsec:warped-equations-proof} We can now complete the proof of Theorem~\ref{thm:dmft-exists} by showing that the operations $\scF^d$ defined above satisfy the $A_\infty$ functor relations.

\begin{proposition}\label{prop:warped-equations-satisfied} $\scF$ satisfies Equation~\eqref{eq:warped-relations}. That is, 
	\begin{multline}\label{eq:warped-relations-again} 
		(-1)^{|\scF^d_{\lc \sigma}|}(d_{y^d}\g\scF^d_{\lc \sigma}-(-1)^{|\scF^d_{\lc\sigma}|}\scF^d_{\lc\sigma}\g d_{y_0}) + 
		\sum_{i=1 }^{d-1}{(-1)^{|\lc\sigma^{\leq i}|'+1}\scF^{d-i}_{\lc\sigma^{> i}}\g\scF^i_{\lc\sigma^{\le i}}}\\ 
		= \sum_{i=1}^{d}(-1)^{|\lc\sigma^{\leq i}|' + 1}\scF^d_{\sigma_1,\dots,\partial_{oc}\sigma_{i},\dots,\sigma_d} +\sum_{i=1}^{d-1}{(-1)^{|\lc\sigma^{\leq i}|' + 1}\scF^{d-1}_{\sigma_d,\dots,\sigma_{i+1}\g \sigma_{i},\dots,\sigma_1}},
	\end{multline}
	for any $d \ge 1$ and any $\sigma_i \in \cdmc_n(y^{i-1},y^{i})$.
\end{proposition}

The proof has several steps and will occupy the rest of the subsection. We first state two consequences.

\begin{cor}\label{cor:our-dmft-properties}
	$\scF$ is semi-strict and operadic.
\end{cor}

\begin{proof}
	Due to Proposition~\ref{prop:correspondences-operations}\eqref{degenerate-vanishes}, the proof boils down to showing that $\cK_{\lc\Gamma,\lc\sigma}$ is a degenerate correspondence whenever $d > 2$ and $\lc\sigma$ is a sequence of chains in $\cdmc_n^{*,\stb}$ or $\cdmc_n^{*,\text{opr}}$.\par 
	It suffices to prove that $\cKc_{\lc\Gamma} = I\times \cK$ for some global Kuranishi chart $\cK$ so that $\cKc_{\lc\Gamma}\to \cMc_{\lc\Gamma^{\stb}}$ and the evaluation maps factor through the projection to $\cK$. The key point is that the functions $\chi^\pm_{\lc\Gamma}$ are constant whenever $\delta_{\Gamma_j\g \dots \Gamma_1}(\kappa)$ and $\delta_{\Gamma_d\g \dots \g \Gamma_j}(\kappa)$ do not depend on $j \in \{1,\dots,d\}$, where $\kappa \in y$ labels an incoming marked point and $\kappa'\in y'$ an outgoing one. It follows from our definition of $\cKc_\Gamma$ in Theorem~\ref{thm:cubical-cobordisms-enhanced} that this is the case whenever $\Gamma$ admits a stabilisation or when each component of $\Gamma$ has only one outgoing marked point.
\end{proof}

\begin{proof}[Proof of Proposition~\ref{prop:warped-equations-satisfied}]
	Suppose $\sigma_i \in \cdmc_n^*(y,y')$ is represented by an analytic simplex $\sigma_i\cl \Delta^{\norm{\sigma_i}}\to \cMc_{\ov\Gamma_i}$, where we take this to be a point if $\ov\Gamma_i$ is a disjoint union of truly unstable corolla. Fix $\beta\in H_2(X,L;\bZ)$, respectively $H_2(X;\bZ)$ if all $\ov\Gamma_i$ are closed, and set 
	 \begin{equation*} 
	 	\scS_d = \set{\lc{\Gamma} = (\Gamma_d,\dots,\Gamma_1)\;\big|\;\sum_{i=1}^d\beta_{\Gamma_i}= \beta,\,\Gamma_i^{\stb} = \ov{\Gamma}_i}.
	 \end{equation*} 
	 Then, we have for $\lc\Gamma\in \scS_d$ by Stokes' formula, \cite[Theorem~1(e)]{ST23b}, that
	 \begin{equation}\label{eq:application-stokes}
	 	d\fq_{\lc\Gamma,\lc\sigma}(\alpha)\ =\ \fq_{\lc\Gamma,\lc\sigma}(d\alpha)\ +\ (-1)^{|\alpha|+\vdim(\cK_{\lc\Gamma,\lc\sigma})}(\eva^+_{\lc\Gamma,\lc\sigma}|_{\del\cT_{\lc\Gamma,\lc\sigma}})_*({\eva^-_{\lc\Gamma,\lc\sigma}}^*(\alpha) \wedge p_{\lc\sigma}^* \obs_{\lc\Gamma}^*\eta_{\lc\Gamma})
	 \end{equation}
	 with
	 
	 \begin{align}\label{eq:dimension-computation}
	 	\notag\vdim(\cK_{\lc\Gamma,\lc\sigma}) \;&=\; d-1+2n|y^d_c|+n|y_o^d|+\sum_{i=1}^d{(\vdim(\cK_{\Gamma_i,\sigma_i})-2n|y^i_c|-n|y_o^i|)}\;\\ \notag&\equiv\; 1+ n|y_o^d|+|\lc\sigma_i|'\mod{2}.
	 	 \end{align}
	 Writing $\Delta^{\norm{\lc\sigma}}:= \Delta^{\norm{\sigma_d}}\times \dots\times\Delta^{\norm{\sigma_1}}$, we have that up to resolving singular points of the fibre product $\cT_{\lc\Gamma,\lc\sigma}$, its boundary is 
	 \begin{equation}\label{eq:boundary-fibre-product}
	 	\del^v\cTc_{\cl\Gamma}\times_{\cMc_{\lc{\ov{\Gamma}}}}\Delta^{\norm{\lc\sigma}} \;\sqcup\; \cT_{\lc\Gamma}\times_{\cMc_{\lc{\ov{\Gamma}}}}\del\Delta^{\norm{\lc\sigma}} 
 	\end{equation}
	 by \cite[Proposition~7.4]{Joy} (ignoring orientations). Here we write $\del^v\cTc_{\lc\Gamma}$ for the vertical boundary of the stabilisation map $\stb \cl \cTc_{\lc\Gamma}\to \cMc_{\lc{\ov{\Gamma}}}$; this corresponds to $\del^{\stb}_+\cTc_\Gamma$ in the notation of \cite{Joy}.\footnote{See Definition~4.1 and Proposition~4.3 op. cit. for his definition of vertical boundary.} The operation associated to the boundary of $\Delta^i$ in the second component in the disjoint union of~\eqref{eq:boundary-fibre-product} comes with the sign $(-1)^{n|y_o^d|+|\sigma_d|'+\dots+|\sigma_{i+1}|'}$. Indeed, this is the orientation sign of the vertical isomorphism on the left in the square
	 
	 \begin{center}\begin{tikzcd}
	 		\orl(\bR)\orl(\cK_{\lc\Gamma,(\sigma_1,\dots,\del\sigma_i,\dots,\sigma_d)}) \arrow[r,""] \arrow[d,""]&\orl(\bR)\orl(Y_{\sigma_{d}}^+)\orl(\sigma_d)\cdots\orl(I^{\{i\}})\orl(\del\sigma_i)\cdots\orl(I^{\{1\}})\orl(\sigma_1) \arrow[d,""]\\ 
	 		\orl(\cK_{\lc\Gamma,\lc\sigma}) \arrow[r,""] & \orl(Y_{\sigma_{d}}^+)\orl(\sigma_d)\orl(I^{\{d-1\}})\orl(\sigma_{d-1})\cdots\orl(I^{\{1\}})\orl(\sigma_1) \end{tikzcd} \end{center}
	 since we use the canonical isomorphism 
	 \[\orl(\sigma_i) = \orl(\det(\sigma_i))\orl(|\sigma_i|)\cong \orl(\Delta^{\norm{\sigma_i}})\orl(\det(\sigma_i))\orl(n\chi_i -2n|y_c^i|-n|y_o^i|)\]
	 recalling that $\orl(\det(\sigma_i))$ has degree $0$.
	 Together with the sign of Equation~\eqref{eq:application-stokes}, this yields the operation
	  \begin{equation}\label{eq:sign-usual-differential-chain} 
	  	(-1)^{|\alpha|+\vdim(\cK_{\lc\Gamma,\lc\sigma})}(\eva^+_{\lc\Gamma,\lc\sigma}|_{\del_i\cT_{\lc\Gamma,\lc\sigma}})_*({\eva^-_{\lc\Gamma,\lc\sigma}}^*(\alpha) \wedge p_{\lc\sigma}^* \obs_{\lc\Gamma}^*\eta_{\lc\Gamma})= (-1)^{|\alpha|+|\lc\sigma^{\le i}|'+1}\fq_{\lc\Gamma,(\sigma_1,\dots,\del\sigma_i,\dots,\sigma_d)}(\alpha).\end{equation}
	 Let us now turn to the vertical boundary of $\cTc_\Gamma$. It follows from its description as a parametrised fibre product, there are three types of contribution to the vertical boundary
	 \begin{itemize}[leftmargin=20pt]
	 	\item The first is of the form $\cTc_{\lc\Gamma}|_{\{s_i =k\}}$ for $k = 0,1$ and some $1 \leq i \leq d-1$.
	 	\item The second comes from the locus $\cTc_{\lc\Gamma}|_{\{t_e =k\}}$, with $k = 0,1$ for an edge $e\in E(\Gamma_i)$ connecting an unstable tree $T$ to the rest of $\Gamma_i$. Note that the case where $T$ is unstable but not a branch is a codimension-$2$ phenomenon, which we may ignore.
	 	\item The third contribution is from boundary strata $\del_{\Gamma'_i} \cTc_{\lc\Gamma}$ for some $i$, where $\Gamma'_i$ is a stable map graph such that $(\Gamma_i')_{e'} = \Gamma_i$ for some edge $e'$ of $\Gamma'_i$.
	 \end{itemize} 
 	  By Lemma~\ref{lem:gkc-for-higher-homotopies}\eqref{higher-edge-zero} and Lemma~\ref{lem:total-degree-contraction}, the operation associated to the third kind of contribution cancels with operations associated to the boundary strata $\cKc_{\lc\Gamma}|_{\{t_e = 0\}}$.
 	  
Recall the notation
 	 	\[\lc\sigma^{> i} \coloneqq (\sigma_d,\dots,\sigma_{i+1}) \qquad\quad \qquad\lc\sigma^{\le i} \coloneqq (\sigma_i,\dots,\sigma_1).\]
 	 	We will use the same notation for sequences of graphs. Moreover, we abbreviate 
 	 	\begin{equation*}\label{} 
 	 		\orl(\lc\sigma)\,\coloneqq\, \orl(\sigma_d)\cdots\orl(\sigma_1),
 	 	\end{equation*}
 	 	using the notation from~\eqref{eq:simplex-line}.
 	 	Note the change in ordering. 
 	
 	 By Lemma~\ref{lem:gkc-for-higher-homotopies}\eqref{higher-side-one}, the operation contributed by $\cKc_{\lc\Gamma}|_{\{s_i = 1\}}\times_{\cMc_{\lc{\ov{\Gamma}}}}\Delta^{\norm{\lc\sigma}}$ to~\eqref{eq:application-stokes} is exactly
 	 \begin{equation}\label{eq:boundary-higher-side-one-before-sign} 
 	 	 (-1)^{|\alpha|+\vdim(\cK_{\lc\Gamma,\lc\sigma})+\varepsilon_1}\fq_{\lc\Gamma^{> i},\lc\sigma^{> i}}(\fq_{\lc\Gamma^{\le i},\lc\sigma^{< i}}(\alpha)),
 	 \end{equation}
 	 where the sign $(-1)^{\varepsilon_1}$ is the orientation sign of the vertical isomorphism on the left in the diagram
 	 \begin{center}\begin{tikzcd}
 	 	\orl((\cKc_{\lc\Gamma}|_{\{s_i = 1\}})_{\lc\sigma}) \arrow[d,]\arrow[r,"(-1)^{n|y_o^d|+|\lc\sigma^{> i}|'+1}\eqref{eq:orientation-higher}"] &\orl(Y^+_d)\orl(\sigma_d)\orl(I^{\{d-1\}})\cdots\orl(\sigma_1) \arrow[d,"(1)"]\\ 
 	 	\orl(\cKc_{\lc\Gamma^{>i},\lc\sigma^{>i}}\times_{Y^+_i}\cKc_{\lc\Gamma^{\le i},\lc\sigma^{\le i}}) \arrow[r] & \orl(Y^+_d)\orl(\sigma_d)\orl(I^{\{d-1\}})\cdots\orl(\sigma_{i+1})\orl(Y_i^+)\dul\orl(Y_i^+)\orl(\sigma_i)\cdots \orl(\sigma_1) 
  	\end{tikzcd}
   \end{center}
 	 	where (1) is orientation-preserving. Thus, $\varepsilon_1 \,=\, n|y_o^d|+|\lc\sigma^{> i}|'+1$ and the sign in front of the right hand side of~\eqref{eq:boundary-higher-side-one-before-sign} is $(-1)^*$ with
 	 \begin{align}\label{eq:sign-composition-functors} 
 	 \notag * \;&= \;|\alpha|+ n|y^d_o|+ |\lc\sigma'|+1 +  n|y_o^d|+|\lc\sigma^{> i}|'+1\\&\equiv\; |\alpha|+|\lc\sigma^{\le i}|'\mod{2}.
 	 \end{align}
 	 Note that $c_{\lc\Gamma} = c_{\lc \Gamma^{>i}} c_{\lc\Gamma^{\leq i}} $ by definition, cf. \eqref{eq:coefficient-higher-homotopies}.
	 
	 \noindent Meanwhile, by Lemma~\ref{lem:gkc-for-higher-homotopies}\eqref{higher-side-zero}, the operation associated to the boundary stratum $\cKc_{\lc\Gamma}|_{\{s_i = 1\}}\times_{\cMc_{\lc{\ov{\Gamma}}}}\Delta^{\norm{\lc\sigma}}$ is $(-1)^{|\alpha|+n|y^d_o|+|\lc\sigma|'+1+\varepsilon_{0}}\fq_{*_i\lc\Gamma,*_i\lc\sigma}$, where 
	 \[ *_i\lc\sigma = (\sigma_1,\dots,\sigma_{i+1}\g \sigma_i,\dots,\sigma_d)\] 
	 and similarly for $\lc\Gamma$. The sign $(-1)^{\varepsilon_0}$ is the orientation sign of the vertical isomorphism on the left in the diagram
	 \begin{center}\begin{tikzcd}
	 	\orl((\cKc_{\lc\Gamma}|_{\{s_i = 0\}})_{\lc\sigma}) \arrow[d,]\arrow[rrr,"(-1)^{n|y_o^d|+|\lc\sigma^{> i}|'}"] &&&\orl(Y^+_d)\orl(\sigma_d)\orl(I^{\{d-1\}})\cdots\orl(\sigma_{i+1})\orl(\sigma_i)\orl(I^{\{i-1\}})\cdots\orl(\sigma_1) \arrow[d,"(1)"]\\ 
	 	\orl(\cKc_{*_i\lc\Gamma,*_i\lc\sigma}) \arrow[rrr] &&& \orl(Y^+_d)\orl(\sigma_d)\orl(I^{\{d-1\}})\cdots\orl(\sigma_{i+1}\g \sigma_i)\orl(I^{\{i-1\}})\cdots\orl(\sigma_1) \end{tikzcd} \end{center}
	 where (1) is orientation-preserving by the definition of the orientation~\eqref{eq:orientation-higher}. We thus obtain the operation $(-1)^{*}\fq_{*_i\lc\Gamma,*_i\lc\sigma}$ where
	 \begin{align}\label{eq:sign-composition-chains} 
	 \notag *\,&=\,|\alpha|+n|y^d_o|+|\lc\sigma|'+1+n|y_o^d|+|\lc\sigma^{> i}|'\\
	 &\equiv\,|\alpha|+|\lc\sigma^{\leq i}|' +1\mod{2}.
	 \end{align}
	 By \cite[Lemma~4.22]{HH25} and the definition of $c_\Gamma$ in~\eqref{eq:q-coefficient}, we have $c_{\lc\Gamma} = c_{\concat_i\lc\Gamma}$ as well.
	 
	  \noindent We are thus only left with the $\{t_{e_v} = 1\}$-boundary strata, which contribute the `positive-energy' terms of~the differential $\del^\scF$ as well as the curvature. We first discuss the case of $d = 1$.
	 
\begin{lemma}\label{lem:chain-map-stable} $\scF^1_\sigma$ satisfies Equation~\eqref{eq:warped-relations-again} when $\sigma$ is a stable chain.
\end{lemma}

\begin{proof} 
	 The strata $\del_e\cTc_\Gamma\coloneqq\cTc_\Gamma|_{\{t_{e_v} = 1\}}$ are of three different types. We discuss them first in the case of $d = 1$, since the general case can then be derived more easily. In this case we have a single graph $\Gamma$ and the boundary strata 
	 \begin{enumerate}[\normalfont i),leftmargin=*,ref=\roman*]
	 	\setlength\itemsep{5pt}
	 	\item\label{incoming-d} $\del_e\cTc_\Gamma \cong \cT_{T}\times_L\cTc_{\Gamma'}$
	 	\item\label{outgoing-d} $\del_e\cTc_\Gamma \cong \cTc_{\Gamma'}\times_L\cT_{T}$
	 	\item\label{curv-d} $\del_e\cTc_\Gamma \cong \cTc_{\Gamma''}\times_L\cT_{T}$
	 \end{enumerate}
 	where in the first two cases $T$ (after contracting all edges) encodes a disc with one incoming and one outgoing marked point, attached to $\Gamma' = \Gamma\sm\{v\}$ either at the incoming or the outgoing marked point of $T$ and in~\eqref{curv-d}, the tree (after contracting all edges) encodes a disc with only an outgoing marked point. By our orientation convention~\eqref{con:orientation}, it suffices to consider the case when $T$ and $\Gamma'$ have no edges. We will suggestively write $\mathsf{a}_1$ and $\mathsf{a}_0$, respectively to distinguish between the two types of vertices. For the proof here, the perspective of \cite[\S B]{She25} is useful. Observe in particular that for $\mathsf{a}_1$ we have a canonical isomorphism 
 	
 	\begin{equation}\label{eq:iso-to-det}
 		\orl(\cK_{\mathsf{a}_1}) \;\cong\; \orl(\mu_L(\beta))\orl(L)\orl(T_-)\orl(sc)\dul\orl(tr)\dul,  
 	\end{equation}
 	where we have fixed the outgoing marked point at $\infty$ (identifying the disc with the upper half-plane), which leaves the two actions of positive scaling and positive translation, indicated by $sc$, respectively $tr$ here. We will henceforth omit the trivial even-rank orientation line as we are only interested in the sign. The isomorphism with $L$ is induced by the evaluation map at the outgoing (or equivalently incoming) marked point. As observed in \cite{She25} the linear gluing identifies the \emph{inward-pointing} normal vector with the positive scaling of the action, while positive translation corresponds to the infinitesimal variations of the marked point on the disc bubble; which are thus tangential to the boundary stratum.
 	This shows that in Case~\eqref{incoming-d} we have the following diagram
 	
 	\begin{center}\begin{tikzcd}
 			\orl(\bR)\orl(\cK_{\mathsf{a}_1}) \orl(L)\dul\orl(\cK_{\mathsf{a}',\sigma})\arrow[rr,"(-1)\eqref{eq:iso-to-det}"] \arrow[dd,""]&&\orl(\bR)\orl(L)\orl(T_+)\orl(sc)\dul\orl(tr)\dul\orl(L)\dul \orl(Y^+_\sigma)\orl(\sigma)\arrow[d,"(1)"]\\ 
 			&&\orl(L)\orl(\bR)\orl(sc)\dul\orl(L)\dul\orl(Y^+_\sigma)\orl(\sigma)\arrow[d,"(2)"]\\
 			\orl(\cK_{\Gamma,\sigma}) \arrow[rr,""]	&& \orl(Y^+_\sigma)\orl(\sigma)
 	\end{tikzcd} \end{center}
 	The map $(1)$ has orientation sign $(-1)^{n+1}$ because we moved $\orl(\bR)$ past $\orl(L)$ and $\orl(T_+)$ past $\orl(sc)\dul$. Meanwhile, $(2)$ has orientation sign $(-1)$ because $\orl(\bR)$ is oriented with respect to the outward point normal vector and the discussion above. Thus, the vertical morphism on the left has orientation sign $(-1)^{n+1}$.
 	As before, the same sign appears if $\Gamma$ (and thus $\Gamma'$) is an arbitrary graph. Then, by Lemma~\ref{lem:different-fibre-products}, the orientation sign of the equivalence between $\del_e\cK_{\Gamma,\sigma}$ and $(L^{i-1}\times \cK_{\mathsf{a}_1}\times L^{|y'_o|-i})\times_{Y_\sigma^+}\cK_{\Gamma',\sigma}$ is $(-1)^{\epsilon_{\eqref{incoming-d}}}$ with 
 	
 	\begin{align*}\label{}
 		\epsilon_{\eqref{incoming-d}}\;&=\;n+1+ n(i-1) 
 		\\&=ni+1 
 	\end{align*}
 	By our sign convention, we thus have 
 	
 	\begin{align*}\label{} 
 		(-1)^{|\alpha|+n\chi+\norm{\sigma}}(\eva_{\Gamma,\sigma}^+|_{\del_e})_*((\eva^-_{\Gamma,\sigma})^*(\alpha)\wedge\lambda_{\Gamma,\sigma}|_{\del_e})\, &=\,(-1)^{|\alpha|+|\sigma|+1} \fq_{T\g_i\Gamma',D_1\g_i\sigma}(\alpha).
 	\end{align*} 
 	Since the coefficient $c_{\Gamma}$, defined in~\eqref{eq:q-coefficient}, only depends on the prestable graph underlying $\Gamma$, we see that $c_\Gamma = c_{\Gamma'}$ for $\Gamma$ and $\Gamma'$ in the above equation.
 	
 	In Case~\eqref{outgoing-d} we have the following diagram
 	
 	\begin{center}\begin{tikzcd}
 			\orl(\bR)\orl(\cK_{\mathsf{a}',\sigma})\orl(L)\dul\orl(\cK_{\mathsf{a}_1}) \arrow[r,""] \arrow[ddd,""]&\orl(\bR)\orl(Y^+_\sigma)\orl(\sigma)\orl(L)\dul\orl(L)\orl(T_-)\orl(sc)\dul\orl(tr)\dul \arrow[d,"(1)"]\\ 
 			&\orl(Y^+_\sigma)\orl(\sigma)\orl(T_-)\orl(\bR)\orl(sc)\dul\orl(tr)\dul\arrow[d,"(2)"]\\
 			& \orl(Y^+_\sigma)\orl(\sigma)\orl(T_-)\orl(tr)\dul\arrow[d,"(3)"]\\
 			\orl(\cK_{\Gamma,\sigma}) \arrow[r,""]	& \orl(Y^+_\sigma)\orl(\sigma)
 	\end{tikzcd} \end{center}
 	Here $(1)$ has orientation sign $(-1)^{|\sigma|+n\ell^++1}$ because we moved $\orl(\bR)$ past the other orientation lines. Note that the orientation sign of $\orl(L)\dul\orl(L)\cong \orl(0)$ is $1$ because the identification uses the map $-d\eva_{z_+}$. The map $(2)$ has orientation sign $(-1)$ because $\orl(\bR)$ is oriented with respect to the outward point normal vector and the discussion above. Finally, $(3)$ is orientation-preserving. In total, we see that the vertical morphism on the right has orientation sign $(-1)^{\epsilon'_{\eqref{outgoing-d}}}$ with 
 	
 	\begin{equation*}\label{} 
 		\epsilon'_{\eqref{outgoing-d}} = n|y'_o| +|\sigma| .
 	\end{equation*}
 	Returning to the case where $\Gamma$ (and thus $\Gamma'$) is an arbitrary graph, it follows from Lemma~\ref{lem:different-fibre-products} that the orientation sign of the equivalence between $\del_e\cK_{\Gamma,\sigma}$ and $\cK_{\Gamma',\sigma}\times_{Y_\sigma^-}(L^{i-1}\times \cK_{\mathsf{a}_1}\times L^{|y_o|-i})$ is $(-1)^{\epsilon_{\eqref{outgoing-d}}}$ with 
 	
 	\begin{align*}\label{} 
 		\epsilon_{\eqref{outgoing-d}} \;&= \;\epsilon_{\eqref{outgoing-d}}' + n(|y_o|-i)(\vdim(\cK_{\mathsf{a}_1})-n)\\&\equiv \; n|y'_o| +|\sigma| + n(|y_o|-i)\mod{2}.
 	\end{align*}
 	Then, by Definition~\ref{de:differential-and-curvature}, we have
 	
 	\begin{align*}\label{} 
 		(-1)^{|\alpha|+n|y'_o|+|\sigma|}(\eva_{\Gamma,\sigma}^+|_{\del_e})_*((\eva^-_{\Gamma,\sigma})^*(\alpha)\wedge\lambda_{\Gamma,\sigma}|_{\del_e}) \;&=\; (-1)^{|\alpha|} \fq_{\Gamma'\g_iT,\sigma\g_i D_1}(\alpha).
 	\end{align*} 
 	due to the fact that $\scF$ is compatible with $D_1$ over $\cdmc_n^{\stb}$. As before, we have $c_\Gamma = c_{\Gamma'}$.
 	
 	Finally, in Case~\eqref{curv-d}, we have to consider $\cK_{\Gamma',\pi^*\sigma}$, or, equivalently the fibre product $\cKc_{\Gamma'}\times_{\cMc_{\ov\Gamma}}\Delta^m$, where the graph $\Gamma'$ has an additional incoming marked point, \emph{ordered before all other incoming marked points} (although it appears on the boundary after the $i$th marked point) and the map $ \cTc_{\Gamma'}\to \cMc_{\ov\Gamma}$ factors through the map forgetting said marked point. Thus, $\cK_{\Gamma',\pi^*\sigma}$ is oriented via the isomorphism 
 	\begin{align}\label{eq:orientation-of-lift} 
 		\notag\orl(\cK_{\Gamma',\pi^*\sigma}) \;&\cong\; \orl(Y^+_\sigma)\orl(\pi^*\sigma)\\
 		\;&\cong \; \orl(\bR)\orl(Y^+_\sigma)\orl(\sigma).
 	\end{align}
 	where we use the isomorphism $\orl(\pi^*\sigma)\cong \orl(\sigma)\orl(\bR)$, cf. Construction~\ref{con:lifted-simplex}. We write $\Gamma'_-$ for the graph obtained by forgetting the additional marked point. The second isomorphism has sign $(-1)^\star := (-1)^{n|y'_o|+|\sigma|}$ because we permute $\bR$ past $\orl(Y^+_{\pi^*\sigma})$. To distinguish the normal direction from this $\bR$-factor, we denote the normal direction by $\bR_{\del}$ in the commutative diagram
 	\begin{center}\begin{tikzcd}
 			\orl(\bR_{\del})\orl(\cK_{\mathsf{a}^+,\pi^*\sigma})\orl(L)\dul\orl(\cK_{\mathsf{a}_0}) \arrow[rr,"(-1)^{\star}\eqref{eq:orientation-of-lift}"] \arrow[ddd,""]&&\orl(\bR_{\del})\orl(\bR)\orl(Y^+_\sigma)\orl(\sigma)\orl(L)\dul\orl(L)\orl(sc)\dul\orl(tr)\dul \arrow[d,"(1)"]\\ 
 			&&\orl(Y^+_\sigma)\orl(\sigma)\orl(\bR_{\del})\orl(\bR)\orl(sc)\dul\orl(tr)\dul\arrow[d,"(2)"]\\
 			&& \orl(Y^+_\sigma)\orl(\sigma)\orl(T_-)\orl(tr)\dul\arrow[d,"(3)"]\\
 			\orl(\cK_{\Gamma,\sigma}) \arrow[rr,""]	&& \orl(Y^+_\sigma)\orl(\sigma)
 	\end{tikzcd} \end{center}
 	where the map $(1)$ has orientation sign $1$ as does $(2)$, using the same arguments as above and because we moved $\orl(\bR_{\del})$ past $\orl(\bR)$. Finally, $(3)$ is orientation-preserving. In total, we see that the vertical morphism on the left has orientation sign $(-1)^{\epsilon_{\eqref{curv-d}}}$, where 
 	\begin{equation}\label{eq:sign-curvature-basic} 
 		\epsilon_{\eqref{curv-d}} = n|y'_o|+|\sigma|.
 	\end{equation}
 	Since $\vdim(\cK_{\mathsf{a}_0})-n \equiv 0$, Lemma~\ref{lem:different-fibre-products} implies that 
 	\[\cK_{\mathsf{a}^+,\pi^*\sigma}\times_{X^{y_c}\times L^{*\sqcup y_o}}(X^{y_c}\times\cK_{\mathsf{a}_0}\times L^{y_o}) \simeq (-1)^{\epsilon_{\eqref{curv-d}}}\del\cK_{\Gamma,\sigma}-\]
 	Thus, the associated operation~\eqref{eq:application-stokes} becomes
 	
 	\begin{equation}\label{eq:curvature-term} 
 		(-1)^{|\alpha|+|\sigma|-n|y'_o|}(\eva_{\Gamma,\sigma}^+|_{\del_e})_*((\eva^-_{\Gamma,\sigma})^*(\alpha)\wedge\lambda_{\Gamma,\sigma}|_{\del_e}) \;=\; (-1)^{|\alpha|} \fq_{\Gamma'\g_* T,\pi^*\sigma\g_* D_0}(\alpha).
 	\end{equation}
 In this case, the underlying prestable graph of $\Gamma'$ differs from the one of $\Gamma$ and we have to invoke \cite[Lemma~4.23]{HH25} to see that
 \begin{equation}\label{} 
 	\s{\pi_j^*\Gamma}{c_\Gamma} = c_{\Gamma'}.
 \end{equation}
 Summing over all suitable graphs $\Gamma$ completes the proof.
\end{proof}	 

\noindent
The next lemma is used in the above proof to show that `unstable' boundaries in the sum over all graphs $\Gamma$ of a given stabilisation cancel.

\begin{lemma}\label{lem:total-degree-contraction}
	Let $\Gamma$ be a stable map graph and $\Gamma'$ be obtained from $\Gamma$ by an orientation-preserving map $f \cl \Gamma\to \Gamma'$, which contracts a single unstable edge. Then,
	\begin{equation}\label{eq:cancellation-for-chain-map} 
		\s{\substack{\widehat{\Gamma} \in \scP\scS^+(\Gamma)}}{\s{\substack{e\in E(\wh\Gamma)\\\widehat{\Gamma}_e \cong_+ \Gamma'}}{c_{\Gamma}\,\fq_{\Gamma,\sigma}|_{\{t_e =0\}}(\cdot)}}\,=\, \s{\widehat{\Gamma}' \in \scP \scS^+(\Gamma')}{\s{\scS\scG^+(\Gamma,\Gamma')/\Aut^+(\Gamma)}{c_{\Gamma'}\,\fq_{\widehat{\Gamma}',\sigma}|_{\del_{\Gamma}\cKc_{\wh\Gamma'}}(\cdot)}}
	\end{equation}
	for any analytic chain $\sigma\in C^{an}_*(\cMc_{\ov\Gamma})$.
\end{lemma}

\begin{proof}
	Note first that 
	\[\s{\substack{\widehat{\Gamma} \in \scP \scS^+(\Gamma)}}{\s{\substack{e\in E(\wh\Gamma)\\\widehat{\Gamma}_e \cong_+ \Gamma'}}{\fq_{\Gamma,\sigma}|_{\{t_e =0\}}(\cdot)}} = |\scP\scS^+(\Gamma)|\s{\substack{e\in E(\Gamma)\\\Gamma_e = \Gamma'}}{\fq_{\Gamma,\sigma}|_{\{t_e =0\}}(\cdot)}
	\]
	and similarly for the right hand side of Equation~\eqref{eq:cancellation-for-chain-map}. Over the smooth locus, the map $\cKc_{\Gamma,\sigma}|_{\{t_e = 0\}}\to \del_\Gamma\cKc_{\Gamma',\sigma}$ is the pullback of a local embedding of degree \[d_1(\Gamma,\Gamma') = |\{\phi\in\Aut(\Gamma)\mid f\g \phi = f\}|\]
	by Theorem~\ref{thm:cubical-cobordisms-enhanced}, where $f \cl \Gamma\to \Gamma'$ is the underlying contraction. (Since the graphs $\Gamma$ and $\Gamma'$ admit a stabilisation, we have $d_1(\Gamma,\Gamma') =1$ in this case. But this equation holds more generally.) Since the contracted half-edges have no orientation sign, we have 
	\[d_1(\Gamma,\Gamma') = |\{\phi\in\Aut^+(\Gamma)\mid f\g \phi = f\}|.\]
	Moreover,
	\begin{equation}\label{} 
		d_2(\Gamma,\Gamma') \coloneqq|\{ e\mid\Gamma_e \cong_+ \Gamma'\}| =|\Aut^+(\Gamma')\backslash \scS\scG_m^+(\Gamma,\Gamma')|.\\
	\end{equation}
	Finally, setting $d_3(\Gamma,\Gamma') := |\scG(\Gamma,\Gamma')/\Aut(\Gamma)|$, we have $d_3(\Gamma,\Gamma')$-many boundary strata of $\cKc_{\Gamma'}$ that come from a local embedding of $\cKc_\Gamma$ restricted to a face of the cube $I^{E(\Gamma)}$. Since $d_1(\Gamma,\Gamma')$ does not depend on the choice of contraction $\Gamma\to \Gamma'$, we have 
	\[|\scS\scG_m(\Gamma,\Gamma')| = |\scS\scG_m^+(\Gamma,\Gamma')/\Aut^+(\Gamma)|\, \frac{|\Aut^+(\Gamma)|}{d_1(\Gamma,\Gamma')}.\]
	Because contractions are surjective, the action of $\Aut^+(\Gamma')$ on $\scS\scG_m^+(\Gamma,\Gamma')$ is free, whence  
	\[\frac{d_1(\Gamma,\Gamma')d_2(\Gamma,\Gamma')}{d_3(\Gamma,\Gamma')} = \frac{|\Aut^+(\Gamma)|}{|\Aut^+(\Gamma')|}.\]
	Together with the first reduction, this allows us to conclude using the equality in~\eqref{eq:q-coefficient} up to sign. Since the difference between $\Gamma$ and $\Gamma'$ are only unstable vertices, their orientations agree, whence the claim follows.
\end{proof}

\begin{lemma}\label{lem:different-fibre-products} Suppose we have submersions $X\to Z$ and $Y\to Z$ as well as maps $X\to V\times W$, where $X$ and $Y$ are oriented orbifolds with corners and $Z,V,W$ are oriented closed manifolds. Then, equipped with the respective fibre product orientation, the isomorphisms 
	\begin{equation}\label{} 
		(V\times X\times W)\times_{V\times Z\times W} Y\;\xra{\phi} \;X\times_Z Y\;\xleftarrow{\varphi}\; X\times_{V\times Z\times W}(V\times Y\times W)
	\end{equation}
	have orientation signs
	\begin{equation}\label{} 
		\epsilon(\phi) = (-1)^{\dim V(\dim X-\dim Z)} \qquad \qquad \epsilon(\varphi) = (-1)^{\dim W(\dim Y-\dim Z)}
	\end{equation}
\end{lemma}

\begin{proof} We use a trick from \cite{Joy} to simplify the computation. Since the question is local, we may assume $X = X'\times Z$ and $Y = Z\times Y'$ respectively, in which case $X\times_Z Y$ is oriented as $X'\times Y$, respectively $X\times Y'$. The signs now follow immediately from the fact that $\dim X' = \dim X- \dim Z$.
\end{proof}

\begin{lemma}\label{lem:curved-chain-map-unstable} $\scF^1_\sigma$ satisfies~\eqref{eq:warped-relations-again} when $\sigma = D_0$ or $D_1$. Explicitly, \[d_{(0,1)} \g \scF_{D_0} = -\scF_{D_1 \g D_0}\] and 
	\begin{equation*}\label{eq:curved-diff-base-case} 
		\scF_{D_1} \g d_{(0,1)} + d_{(0,1)} \g \scF_{D_1} = -\scF_{D_1 \g D_1} + \scF_{D_2 \g_1 D_0} - \scF_{D_2 \g_2 D_0}.
	\end{equation*}
\end{lemma}

\begin{proof} 
	Recall that $\scF_{D_0} = \s{\Gamma}{(\eva^+_\Gamma)(\obs_\Gamma^*\eta_\Gamma)}$ and that, since $|D_0|$ is even, $d_{(0,1)}(\scF_{D_0})  = d(\scF_{D_0})$. By Stokes' formula,
	\begin{equation*}\label{} 
		d_{(0,1)}\scF_{D_0} = d\scF_{D_0} =  0 + (-1)^{n} \s{\Gamma}{(\eva^+_\Gamma|_{\del\cTc_\Gamma})(\obs_\Gamma^*\eta_\Gamma|_{\del\cTc_\Gamma})}
	\end{equation*}
	where the index set runs over stable map graphs of type $(0,1)$ with one outgoing boundary marked point. By the same arguments as in the proof of Lemma~\ref{lem:chain-map-stable}, we can reduce to computing the sign in the equivalence $\cKc_\Gamma|_{\{t_e = 0\}} \simeq \cKc_{T}\times_L \cKc_{\Gamma_0}$, where $e$ is an edge of $\Gamma$, $T$ is a stable map graph of type $(0,1)$ with one incoming and one outgoing boundary marked point and $\Gamma_0$ is a stable map graph of type $(0,1)$ with one outgoing boundary marked point. By our orientation convention, we can in turn reduce this to the sign computation in \cite[Theorem~3.5]{HH25}, which yields the sign $(-1)^{n+1}$. This could also be obtained by a computation of a diagram as above. 
	For the second claim, we have  
	\begin{multline*}\label{} 
		d_{(0,1)}\scF_{D_1}(\alpha) \,=\, (-1)^{|\alpha|+1}d\s{\Gamma}{(\eva^+_\Gamma)({\eva^-_\Gamma}^*(\alpha)\wedge \obs_\Gamma^*\eta_\Gamma)Q^{\beta_\Gamma}}\\
		=\, -\s{\Gamma}{(\eva^+_\Gamma)({\eva^-_\Gamma}^*(d_{(0,1)}\alpha)\wedge \obs_\Gamma^*\eta_\Gamma)Q^{\beta_\Gamma}} \,-\, (-1)^{n+1} \s{\Gamma}{(\eva^+_\Gamma|_{\del\cTc_\Gamma})({\eva^-_\Gamma}^*(\alpha)\wedge \obs_\Gamma^*\eta_\Gamma|_{\del\cTc_\Gamma})Q^{\beta_\Gamma}.}
	\end{multline*}
	By the same argument as before and our orientation convention in~\eqref{con:orientation} we have to determine the orientation sign between $\cKc_\Gamma|_{\{t_e = 1\}}$ and $\cKc_{T}\times_L \cKc_{\Gamma''}$, where $T$ and $\Gamma'$ are two graphs obtained from $\Gamma$ by cutting the edge $e$ and where $\Gamma'$ and $\Gamma''$ are single vertices. The graph $T$ is of type $(0,1)$ with one incoming and one or two outgoing boundary marked points, while $\Gamma''$ has one outgoing marked point and one incoming, respectively no incoming, boundary marked point. We write $e = e_1$ in the first case and $e = e_2$ in the second one. In the second case, we also have that the boundary node is either ordered before the incoming and after the outgoing marked point on $\Gamma'$, in which case we say $e_2$ is of type $+$, or it is ordered after the incoming marked point. \cite[Theorem~3.5(1)]{HH25} yields the signs 
	\begin{itemize}
		\item $(-1)^{n+1}$ for $e_1$,
		\item $(-1)^{n+1}$ for $e_2$ of type $+$,
		\item $(-1)^{n}$ for $e_2$ of type $-$.
	\end{itemize}
	Write $\scF_{D_2} = \s{\beta}{\fq_{D_2}^\beta Q^\beta}$ and $\scF_{D_0} = \sum_\beta \fq_{D_0}^\beta Q^\beta$. With the notation from Definition~\ref{de:differential-and-curvature}, one has
	\begin{align*}\label{} 
		&(-1)^{n+1} \s{\Gamma}{(\eva^+_\Gamma|_{\del\cTc_\Gamma})_*({\eva^-_\Gamma}^*(\alpha)\wedge \obs_\Gamma^*\eta_\Gamma|_{\del\cTc_\Gamma})Q^{\beta_\Gamma}.}\, \\
		&\,=\,\s{\substack{\omega(\beta')>0\\\omega(\beta'') > 0}}{\fq_{D_1}^{\beta'}(\fq_{D_1}^{\beta''}(\alpha))Q^{\beta'+\beta''}} 
		\,+\, \s{\substack{\omega(\beta')\ge 0\\\omega(\beta'') > 0}}{\fq_{D_2}^{\beta'}(\fq_{D_0}^{\beta''},\alpha)Q^{\beta'+\beta''}} \,-\,
		\s{\substack{\omega(\beta')\ge 0\\\omega(\beta'') > 0}}{\fq_{D_2}^{\beta'}(\alpha,\fq_{D_0}^{\beta''})Q^{\beta'+\beta''}}  
	\end{align*}
	Since we constructed $\scF$ to be compatible with $D_0$ and $D_1$, it follows that
	\[\scF_{D_1} \g d_{(0,1)} + d_{(0,1)} \g \scF_{D_1} = -\scF_{D_1 \g D_1} + \scF_{D_2 \g_1 D_0} - \scF_{D_2 \g_2 D_0}\]
	as claimed.
\end{proof} 

\begin{lemma}\label{eq:chain-map-unstable} 
	$\scF^1_{\sigma}$ satisfies Equation~\eqref{eq:warped-relations-again} when $\sigma$ is an unstable chain.
\end{lemma}

\begin{proof}
	It follows from Stokes' theorem and the fact that $L$ and $X$ are closed that $[d_y,\scF^1_{\sigma}] = 0 = \scF^1_{\partial_{oc}\sigma}$ whenever $\sigma$ is a $0$-dimensional unstable simplex, while both sides vanish whenever $\sigma$ has higher dimension. It remains thus to consider the case with boundary marked points, where we have to check that $\scF^1_{\partial_{oc}\sigma} = 0$.
	\begin{itemize}[leftmargin=20pt]
	\item For the identity, we have $\scF^1_{\partial_{oc} \sigma} = \scF^1_{\idsimp \g D_1} - \scF^1_{D_1 \g \idsimp} = \scF^1_{D_1} - \scF^1_{D_1} = 0$, as required.
	\item For the co-unit~\eqref{forgetful-correspondence}, we have that $\scF^1_{\partial_{oc}\sigma}(\alpha) = \pm \scF_{\sigma \g D_1}(\alpha) = \pm \int_L \scF_{D_1}(\alpha)$ for any $\alpha \in \Omega^*(L)$. We prove this vanishes for each energy $\beta$ separately. By definition, the correspondences involved are of the form $L \xleftarrow{} \cKc_\Gamma \rightarrow \pt,$ 
	where $\Gamma$ is a stable map graph of type $(\beta,0,1)$ with one incoming and one outgoing boundary marked point. By Theorem~\ref{thm:cubical-cobordisms-enhanced}, this factors through the correspondence $L \xleftarrow{} \cK_{\pi_*\Gamma} \rightarrow \pt$ and is thus degenerate. By Proposition \ref{prop:correspondences-operations}(\ref{degenerate-vanishes}) the associated operation vanishes.
	\item For the unit~\eqref{unit-correspondence}, $\scF^1_{\partial_{oc}\sigma} = \pm \scF_{D_1 \g \sigma} = \pm \scF_{D_1}(1)$, which vanishes by a similar argument as for the co-unit.
	\item For the pairing~\eqref{pairing-correspondence}, $\scF^1_{\partial_{oc}\sigma}(\alpha) = \scF_{\sigma \g (\idsimp \times D_1)}(\alpha) + \scF_{\sigma \g (D_1 \times \idsimp)}(\alpha)$. For each $\Gamma$ contributing to $\scF_{D_1}$, we get two contributions to the above, given by the correspondences 
	\begin{equation*}
	L \times L \xleftarrow{ev^-_\Gamma \times ev^+_\Gamma} \cK_\Gamma \rightarrow\pt\qquad \text{and}\qquad L \times L \xleftarrow{ev^+_\Gamma \times ev^-_\Gamma} \cK_\Gamma \rightarrow\pt.
	\end{equation*}
	Consider the map $\tau: \cK_\Gamma \rightarrow \cK_\Gamma$ exchanging the two marked points. This is an isomorphism with sign $\sign (\tau) = -1$. Whence the two contributions cancel and $\scF^1_{\partial_{oc}\sigma} = 0$.
	\item For the co-pairing~\eqref{copairing-correspondence}, $\scF^1_{\partial_{oc}\sigma} = -\scF_{(\idsimp \times D_1) \g \sigma}- \scF_{(D_1 \times \idsimp) \g \sigma}$. For each $\Gamma$ contributing to $\scF_{D_1}$, we get two contributions to the above, given by the correspondences 
	\begin{equation*}
		\pt \leftarrow \cK_\Gamma \xrightarrow{ev^-_\Gamma \times ev^+_\Gamma} L \times L\qquad \text{and}\qquad \pt \leftarrow \cK_\Gamma \xrightarrow{ev^+_\Gamma \times ev^-_\Gamma} L \times L.
		\end{equation*}
	Again, since they are related by a factor $(-1)$, the associated operations cancel.
	\end{itemize}
\end{proof}

\noindent
We now discuss the $A_\infty$ functor equation in the case of $d > 1$. To make the notation less cluttered, write $T_{\ide}$ for any (finite) disjoint union of corollas that encode the identity via Definition~\ref{de:unstable-operations}. We call $T_{\ide}$ an \emph{identity forest}. 

For any edge $e\in E(\Gamma_i)$ as in Case~\eqref{incoming-d} we have by the computation in Lemma~\ref{lem:chain-map-stable} and by Lemma~\ref{lem:gkc-for-higher-homotopies}\eqref{higher-side-one}
 	
 	\begin{equation}\label{eq:boundary-incoming}
 	(\del_e\cKc_{\lc\Gamma})_{\lc\sigma} \; = \; (-1)^{\delta_{\eqref{incoming-d}}}\,\cKc_{\Gamma_d,\dots, T\sqcup T_{\ide},\Gamma'_i,\dots,\Gamma_1,\sigma_d,\dots,\pt, \sigma_i,\dots,\sigma_1}|_{\{s_{i+1}=0\}},
 	\end{equation}
 	where
 	 
 	\begin{align*}\label{} 
 		\delta_{\eqref{incoming-d}}\;& =\; \epsilon_{\eqref{incoming-d}} + n|y_o^d|+|\lc\sigma^{> i}|'\\& =\; nk+1  +  n|y_o^d|+|\lc\sigma^{> i}|'
 	\end{align*}
	with $k$ the marked point lying on the unstable tree $T$ and
 	the last term comes from the way we orient $\cK_{\lc\Gamma,\lc\sigma}$.
 	Thus, the associated term in Equation~\eqref{eq:application-stokes} is 
 	\begin{multline}\label{eq:boundary-diff-before} 
 	(-1)^{|\alpha|+\vdim(\cK_{\lc\Gamma,\lc\sigma})}(\eva^+_{\lc\Gamma,\lc\sigma}|_{\del_e\cT_{\lc\Gamma,\lc\sigma}})_*({\eva^-_{\lc\Gamma,\lc\sigma}}^*(\alpha) \wedge p_{\lc\sigma}^* \obs_{\lc\Gamma}^*\eta_{\lc\Gamma})\\ =\; (-1)^{|\alpha|+|\lc\sigma^{\le i}|'}\fq_{(\Gamma_d,\dots,T\g_k\Gamma_i,\dots,\Gamma_1),(\sigma_d,\dots,D_1\g_k\sigma_i,\dots,\sigma_1)}(\alpha).
 	\end{multline}
 	
\noindent In Case~\eqref{outgoing-d}, we have 
 	
 	\begin{equation*}\label{eq:boundary-outgoing}
 	(\del_e\cKc_{\lc\Gamma})_{\lc\sigma}\; = \;(-1)^{\delta_{\eqref{outgoing-d}}}\cKc_{\Gamma_d,\dots,\Gamma'_i,T_{\ide}\sqcup T',\dots,\Gamma_1;\sigma_d,\dots,\sigma_i,D_1,\dots,\sigma_1}|_{\{s_i = 0\}},\end{equation*}
 	where the exponent of the sign is given by
 	\begin{equation*}\label{}
 		\delta_{\eqref{outgoing-d}}\; \equiv\; n|y^d_o|+|\lc\sigma^{> i}|' + n(|y^{i-1}_o|-k)+|\sigma_i|.
 	\end{equation*}
  	Together with the sign from~\eqref{eq:application-stokes}, it follows that
  	\begin{multline}\label{eq:boundary-diff-after} 
  		(-1)^{|\alpha|+\vdim(\cK_{\lc\Gamma,\lc\sigma})}(\eva^+_{\lc\Gamma,\lc\sigma}|_{\del_e\cT_{\lc\Gamma,\lc\sigma}})_*({\eva^-_{\lc\Gamma,\lc\sigma}}^*(\alpha) \wedge p_{\lc\sigma}^* \obs_{\lc\Gamma}^*\eta_{\lc\Gamma})\\ =\; (-1)^{|\alpha|+|\lc\sigma^{\le i}|'+|\sigma_i|+1}\fq_{(\Gamma_d,\dots,\Gamma_i\g_k T',\dots,\Gamma_1),(\sigma_d,\dots,\sigma_i\g_k D_1,\dots,\sigma_1)}(\alpha).
  	\end{multline}

  \noindent Finally, in Case~\eqref{curv-d} we have
 	\begin{align*}
 	(\del_e\cKc_{\lc\Gamma})_{\lc\sigma} \; &=\; (-1)^{\delta_{\eqref{curv-d}}}\,\cKc_{\Gamma_d, \dots, \Gamma_{i+1},\pi_j^*\Gamma'_i,T_{\ide}\sqcup T'',\Gamma_{i-1},\dots,\Gamma_{1}}|_{\{s_i =0\}},\\
 	&=\; (-1)^{\delta_{\eqref{curv-d}}}\,\cKc_{\Gamma_d, \dots, \Gamma_{i+1},\pi_j^*\Gamma'_i\g_* T'',\Gamma_{i-1},\dots,\Gamma_{1}},
 	\end{align*}
 	where the left hand side is equipped with the boundary orientation,
	\begin{equation*}\label{} 
		\delta_{\eqref{curv-d}} \;=\;1+ n|y^d_o|+|\lc\sigma^{\ge i}|'
	\end{equation*}
 	and $\pi_j^*\Gamma_i'$ is the graph $\Gamma_i'$ with an additional boundary marked point `after' the boundary marked point labelled by $j$. 
 	This sign follows from the same argument as was used to compute the exponent $\epsilon_{\eqref{curv-d}}$ in Equation~\eqref{eq:sign-curvature-basic}. However, the horizontal map in the diagram above said equation comes with the sign $(-1)^{n|y^d_o|+1+\sum_{j=i}^{d}|\sigma_j|'}$ because the isomorphism~\eqref{eq:orientation-of-lift} has this sign. Thus, the corresponding operation is 
 	\begin{multline}\label{eq:sign-curvature} 
 		(-1)^{|\alpha|+\vdim(\cK_{\lc\Gamma,\lc\sigma})}(\eva^+_{\lc\Gamma,\lc\sigma}|_{\del_e\cT_{\lc\Gamma,\lc\sigma}})_*({\eva^-_{\lc\Gamma,\lc\sigma}}^*(\alpha) \wedge p_{\lc\sigma}^* \obs_{\lc\Gamma}^*\eta_{\lc\Gamma})\\ =\; (-1)^{|\alpha|+|\lc\sigma^{\le i}|'+1}\fq_{(\Gamma_d,\dots,\pi^*\Gamma_i\g_*T'',\dots,\Gamma_1),(\sigma_d,\dots,(-1)^{|\sigma_i|}\pi^*\sigma_i\g_*D_0,\dots,\sigma_1)}(\alpha).
 	\end{multline}
Using the definition of the $A_\infty$ structure on $\cdmc_n$ from Definition~\ref{de:dg-to-a-infinity}, the relations of~\eqref{eq:warped-relations-again} become
 \begin{multline*}
 	(-1)^{|\scF^d_{\lc \sigma}|}(d_{y^d}\g\scF^d_{\lc \sigma}-(-1)^{|\scF^d_{\lc\sigma}|}\scF^d_{\lc\sigma}\g d_{y^0}) + 
 	\sum_{i=1 }^{d-1}{(-1)^{|\lc\sigma^{\leq i}|'+1}\scF^{d-i}_{\sigma_d,\dots,\sigma_{i+1}}\g\scF^i_{\sigma_i,\dots,\sigma_1}}\\ 
 	+\sum_{i=1}^{d}(-1)^{|\lc\sigma^{\leq i}|'}\scF^d_{\sigma_d,\dots,\partial\sigma_{i}-D_1\gt \sigma + (-1)^{|\sigma|}\sigma \gt D_1 + (-1)^{|\sigma|}\pi^*\sigma\g_* D_0,\dots,\sigma_1} +\sum_{i=1}^{d-1}{(-1)^{|\lc\sigma^{\leq i}|'}\scF^{d-1}_{\sigma_d,\dots,\sigma_{i+1}\g \sigma_{i},\dots,\sigma_1}} = 0.
 \end{multline*}
  On the other hand, multiplying Equation~\eqref{eq:application-stokes} with $(-1)^{|\alpha|}$ yields
 \begin{equation*}\label{eq:stokes-rearranged}
 	(-1)^{|\scF^d_{\lc\sigma}|}(d_{y^d}\fq_{\lc\Gamma,\lc\sigma}(\alpha)\ -(-1)^{|\scF^d_{\lc\sigma}|}\ \fq_{\lc\Gamma,\lc\sigma}(d_{y^0}\alpha))\ +\ (-1)^{1+\vdim(\cK_{\lc\Gamma,\lc\sigma})}(\eva^+_{\lc\Gamma,\lc\sigma}|_{\del\cT_{\lc\Gamma,\lc\sigma}})_*({\eva^-_{\lc\Gamma,\lc\sigma}}^*(\alpha) \wedge p_{\lc\sigma}^* \obs_{\lc\Gamma}^*\eta_{\lc\Gamma}) \ = \ 0.
 \end{equation*}
 Thus, summing over all suitable sequences $\lc\Gamma$ of graphs, the last term becomes
\begin{equation*}
	 \sum_{i=1}^{d-1}(-1)^{|\lc\sigma^{\le i}|'+1}\scF^{d-i}_{\lc\sigma^{> i}}(\scF^i_{\lc\sigma^{\le i}}(\alpha)) +  \sum_{i=1}^{d}(-1)^{|\lc\sigma^{\leq i}|'}\scF^d_{(\sigma_d,\dots,\del_{oc}\sigma,\dots,\sigma_1)}(\alpha)+
	 \sum_{i=1}^{d-1}(-1)^{|\lc\sigma^{\le i}|'}\scF^{d-1}_{*_i\lc\sigma}(\alpha)
\end{equation*}
by Equation~\eqref{eq:sign-composition-functors} for the first term, Equations~\eqref{eq:sign-usual-differential-chain},\eqref{eq:boundary-diff-before}, \eqref{eq:boundary-diff-after} and~\eqref{eq:sign-curvature} for the second term and Equation~\eqref{eq:sign-composition-chains} for the last term.
\end{proof} 

\subsection{Symmetric monoidal structure}
We can now complete the proof of Theorem~\ref{thm:dmft-exists} by showing that $\scF$ admits a symmetric monoidal structure. The strategy of the proof is to define the required natural transformations using certain loci inside the hypercubes $I^{d-1}$ over which the (global Kuranishi charts defining the) higher homotopies live.

\begin{proposition}
	\label{prop:dmft-symmetric-monoidal}
	The $A_\infty$ functor $\scF_L$ is unital symmetric monoidal.
\end{proposition}

\begin{proof} The first observation is that for our curved open-closed DMFT, the terms in Equation~\eqref{eq:F-tensor} and Equation~\eqref{eq:tensor-F} are given by correspondences obtained from certain loci inside the product $\cK_{\lc\Gamma,\lc\sigma} \times \cK_{\lc\Gamma',\lc \tau}$, which is equipped with a canonical map to $I^{2(d-1)}$.
	To this end, define the subspaces
	\begin{align*}
		\Delta_k &= \{0 \leq s_k \leq t_k \leq 1\}\\
		D_k &= \{0 \leq s_k = t_k \leq 1\}\\
		C_k &= \{s_k = 0\} \cup \{t_k = 1\}\\
		\bm{1}_k &= \{s_k = t_k = 1\}\\
		\bm{0}_k &= \{s_k = t_k = 0\}
	\end{align*}
	of $I^2$ with coordinates $s_k,t_k$, which are real analytic submanifolds with corners.
	In particular, 
	\begin{align}\label{eq:boundary-locus}
		\notag\partial \Delta_k &= C_k \sqcup D_k\\
		\partial C_k = \partial D_k &= \bm{1}_k \sqcup \bm{0}_k.
	\end{align}
	
	\begin{nota}
		Given an object $y$ of $\cdmc_n$, we write $\fo(y) = \fo(X^{y_c}\times L^{y_o})$.
	\end{nota}
	
	\begin{definition}\label{de:symmetric-monoidal-transformation} 
		The pre-natural transformation $R^d$ is defined by 
		\[R^0(\alpha\otimes \beta) \;=\; (-1)^{|\alpha||\beta|}\alpha\times \beta,\]
		where $\times$ is the usual cross product of differential forms, by $R^1 = 0$, and for $d \ge 2$ by
		\[R^d(\lc\sigma\otimes\lc\tau)(\alpha\otimes \beta) = (-1)^{|\alpha||\beta|}\sum_{\lc\Gamma,\lc\Gamma'}c_{\lc\Gamma}c_{\lc\Gamma'}\fq_{\cR_{\lc\Gamma,\lc \sigma;\lc\Gamma',\lc \tau}}(\alpha\times \beta)\]
		where $\cR_{\lc\Gamma,\lc \sigma;\lc\Gamma',\lc \tau} = \sqcup_i\cR_{\lc\Gamma,\lc \sigma;\lc\Gamma',\lc \tau}(i)$ is the disjoint union of the correspondences
		\begin{equation*}
			\cR_{\lc\Gamma,\lc \sigma;\lc\Gamma',\lc \tau}(i)\, :=\,  \left(\cK_{\lc\Gamma,\lc\sigma} \times \cK_{\lc\Gamma',\lc\tau}\right) \times_{I^{2(d-1)}} \left( D_{d-1} \times \dots \times D_{i+1} \times \Delta_i \times C_{i-1} \times \dots \times C_1 \right)
		\end{equation*}
		for $1\le i \le d-1$, oriented by the isomorphism
		\begin{align}\label{} 
			\orl(\cR_{\lc\Gamma,\lc \sigma;\lc\Gamma',\lc \tau}(i))
			&\; \cong \orl(x^d)\orl(y^d)\orl(\sigma_d)\orl(\tau_d)\orl(D_{d-1})\dots\orl(\tau_{i+1})\orl(t_{i})\orl(\sigma_i) \dots \orl(C_1)\orl(\sigma_1)\orl(\tau_1)\orl(s_i).
		\end{align}
	\end{definition}
	
	\noindent
	Lemma~\ref{lem:right-sym-mon} asserts that $R$ is a natural transformation. 
	To simplify the notation, abbreviate
	\begin{equation*}\label{}
		\scG(\lc \sigma, \lc \tau) \,\coloneqq\, \otimes_{\text{Ch}_\Lambda} \g (\scF \otimes \scF)(\lc \sigma, \lc \tau) 
	\end{equation*}
	and
	\begin{equation*}\label{}
		\scH(\lc \sigma, \lc \tau)\,\coloneqq\, \scF \g \otimes_{\cdmc} (\lc \sigma, \lc \tau) =\s{\lc\Gamma,\lc\Gamma'}{c_{\lc\Gamma}c_{\lc\Gamma'}}\,\fq_{\lc\Gamma\sqcup\lc\Gamma',\lc\sigma\times\lc\tau}
	\end{equation*}
	for composable sequences in $\cdmc_n$. Roughly, these functors are induced by the operations associated to the following correspondences.
	
	\begin{definition}\label{def:correspondences-for-tensor-product}
		Define the correspondences
		\[
		\cD_{\lc\Gamma,\lc \sigma;\lc\Gamma',\lc \tau} \,= \,\left(\cK_{\lc\Gamma,\lc \sigma} \times \cK_{\lc\Gamma',\lc \tau}\right) \times_{I^{2(d-1)}} \lbr{D_{d-1} \times \dots \times D_1}
		\] 
		oriented via the isomorphism
		\begin{equation}\label{} 
			\orl(\cD_{\lc\Gamma,\lc \sigma;\lc\Gamma',\lc \tau})\,\cong\, \orl(x^d)\orl(y^d)\orl(\sigma_d)\orl(\tau_d)\orl(D_{d-1})\cdots \orl(D_1)\orl(\sigma_1)\orl(\tau_1),
		\end{equation}
		and
		\[
		\cC_{\lc\Gamma,\lc \sigma;\lc\Gamma',\lc \tau} := \left(\cK_{\lc\Gamma,\lc \sigma} \times \cK_{\lc\Gamma',\lc \tau}\right) \times_{I^{2(d-1)}} \left( C_{d-1} \times \dots \times C_1 \right),
		\]
		oriented by 
		\begin{equation}\label{}
			\orl(\cC_{\lc\Gamma,\lc \sigma;\lc\Gamma',\lc \tau})\,\cong\, \orl(x^d)\orl(y^d)\orl(\sigma_d)\orl(\tau_d)\orl(C_{d-1})\cdots \orl(C_1)\orl(\sigma_1)\orl(\tau_1).
		\end{equation}
	\end{definition}
	
	\begin{lemma}\label{lem:described-by-locus}
		As maps $\Omega^*(X^{x^0_c}\times L^{x^0_o};\Lambda)\otimes_\Lambda \Omega^*(X^{y^0_c}\times L^{y^0_o};\Lambda) \rightarrow \Omega^*(X^{x^d_c + y^d_c}\times L^{x^d_o + y^d_o};\Lambda)$, we have that 
		\begin{equation}
			\label{eq:locus-tensor-F}
			\scH(\lc \sigma, \lc \tau) \g R^0_{x^0,y^0} (\alpha\otimes\beta) = (-1)^{|\alpha||\beta|} \s{\lc\Gamma,\lc\Gamma'}{c_{\lc\Gamma}c_{\lc\Gamma'}}\,\fq_{\cD_{\lc\Gamma,\lc \sigma;\lc\Gamma',\lc \tau}}(\alpha \times \beta)
		\end{equation}
		and 
		\begin{equation}
			\label{eq:locus-F-tensor-F}
			R^0_{x^d,y^d} \g \scG(\lc \sigma, \lc \tau)(\alpha\otimes\beta) = (-1)^{|\alpha||\beta|} \s{\lc\Gamma,\lc\Gamma'}{c_{\lc\Gamma}c_{\lc\Gamma'}}\,\fq_{\cC_{\lc\Gamma,\lc \sigma;\lc\Gamma',\lc \tau}}(\alpha \times \beta)
		\end{equation}
		More generally
		\begin{equation*}
			R^{d-j}(\lc \sigma^{>j},\lc \tau^{>j}) \g \scG^j(\lc \sigma^{\leq j}, \lc \tau^{\leq j}) = (-1)^{|\alpha||\beta|}\s{\lc\Gamma,\lc\Gamma'}{c_{\lc\Gamma}c_{\lc\Gamma'}} \fq_{\cR_{\lc\Gamma^{>j},\lc\sigma^{> j};(\lc\Gamma')^{>j},\lc\tau^{> j}}} \g \fq_{\cC_{\lc\Gamma^{\le j},\lc\sigma^{\le j};(\lc\Gamma')^{\le j},\lc\tau^{\le j}}}(\alpha \times \beta)
		\end{equation*}
	\end{lemma}
	
	\begin{proof}
		The map $\cKc_{\lc\Gamma}\to I^{d-1}$ is a submersion due to the fact that the parametrised fibre products in the definition of $\cKc_{\lc\Gamma}$ are by submersions. Thus, the fibre products in Definition~\ref{def:correspondences-for-tensor-product} are well-defined correspondences required in Definition~\ref{de:operation-from-correspondence} and we can associate the operation $\fq$ to it.
		The first claim then follows from the construction of the higher homotopies (Lemma~\ref{lem:gkc-for-higher-homotopies}), the way they are oriented, cf.~\eqref{eq:orientation-higher}, and the definition of $R^0$.
		To prove the second claim, fix sequences $\lc\Gamma$ and $\lc\Gamma'$ of graphs lying over $\lc\sigma$ and $\lc\tau$ respectively. We have to show \begin{equation*}
			\sum_{\lc p} (-1)^{\dagger(|\lc \sigma + \lc p|, |\lc \tau + \lc p^c|)} R^0_{x^d,y^d} \g \fq_{\concat_{\lc p}\lc\Gamma,\concat_{\lc p}\lc\sigma} \otimes_{\scD} \fq^{\g\lc p}_{\lc\Gamma',\lc\tau}(\alpha\otimes\beta) = (-1)^{|\alpha||\beta|}\fq_{\cC_{\lc\Gamma,\lc \sigma;\lc\Gamma',\lc \tau}}(\alpha \times \beta)
		\end{equation*}
		where we define $ \fq^{\g\lc p}_{\lc\Gamma',\lc\tau}$ analogous to how we defined $\scF^{\g\lc p}$.
		Fix a sequence $\lc p = (0= p_0 <p_{1} < \dots < p_r < p_{r+1} = d)$. Define $\lc \tau^{(i)} := (\tau_{p_{i}}, \dots, \tau_{p_{i-1} + 1})$ for $i \leq r+1$ and similarly for $\lc \Gamma'$. Then, $\fq^{\g\lc p}_{\lc\Gamma',\lc\tau}$ is the operation associated to the correspondence
		\[\cK^{\g \lc p}_{\lc \Gamma',\lc \tau} \;:=\; \cK_{\lc \Gamma'^{(r+1)},\lc \tau^{(r+1)}} \times_{Y^{y_{p_r}}} \dots \times_{Y^{y_{p_1}}} \cK_{\lc \Gamma'^{(1)},\lc \tau^{(1)}}\]
		carrying the cross product of Thom forms.\par 
		On the other hand, set $\cC^{\lc p}_{\lc\Gamma,\lc \sigma;\lc\Gamma',\lc \tau} := \left(\cK_{\lc\Gamma,\lc \sigma} \times \cK_{\lc\Gamma',\lc \tau}\right) \times_{I^{2(d-1)}} C^{\lc p}$, where
		\[C^{\lc p} := \set{x \in C_{d-1}\times \dots\times C_1\mid \forall i \in \{p_1,\dots,p_r\}: t_i =1,\;\forall j \notin  \{p_1,\dots,p_r\}:s_j = 0}.\] 
		Since the canonical isomorphism $\orl(X\times_Z Y)\cong \orl(X)\orl(Z)\dul\orl(Y)$ yields the fibre product orientation, the isomorphism
		\begin{equation*}
			\cK_{\star_{\lc p} \lc \Gamma, \star_{\lc p} \lc \sigma} \times \cK^{\g \lc p}_{\lc \Gamma', \lc \tau} \;\cong\; \cC^{\lc p}_{\lc\Gamma,\lc \sigma;\lc\Gamma',\lc \tau}
		\end{equation*}
		has orientation sign $(-1)^{|y^d||\scF_{\star_{\lc p}\lc \sigma}| + \dagger(|\lc \sigma + \lc p|, |\lc \tau + \lc p^c|)}$. 
		Therefore,
		\begin{align*}
			R^0_{x^d,y^d} \g \left( \scF(\concat_{\lc p}\lc\sigma) \otimes_{\scD} \scF^{\g\lc p}(\lc\tau)\right)(\alpha\otimes\beta)
			&= R^0_{x^d,y^d} \g \left(\fq_{\cK_{\star_{\lc p} \lc \Gamma, \star_{\lc p} \lc \sigma}} \otimes \fq_{\cK^{\g \lc p}_{\lc \Gamma', \lc \tau}}\right)(\alpha\otimes\beta)\\
			&= (-1)^{\epsilon_1}\fq_{\cK_{\star_{\lc p} \lc \Gamma, \star_{\lc p} \lc \sigma}}(\alpha) \times \fq_{\cK^{\g \lc p}_{\lc \Gamma', \lc \tau}}(\beta).
		\end{align*}
		where $\epsilon_1 = (|\scF_{\star_{\lc p}\lc \sigma}|+|\alpha|)(|\scF^{\g \lc p}_{\lc \tau}| + |\beta|) + |\scF^{\g \lc p}_{\lc \tau}||\alpha|$.
		It follows from the properties of the pushforward, that 
		\[(f\times g)_*(\alpha\times \beta) = (-1)^{\text{rdim}(f)(\dim Z-|\beta|)}f_*(\alpha)\times g_*(\beta)\]
		for any submersions $f$ and $g\cl Z\to Z'$. We thus find that 
		\begin{align}
			\fq_{\cK_{\star_{\lc p} \lc \Gamma, \star_{\lc p} \lc \sigma}}(\alpha) \times \fq_{\cK^{\g \lc p}_{\lc \Gamma', \lc \tau}}(\beta) &= (-1)^{\epsilon_2}\fq_{\cK_{\star_{\lc p} \lc \Gamma, \star_{\lc p} \lc \sigma} \times \cK^{\g \lc p}_{\lc \Gamma', \lc \tau}} (\alpha \times \beta),
		\end{align}
		where $\epsilon_2 = |\scF_{\star_{\lc p}\lc \sigma}|(|y_d| + |\scF^{\g \lc p}_{\lc \tau}| + |\beta|)$. Combining all previous equations yields the second equality. The third equation follows by a very similar computation.
	\end{proof}

\begin{lemma}\label{lem:right-sym-mon}
	$R$ defines a natural transformation of degree $0$ from $\scG$ to $\scH$.
\end{lemma}

\begin{proof}
	Abbreviate $\lc a := (\sigma_d\otimes\tau_d,\dots,\sigma_1\otimes \tau_1)$. Then, Equation~\eqref{eq:fun-differential} becomes
	\begin{align}\label{eq:simplied-differential}
		\notag(-1)^{|\lc a|'}\lbr{d_{x^d\sqcup y^d} \g R^d(\lc a)-(-1)^{|\lc a|'}&R^d(\lc a)\g d_{x^0,y^0}}\;=\;-\sum_{j= 0}^d (-1)^{|\lc a^{\leq j}|'} \scH(\lc a^{>j}) \g R^j(\lc a^{\leq j})\\
		&\notag +\sum_{j = 0}^{d} R^{d-j}(\lc a^{>j}) \g \scG(\lc a^{\leq j}) 
		+ \sum_{j=1}^{d-1} (-1)^{|\lc a^{\le j}|'}R^{d-1}(\star_j \lc a) \\
		&	+ \sum_{j = 1}^d (-1)^{|\lc a^{\leq j}|'}\lbr{R^d((\partial^{oc}_j \lc \sigma)\otimes \lc \tau) + (-1)^{|\sigma_j|}R^d(\lc \sigma\otimes (\partial^{oc}_j \lc \tau))} = 0.
	\end{align}
	Since $d_y(\alpha) = (-1)^{|\alpha|}d\alpha$ by definition, $R^0$ is a chain map, which shows the claim for $d = 0$. 
	For $d = 1$, the equation becomes
	$$ \scH(\sigma\otimes\tau) \g R^0 = R^0\g \scG(\sigma\otimes\tau).$$\noindent
	Let $\lc \Gamma$ and $\Gamma'$ be graphs as `lying' over the respective simplices $\sigma$ and $\tau$. Given the orientations determined by the isomorphism~\eqref{eq:orientation-higher}, we have
	\begin{equation*}\label{} 
		\cK_{\lc\Gamma\sqcup \Gamma',\sigma\times\rho} \;=\; (-1)^{n|y'_o||\sigma|}\cK_{\lc\Gamma,\lc\sigma}\times\cK_{\Gamma',\rho}.
	\end{equation*}
	It follows from the properties of the pushforward, that 
	\[(f\times g)_*(\alpha\times \beta) = (-1)^{\text{rdim}(f)(\dim Z-|\beta|)}f_*(\alpha)\times g_*(\beta)\]
	for any submersions $f$ and $g\cl Z\to Z'$.
	Thus, 
	\begin{align*}
		\fq_{\Gamma\sqcup \Gamma',\sigma\times\rho}(R^0_{x,y}(\alpha\otimes\beta)) &= (-1)^{|\alpha||\beta|}\fq_{\Gamma\sqcup \Gamma',\sigma\times\rho}(\alpha\times\beta)\\
		&= (-1)^{|\alpha||\beta|+n|y'_o||\sigma| }(\fq_{\Gamma,\sigma}\times\fq_{\Gamma',\rho})(\alpha\times\beta)\\
		&=(-1)^{|\alpha||\beta|+(|\rho| +|\beta|) \sigma|}\fq_{\Gamma,\sigma}(\alpha)\times\fq_{\Gamma',\rho}(\beta)\\
		&= (-1)^{|\alpha||\beta|+(|\rho| +|\beta|)|\sigma|+(|\alpha|+|\sigma|)(|\beta|+|\rho|)}R^0_{x',y'}(\fq_{\Gamma,\sigma}(\alpha)\otimes\fq_{\Gamma',\rho}(\beta))\\
		&= R^0_{x',y'} \g (\fq_{\Gamma,\sigma} \otimes \fq_{\Gamma',\rho})(\alpha \otimes \beta)
	\end{align*}
	as required.\\
	
	\noindent
	Suppose now $d \ge 2$ and fix $1 \le i \le d$ as well as sequences of graphs $\lc\Gamma$ lying over $\lc\sigma$ and $\lc\Gamma'$ lying over $\lc\tau$. By Stokes' Theorem, we have 
	\begin{equation}\label{eq:stokes-symm-mon} 
		(-1)^{|\gamma|}d\fq_{\cR^d_{\lc\Gamma,\lc \sigma;\lc\Gamma',\lc \tau}}(\gamma)\;=\; (-1)^{|\gamma|}\fq_{\cR^d_{\lc\Gamma,\lc \sigma;\lc\Gamma',\lc \tau}}(d\gamma)\, +\, (-1)^{\vdim(\cR^d_{\lc\Gamma,\lc \sigma;\lc\Gamma',\lc \tau})}\fq_{\del\cR^d_{\lc\Gamma,\lc \sigma;\lc\Gamma',\lc \tau}}(\gamma).
	\end{equation}
	The boundary is given by
	\begin{align}\label{eq:boundary-natural-transformation} 
		\notag\del \cR_{\lc\Gamma,\lc \sigma;\lc\Gamma',\lc \tau}(i) \; &=\; \del^v(\cK_{\lc\Gamma,\lc\sigma}\times\cK_{\lc\Gamma',\lc\tau})\times_{I^{2(d-1)}}\left( D_{d-1} \times \dots \times D_{i+1} \times \Delta_i \times C_{i-1} \times \dots \times C_1 \right)\\&\sqcup 
		(\cK_{\lc\Gamma,\lc\sigma}\times\cK_{\lc\Gamma',\lc\tau})\times_{I^{2(d-1)}}\del\lbr{ D_{d-1} \times \dots \times D_{i+1} \times \Delta_i \times C_{i-1} \times \dots \times C_1}
	\end{align}
	as unoriented correspondences by \cite{Joy}. The first component of this disjoint union is similar to the terms determined in \S\ref{subsec:warped-equations-proof} and it yields the last three terms of~\eqref{eq:simplied-differential}. Observe that \[\vdim(\cR_{\lc\Gamma,\lc\sigma;\lc\Gamma',\lc\tau}) \equiv \vdim(\cK_{\lc\Gamma\sqcup\lc\Gamma',\lc a}) + 1 \equiv |\lc a|'+ n|x^d_o|+n|y^d_o|,\]
	which in particular implies that $R$ has degree $0$. Summing over all suitable sequences of graphs, the last term in Equation~\eqref{eq:stokes-symm-mon} gives rise to the following operations, due to the way we orient $\cR^d_{\lc\Gamma,\lc\sigma;\lc\Gamma',\lc\tau}$.
	\begin{enumerate}[label=\Roman*),leftmargin=20pt,ref=\Roman*]
		\item\label{boundary-simplex-sym} Using the same computation as below Equation~\eqref{eq:boundary-fibre-product}, we see that the first component in~\eqref{eq:boundary-natural-transformation} yields one one hand the operations $(-1)^{|\lc a^{\le j}|'}R^d(\del_j\lc a)$ from the boundary of the simplices $\sigma_j$ and $\tau_j$. On the other hand, it yields the terms 
		\begin{itemize}[leftmargin=20pt]
			\item $(-1)^{|\lc a^{\le j}|'}R^d((- D_1 \gt_j\lc\sigma)\otimes\lc\tau)$ and $(-1)^{|\lc a^{\le  j}|'+|\sigma_j|}R^d(\lc\sigma\otimes (- D_1 \gt_j\lc\tau))$ by the same computation as before Equation~\eqref{eq:boundary-diff-before}	\item $(-1)^{|\lc a^{\le j}|'}R^d((-1)^{|\sigma_j|}(\lc\sigma\gt_j D_1 )\otimes\lc\tau)$ and $(-1)^{|\lc a^{\le  j}|'+|\sigma_j|}R^d(\lc\sigma\otimes ((-1)^{|\tau_j|}\lc\tau\gt_j D_1))$ by the same computation as before Equation~\eqref{eq:boundary-diff-after}
			\item $(-1)^{|\lc a^{\le j}|'}R^d((\sigma_d,\dots,(-1)^{|\sigma_j|}\pi^*\sigma_j\g_*D_0,\dots,\sigma_1)\otimes\lc\tau)$ as well as \\ $(-1)^{|\lc a^{\le j}|'+|\sigma_j|}R^d(\lc\sigma\otimes(\tau_d,\dots,(-1)^{|\tau_j|}\pi^*\tau_j\g_*D_0,\dots,\tau_1))$ by the same computation as before Equation~\eqref{eq:sign-curvature}.
		\end{itemize}
		The remaining operations associated to boundary strata in $\del^v(\cK_{\lc\Gamma,\lc\sigma}\times\cK_{\lc\Gamma',\lc\tau})$ cancel by Lemma~\ref{lem:total-degree-contraction}.
		\item\label{boundary-D} We now turn to the remaining components of the boundary. For $j > i$, we have $\del D_j =\textbf{1}_j- \textbf{0}_j$ by Equation~\eqref{eq:boundary-locus} so that  	
		\begin{align*}\label{}
			&(\cK_{\lc\Gamma,\lc\sigma}\times\cK_{\lc\Gamma',\lc\tau})\times_{I^{2(d-1)}}\lbr{ D_{d-1} \times \dots\del D_j\times\dots \times D_{i+1} \times \Delta_i \times C_{i-1} \times \dots \times C_1} \\&= (\cK_{\lc\Gamma,\lc\sigma}|_{\{s_j =1\}}\times\cK_{\lc\Gamma',\lc\tau}|_{\{t_j =1\}})\times_{I^{2(d-2)}}\lbr{ D_{d-1} \times\dots \wh{D_j}\times\dots \times D_{i+1} \times \Delta_i \times C_{i-1} \times \dots \times C_1}\\&\;\;\sqcup (\cK_{\lc\Gamma,\lc\sigma}|_{\{s_j =0\}}\times\cK_{\lc\Gamma',\lc\tau}|_{\{t_j =0\}})\times_{I^{2(d-2)}}\lbr{ D_{d-1} \times\dots \wh{ D_j}\times\dots \times D_{i+1} \times \Delta_i \times C_{i-1} \times \dots \times C_1}\\
			&=\, \cD_{\lc\Gamma^{>j},\lc\sigma^{> j};(\lc\Gamma')^{>j},\lc\tau^{> j}}\times_{Y_{x^j}\times Y_{y^j}}\cR_{\lc\Gamma^{\le j},\lc\sigma^{\le j};(\lc\Gamma')^{\le j},\lc\tau^{\le j}}(i)
			\,\sqcup\, 
			\cR_{\concat_j\lc\Gamma,\concat_j\lc\sigma;\concat_j\lc\Gamma',\concat_j\lc\tau}(i)
		\end{align*}
		as unoriented correspondences. The sign in front of the first operation is $(-1)^{\epsilon_1}$, where
		\begin{align*}\label{} 
			\epsilon_1 &\equiv |\lc a|'+ n|x^d_o|+n|y^d_o| + |\lc a^{> j}|'+1+ n|x^d_o|+n|y^d_o|\mod{2}\\
			&\equiv |\lc a^{\le j}|'+1.
		\end{align*}
		The second operation on the right hand side appears with sign $(-1)^{|\lc a^{\le j}|'}$,
		where the extra $+1$ is from the fact that it occurs as $\textbf{0}_j$ occurs with sign $-1$ in $\del D_j$.
		\item\label{boundary-C} For $j < i$, we have similarly $\del C_j = \textbf{1}_j-\textbf{0}_j$ so that 
		\begin{align*}\label{}
			(\cK_{\lc\Gamma,\lc\sigma}&\times\cK_{\lc\Gamma',\lc\tau})\times_{I^{2(d-1)}}\lbr{ D_{d-1} \times \dots \times D_{i+1} \times \Delta_i \times C_{i-1} \times \dots\times \del C_j\times \dots\times C_1} \\
			&=\, \cR_{\lc\Gamma^{>j},\lc\sigma^{> j};(\lc\Gamma')^{>j},\lc\tau^{> j}}(i)\times_{Y_{x^j}\times Y_{y^j}}\cC_{\lc\Gamma^{\le j},\lc\sigma^{\le j};(\lc\Gamma')^{\le j},\lc\tau^{\le j}}
			\,\sqcup\, 
			\cR_{\concat_j\lc\Gamma,\concat_j\lc\sigma;\concat_j\lc\Gamma',\concat_j\lc\tau}(i)
		\end{align*}
		as unoriented correspondences. The sign for the first operation is $(-1)^{\epsilon_1}$ with
		\begin{equation*}\label{} 
			\epsilon_1 \,\equiv\, |\lc a^{\le j}|' + 1 + |\lc a^{\le j}|' + 1\,\equiv\, 0.
		\end{equation*}
		The term $|\lc a^{\le j}|' + 1$ comes from the fact that we have to move $\orl(s_i)$ forward appear before $\orl(\sigma_j)\orl(\tau_j)$. The sign in front of the second operation is $(-1)^{ |\lc a^{\le j}|'}$.
		\item\label{boundary-Delta} The final component is 
		\begin{equation*}\label{eq:remaining-terms*}
			(\cK_{\lc\Gamma,\lc\sigma}\times\cK_{\lc\Gamma',\lc\tau})\times_{I^{2(d-1)}}\lbr{ D_{d-1} \times \dots \times D_{i+1} \times (C_i\sqcup D_i) \times C_{i-1} \times \dots \times C_1} 
		\end{equation*}
		This yields the remaining terms by Lemma~\ref{lem:described-by-locus}. To understand, the sign we use the coordinates $s_i$ and $t_i$ to orient $\Delta_i\sub \bR^2$. Observe that the canonical isomorphism 
		\[\quad\orl(\cR^d_{\lc\sigma,\lc\tau}(i))\;\cong \; \orl(x^d)\orl(y^d)\orl(\sigma_d)\orl(\tau_d)\orl(D_{d-1})\dots\orl(\tau_{i+1})\orl(s_i)\orl(t_i)\orl(\sigma_i) \orl(\tau_i) \dots \orl(C_1)\orl(\sigma_1)\orl(\tau_1)\]
		has sign $(-1)^{|\lc a^{\le i}|'+1+1}$.
		The normal vectors to the one-dimensional strata of $C_i$ are given by $t_i$ and $-s_i$ respectively, while the normal vector to $D_i$ is given by $s_i - t_i$.
		The term with $D_i$ thus appears with the sign $(-1)^{\epsilon'}$ in Equation~\eqref{eq:stokes-symm-mon}, where
		\begin{equation*}\label{}
			\epsilon' \,\equiv\, |\lc a|'+ n|x^d_o|+n|y^d_o| + |\lc a^{>  i}|'+1+ n|x^d_o|+n|y^d_o| + |\lc a^{\le i}|'\,\equiv\, 1 \mod{2}.
		\end{equation*}
		Similarly the term with $C_i$ appears with the sign $1$.
		Summing over $1 \leq i \le d-1$, we are left with the operations coming from the correspondences $\cD_{\lc\Gamma,\lc \sigma;\lc\Gamma',\lc \tau}$ and $\cC_{\lc\Gamma,\lc \sigma;\lc\Gamma',\lc \tau}$, appearing
		with signs $-1$ and $1$ respectively.
	\end{enumerate}
	The result then follows by summing over all terms and using Lemma \ref{lem:described-by-locus}.	
\end{proof}

\noindent
By switching the roles of $s$ and $t$ in the definition of the loci $D_i,C_i$ and $\Delta_i$, we obtain a natural transformation $L$ between $\scF^d \g \otimes_{\scC}$ and $\otimes_{\scD} \g (\scF \ov\otimes \scF)$ as required. It follows from the construction that this is symmetric. By the same argument as in Lemma~\ref{lem:unital-functor}, one shows that the defining correspondence is degenerate whenever $\sigma_i =\idsimp$ or $\tau_i = \idsimp$ for $1 \leq i \le d$. To prove that $\scF$ is unital symmetric monoidal, let $\lc \sigma$, $\lc \tau$ be such that for each $i$ either $\sigma_i$ or $\tau_i$ is the identity morphism. By the proof of Lemma \ref{lem:unital-functor}, the correspondence $\cK_{\lc \sigma} \times \cK_{\lc \tau}$ factors through a correspondence of codimension $d$. The correspondence used to define $R^d_{\lc \sigma, \lc \tau}$ is of codimension $d-1$ in $\cK_{\lc \sigma} \times \cK_{\lc \tau}$, and thus factors through a correspondence of codimension $1$, so it is degenerate. Hence $R^d_{\lc \sigma, \lc \tau} = 0$ for such $\lc \sigma$, $\lc \tau$.
\end{proof}

\noindent
We can now prove that $\scF$ has the property required for showing that deformations by a bounding cochain lead to an uncurved open-closed DMFT.

\begin{lemma}\label{lem:additional-properties} $\scF$ satisfies Property~\ref{assum:weakly-curved-to-uncurved}.
\end{lemma}

\begin{proof}
	We only have to consider the case where $\sigma_i$ is stable or $\sigma_i = \sigma'_i\g D_j$ for $j = 0,1$.
	In the either case, Property \eqref{assum:weakly-curved-unit} follows from the fact that forgetting an incoming marked point does not change stabilisation patterns, whence 
	\[\scF_{\sigma_d,\dots,\pi^*\sigma_i,\dots,\sigma_1\times\idsimp}(R^0(\alpha\otimes e_\scF)) \;=\;\pm\, \scF_{\sigma_d,\dots,\pi_*\pi^*\sigma_i,\dots,\sigma_1}(\alpha) = 0\]
	due to the degeneracy condition. For the second property, we have to show that 
	\begin{equation}\label{eq:second-property-proof} 
		\scF^{d}_{\sigma_d,\dots,\pi^*\sigma_i,\sigma_{i-1}\times D_k,\dots,\sigma_1\times\idsimp}\g R^0\; =\;(-1)^{k|\lc\sigma^{< i}|'} \scF^{d}_{\sigma_d,\dots,\pi^*\sigma_i,\sigma_{i-1}\times D_k,\dots,\sigma_1\times\idsimp}\g R^0\g (\ide\otimes\scF_{D_k}(\gamma)).
	\end{equation}
	By Lemma~\ref{lem:gkc-for-higher-homotopies}, this reduces to showing that for graphs $\Gamma_i$ lying over $\sigma_i$ and $\Gamma$ lying over $D_k$ we have an isomorphism 
	\begin{equation}\label{eq:correspondences-to-compare} 
		\cKc_{(\Gamma_d,\dots,\pi^*\Gamma_i,\Gamma',\Gamma_{i-1},\dots,\Gamma_1)}|_{\{s_i = 1\}}\;\cong\; \cKc_{(\Gamma_d,\dots,\pi^*\Gamma_i,\dots,\Gamma_1)}\times_{L} \cKc_{\Gamma'}
	\end{equation}
	of (unoriented) correspondences, where we have suppressed the identity trees from the subscript. We may assume without loss of generality that $\Gamma_d \g \dots \g (\pi^*\Gamma_i\g_*\Gamma')\g \dots \g \Gamma_1$ is connected. Write $\lc{\wh{\Gamma}}$ for the sequence of graphs on the left hand side of~\eqref{eq:correspondences-to-compare} and $\lc{\wt{\Gamma}}$ for the sequence of graphs on the right hand side, that is, without $\Gamma'$. 
	Since $s_i = 1$, the maps $\chi^\pm_{\lc{\wh\Gamma}}$ factor through the map $p\cl I^d\to I^{d-1}$ that forgets the $s_i$ coordinate (and renumbers the other coordinates). 
	We have to show that for an incoming marked point $\kappa$ of $\Gamma_1$ and an outgoing marked point $\kappa'$ of $\Gamma_d$, we have $\chi^-_{\lc{\wh\Gamma};\kappa}(s) = \chi^-_{\lc{\wt\Gamma};\kappa}(p(s))$ and $\chi^+_{\lc{\wh\Gamma};\kappa'}(s) = \chi^+_{\lc{\wt\Gamma};\kappa'}(p(s))$. 
	Observe that one only has to consider the case where $\Gamma_i$ has a stable component since both sides of~\eqref{eq:second-property-proof} vanish otherwise. If $\Gamma_j \g \dots \g \pi^*\Gamma_i$ is a stable map graph for any $i \le j \le d$, then the claim follows from the fact that we constructed the global Kuranishi charts to be compatible with stable clutching and that $\Gamma_j \g \dots \g (\pi^*\Gamma_i\g_*\Gamma')$ is always a stable map graph.
	This is in particular the case if $\Gamma_i$ has positive energy. If $\Gamma_i$ has zero energy, then forgetting outgoing marked points (or clutching them with another marked point) does not change the stabilisation pattern because $\int_X\alpha = \int_{TX}\pi_{TX}^*\alpha \wedge \eta_X$. Forgetting incoming marked points also does not change the stabilisation patter of $\Gamma_i$ in this case and the remaining case to consider is if $\Gamma_{i-1} \g \dots \g\Gamma_1$ contains a copairing corolla that interacts with $\Gamma_i$. But then, the respective component of $\pi^*\Gamma_i \g \Gamma_{i-1}$ either remains stable or loses an incoming marked point. In either case, composing $\pi^*\Gamma_i$ with $\Gamma'$ at the added marked point does not change the stabilisation pattern, whence we obtain the isomorphism in~\eqref{eq:correspondences-to-compare}.
	The sign follows from the definition of the orientation of the higher homotopies, Equation~\eqref{eq:orientation-higher}.
\end{proof}

\subsection{Independence}
In this section we can prove independence up to homotopy.

\begin{theorem}
	\label{thm:independence-of-dmft}Let  $(X,\omega)$ be closed and $L \sub X$ be an embedded relatively spin Lagrangian.
	\begin{enumerate}[leftmargin=20pt]
		\item\label{independence-curved} The curved open-closed DMFT of Theorem~\ref{thm:dmft-exists} is independent of the choice of auxiliary data up to homotopy of curved open-closed DMFTs.
		\item\label{independence-uncurved} The open-closed DMFT of Corollary~\ref{cor:actual-dmft} is independent of the choice of the gauge equivalence class of weak bounding cochain and auxiliary data up to homotopy of open-closed DMFTs.
	\end{enumerate}
\end{theorem}

\noindent
We are only defining the notion of gauge equivalence of bounding cochains in our specific context. To this end, recall that any curved open-closed DMFT $\scF$ given by Theorem~\ref{thm:dmft-exists} satisfies $\scF_L(y) = \Omega^*(X^{y_c}\times L^{y_o};\Lambda)$. Thus, its associated curved $A_\infty$ algebra $\scA$ has underlying vector space $\Omega^*(L;\Lambda)$ and we define gauge equivalence for such $A_\infty$ algebras.

\begin{definition}
	Two weak bounding cochains $\bc_0$ for $\scA_0$ and $\bc_1$ for $\scA_1$ are \emph{gauge-equivalent} if there exists a pseudo-isotopy, that is, an $A_\infty$ structure on $\wt\scA := \Omega^*(L \times [0,1]; \Lambda)$, from $\scA_0$ to $\scA_1$ and a weak bounding cochain $\wt \bc$ for $\wt \scA$ such that $j_i^*\wt \bc = \bc_i$ for $i \in \{0,1\}$, where $j_i\cl L \hookrightarrow L \times [0,1]$ is the inclusion at $\{i\}$.
\end{definition}

\begin{proof}
	We recall first the choices that are required for the construction of the curved DMFT, in the logical dependency in which they have to be made. The first two choices are
	\begin{enumerate}
		\item an $\omega$-tame almost complex structure $J$.
		\item an unobstructed auxiliary datum $\alpha_{\mathsf{a}}$ for $\Mbar_{\mathsf{a}}^J(X,L)$ as in \cite[Definition~2.20]{HH25} for each $\mathsf{a}\in \cA$.
	\end{enumerate}
	Given this, Theorem~\ref{thm:cubical-cobordisms-enhanced} yields a system $\{\cKc_\Gamma\}_{\Gamma}$ of cubical cobordisms, with Lemma~\ref{lem:gkc-for-higher-homotopies} providing the global Kuranishi charts for the higher homotopies. The final choice is a system of compatible Thom forms, guaranteed by Proposition~\ref{prop:thom systems exist}.
	We first discuss the case where we fix some tame $J$ but choose the auxiliary data and the Thom forms arbitrarily. 
	Then, Theorem \ref{thm:dmft-exists} yields two curved open-closed DMFTs $\scF^0$ and $\scF^1$. 
	By definition, they agree on objects. 
	Write $\cKd{i}_{\lc\Gamma}$ for the global Kuranishi chart of Theorem~\ref{thm:cubical-cobordisms-enhanced}, respectively Lemma~\ref{lem:gkc-for-higher-homotopies}, used to define $\scF^i$. 
	We introduce the following definition.
	\begin{definition}\label{}
		Let $Y$ be a smooth manifold, possibly with corners. A \emph{system of cubical cobordisms over $Y$} is a system $\{\cNc_\Gamma\}_\Gamma$ of global Kuranishi charts equipped with submersions to $Y\times I^{E(\Gamma)}$ so that the properties of Theorem~\ref{thm:cubical-cobordisms-enhanced} hold with fibre products over $X^k\times L^\ell$ replaced by fibre products over $Y\times X^k\times L^\ell$.
	\end{definition}

	\begin{lemma}\label{lem:cobordism-for-homotopy}
		There exists a system $\cKd{01}_\bullet = \{\cKd{01}_\Gamma\}_\Gamma$ of cubical cobordisms over $[0,1]$ so that $\cKd{i}_\Gamma\simeq \cKd{01}_\Gamma|_{\{i\}}$ for $i \in \{0,1\}$, where these equivalences are functorial with the equivalences of the system. We can equip it with a system of Thom forms extending the Thom forms on $\cKd{0}_{\bullet}$ and $\cKd{1}_{\bullet}$.
	\end{lemma}
	
	\begin{proof}
		We first recall that $\cK^{01}_\Gamma$ is obtained from a double sum thickening constructed in \cite[Proposition~2.57]{HH25} by not adding the perturbation terms coming from $\alpha^0$ and $\alpha^1$ but taking their convex sum over $[0,1]$ as in the proof of \cite[Theorem~4.1]{HH25}, where this step is explained in detail. Then, the constructions in said proof carry over verbatim, except for replacing fibre products over $X^k\times L^\ell$ by fibre products over $I\times X^k\times L^\ell$.
	\end{proof}
	
	\noindent
	 The proof of Lemma~\ref{lem:gkc-for-higher-homotopies} now carries over verbatim to give correspondences $\cKd{01}_{\lc\Gamma}$ equipped with submersion to $I\times I^{d-1}$ as well as Thom forms for composable sequences $\lc\Gamma = (\Gamma_d,\dots,\Gamma_1)$ of graphs.
	 However, the restriction $\cKd{01}_{\lc \Gamma}|_{\{s_i=1\}}$ is given by
	\begin{equation*}
		\cKd{01}_{\lc\Gamma^{> i}} \times_{I\times X^{y^i_c}\times L^{y^i_o}} \cKd{01}_{\lc\Gamma^{\le i}}
	\end{equation*}
	instead of
	\begin{equation}\label{eq:required-boundary}
		\cK^{1}_{(\Gamma_d, \dots, \Gamma_{i+1})} \times_{Y^+_{\Gamma_i}} \cKd{01}_{(\Gamma_i, \dots, \Gamma_1)}\; \sqcup\; \cKd{01}_{(\Gamma_d, \dots, \Gamma_{i+1})} \times_{Y^+_{\Gamma_i}} \cK^{0}_{(\Gamma_i, \dots, \Gamma_1)},
	\end{equation}
	which would be the required boundary structure for a homotopy of open-closed DMFTs.
	However, these global Kuranishi charts also appear as boundary strata of
	\begin{equation}
		(\cKd{01}_{\lc\Gamma^{> i}} \times_{Y^+_{i}} \cKd{01}_{\lc\Gamma^{\le i}})\times_{I^2} \set{(t_1,t_0) | t_1 \geq t_0},
	\end{equation}
	where $t_i$ denotes the parameter $t$ on each factor. By adding these correction terms, we will define a homotopy between $\scF^0$ and $\scF^1$.
	
	To this end, define for a tuple $\textbf{i} = (d \geq i_{r} > \dots > i_1 \geq 1)$ the global Kuranishi chart 
	\begin{equation*}
		\cKd{01}_{\textbf{i}, \lc\Gamma} \;\coloneqq\; \cKd{01}_{\lc\Gamma^{> i_r}} \times_{Y^+_{i_{r}}} \dots \times_{Y^+_{i_{1}}} \cKd{01}_{\lc\Gamma^{\le i_1}}\,\times_{I^{r+1}}\set{t \in I^{r+1}: t_r \geq \dots \geq t_0}.
	\end{equation*}
	If $\textbf{i}$ is the empty sequence, we simply have $\cKd{01}_{\lc\Gamma}$.
	Given a composable sequence $\lc\sigma$ of analytic simplices and a sequence $\lc\Gamma$ of graphs lying over it, equip $\cK^{01}_{\lc \Gamma,\lc\sigma}$ with the orientation 
	\[
	\orl(\cK^{01}_{\lc \Gamma,\lc\sigma}) \cong \orl(Y_d^+)\orl(\sigma_d)\orl(s_{d-1})\dots\orl(s_1)\orl(\sigma_1)\orl(t).
	\]
	and 
	$$\cK^{01}_{\textbf{i},\lc\Gamma,\lc\sigma} :=	\cKd{01}_{\textbf{i}, \lc\Gamma}\times_{\Mbar_{\lc{\ov{\Gamma}}}} \prod_{j =1}^{d}\Delta^{\norm{\sigma_j}}$$\noindent
	with the fibre product orientation.
	In other words, its orientation line is given by
	\begin{equation}
		\orl(\cKd{01}_{\textbf{i}, \lc\Gamma}) \;\cong\; \orl(y^d)\orl(\sigma_d)\orl(s_{d-1})\dots \orl(\sigma_{i_{r}+1})\orl(t_r)\orl(\sigma_{i_{1}+1})\dots\orl(t_1)\orl(\sigma_{i_1}) \dots \orl(\sigma_1)\orl(t_0),
	\end{equation}
	where we use the notation $\orl(y_d) = \orl(X^{y^d_c}\times L^{y^d_o})$ and we write $\q^{01}_{\textbf{i},\lc \Gamma,\lc\sigma}$ for operation associated to this correspondence.
	
	\begin{lemma}\label{lem:homotopy-dmft}
		The maps 
		$T^d \,\cl\, \bigotimes\limits_{j= 1}^d\cdmc_n(y^j,y^{j-1})\otimes \scF^0(y^0)\;\to\; \scF^1(y^d)$
		defined by
		\begin{equation}
			T^d_{\lc\sigma}(\alpha) \;=\;  \sum_{\lc \Gamma, \textbf{i}} c_{\lc \Gamma}\,\fq^{01}_{\textbf{i},\lc\Gamma,\lc\sigma}(\alpha)\,Q^{\beta(\lc \Gamma)}
		\end{equation}
		for $d\ge 1$ and with $T^0 = 0$ define a homotopy of $A_\infty$ functors from $\scF^0$ to $\scF^1$.
	\end{lemma}
	
	\begin{proof}
		By Stokes' Theorem,
		\begin{equation}\label{eq:stokes-homotopy} 
			(-1)^{|\lc \sigma|'}d_{y^d}\fq^{01}_{\textbf{i}, \lc\Gamma,\lc\sigma}(\alpha)\; -\; \fq^{01}_{\textbf{i}, \lc\Gamma,\lc\sigma}(d_{y^0}\alpha) \;+\; (-1)^{1+n|y^d_o| + |\lc \sigma|'}\,\fq_{\del\cK^{01}_{\textbf{i}, \lc\Gamma,\lc\sigma}}(\alpha) \;=\; 0,
		\end{equation}
		since $\vdim(\cK^d{01}_{\textbf{i}, \lc\Gamma,\lc\sigma}) \equiv n|y^d_o| + |\lc \sigma|' \mod{2}$. Summing over all sequences $\lc\Gamma$ and all $\textbf{i}$ the last term on the left side yields the following operations.
		\begin{enumerate}[label=\Roman*),leftmargin=20pt]
			\item Arguing exactly as in~\eqref{boundary-simplex-sym}, we obtain the term $(-1)^{|\lc\sigma^{\le j}|'+1}T_{(\sigma_d,\dots,\del_{oc}\sigma_j,\dots,\sigma_1)}(\alpha)$ from the boundary of the simplices $\Delta^{\norm{\sigma_j}}$ and the boundary strata coming from disc bubbling, which are discussed in Lemma~\ref{lem:chain-map-stable}. 
			\item By the discussion before Equation~\eqref{eq:sign-composition-chains}, the vertical boundary of $\cK^{01}_{\lc\Gamma,\lc\sigma}$ (over $I^{d-1}$) yields the term $(-1)^{|\lc\sigma^{\le j}|'+1}T_{\concat_j\lc\sigma}(\alpha)$, noting that $\vdim(\cK^{01}_{\textbf{i}, \lc\Gamma,\lc\sigma})  = \vdim(\cK_{\lc\Gamma,\lc\sigma}) +1$.
			\item Consider now $\cK^{01}_{\textbf{i}, \lc\Gamma,\lc\sigma}|_{\{s_j = 1\}}$ for $j \notin \{i_1,\dots,i_r\}$. Choose $p \ge 0$ so that $i_p < j< i_{p+1}$ and define $\textbf{i}' := (i_r > \dots > i_{p+1} > j > i_p >\dots i_1)$. By the equivalent of Lemma~\ref{lem:gkc-for-higher-homotopies}, there exists an equivalence
			\begin{equation*}\label{eq:boundary-parameter-one}
				\cK^{01}_{\textbf{i}^{> p},\lc\Gamma^{> j},\lc\sigma^{> j}}\times_{I\times Y_j}\cK^{01}_{\textbf{i}^{\le p}, \lc\Gamma^{\le p},\lc\sigma^{\le j}}\;\simeq\; \cK^{01}_{\textbf{i}, \lc\Gamma,\lc\sigma}|_{\{s_j = 1\}}. 
			\end{equation*}
			The orientation sign comparing the fibre product orientation with the boundary orientation is $(-1)^{n|y^d_o|+|\lc\sigma^{> j}|'+1}$. On the other hand, we have by definition that
			\begin{equation*}\label{eq:boundary-parameter-diagonal}
				\cK^{01}_{\textbf{i}^{> p},\lc\Gamma^{> j},\lc\sigma^{> j}}\times_{I\times Y_j}\cK^{01}_{\textbf{i}^{\le p}, \lc\Gamma^{\le p},\lc\sigma^{\le j}}\;=\; \cK^{01}_{\textbf{i}', \lc\Gamma,\lc\sigma}|_{\{t_{p-1} = t_p\}}. 
			\end{equation*}
			In this case the orientations differ by $(-1)^{n|y^d_o|+|\lc\sigma^{> j}|'}$ due to the fact that the normal vector of $\{(s,t)\in I^2\mid s \le t\}$ is given by $s-t$. Thus, the operations associated to these correspondences cancel.
			\item If $t_j =1$ for some $j \in \{1,\dots,r\}$, then $t_{j'} = 1$ for any $j'\ge j$, whence
			\[\cK^{01}_{\textbf{i}, \lc\Gamma,\lc\sigma}|_{\{t_j = 1\}}\;\simeq\; \cK^{1}_{\lc\Gamma^{>i_j},\lc\sigma^{>i_j}}\times_{Y_j}\cK^{01}_{\textbf{i}^{\le j}, \lc\Gamma^{\le i_j},\lc\sigma^{\le i_j}}
			\]
			where $\textbf{i}^{\le j} = (i_{j-1} >\dots > i_1)$. Thus, the associated operation is $(-1)^{|\lc \sigma^{\le i_j}|'}\scF^1_{\lc\sigma^{> i_j}}\g T_{\lc\sigma^{\le i_j}}(\alpha)$ if $j > 0$, respectively $\scF^1_{\lc\sigma}(\alpha)$ if $j = 0$.
			\item  Similarly, 
			\[\cK^{01}_{\textbf{i}, \lc\Gamma,\lc\sigma}|_{\{t_j = 0\}}\;\simeq\; \cK^{01}_{\textbf{i}^{> j},\lc\Gamma^{>i_j},\lc\sigma^{>i_j}}\times_{Y_j}\cK^{0}_{ \lc\Gamma^{\le i_j},\lc\sigma^{\le i_j}}
			\]
			where $\textbf{i}^{\le j} = (i_r-i_j >\dots > i_{j+1}-i_j)$. Therefore, the associated operation appearing in Equation~\eqref{eq:stokes-homotopy} is $(-1)^{|\lc \sigma^{\le i_j}|'+1}T_{\lc\sigma^{> i_j}}\g \scF^0_{\lc\sigma^{\le i_j}}(\alpha)$ if $j < r$, respectively $-\scF^0_{\lc\sigma}(\alpha)$ if $j = r$.
		\end{enumerate}
		Adding these terms one exactly obtains the relation a homotopy is required to satisfy.
	\end{proof}
	
	\begin{lemma}\label{lem:compatible-symmetric-monoidal}
		$T$ is a homotopy of symmetric monoidal $A_\infty$ functors in the sense of Definition~\ref{de:homotopy-dmfts}.
	\end{lemma}
	
	\begin{proof}
		The natural transformation $P$ of degree $-1$ with 
		\[\mu^1(P)\; =\; (T\g\otimes_{\dmc_n})\g R^0 \,-\, R^1\g (\otimes_{\text{Ch}_\Lambda}\g(T\otimes T))\]
		is defined by $P^d = 0$ for $d \le 1$. 
		The maps $P^d$ for $d \ge 2$ are constructed by combining the correspondences above with the correspondences in Definition~\ref{de:symmetric-monoidal-transformation}. To this end, define for $\textbf{i} = (d> i_r \dots > i_r\ge 1)$ the subspace $I(\textbf{i}) \coloneqq I^{d-1-r}$. Given composable sequences $\lc\sigma$ and $\lc\tau$ of length $d$ in $\dmc_n$ and sequences of graphs $\lc\Gamma$ and $\lc\Gamma'$ lying over them, we define for $1 \le \kappa \le d-1-r$ the correspondence
		\begin{equation*}\label{} \cP_{\textbf{i},\lc\Gamma,\lc\sigma,\lc\Gamma',\lc\tau}(\kappa) := (\cKd^{01}_{\textbf{i},\lc\Gamma,\lc\sigma}\times\cK^{01}_{\textbf{i},\lc\Gamma',\lc\tau})\times_{I(\textbf{i})^2}(D_{d-1-r}\times\dots\times D_{\kappa} \times \Delta_\kappa\times C_{\kappa-1}\times \dots\times C_1).
		\end{equation*}
		The same computations as in the proof of Lemma~\ref{lem:homotopy-dmft} now show the claim.
	\end{proof}
	
	\noindent
	Suppose now we have made two different choices $J_0$ and $J_1$ of almost complex structures and two different choices of auxiliary data that define curved open-closed DMFTs $\scF^0$ and $\scF^1$ respectively.
	Choose a path $\{J_\kappa\}_\kappa$ of almost complex structures starting at $J_0$ and ending at $J_1$. By the same argument as in \cite{HS22}, we can find an auxiliary datum $\wt\alpha_{\mathsf{a}}$ for each $\mathsf{a}\in \cA$ that yields auxiliary data for $\Mbar_{\mathsf{a}}^{J_0}(X,L)$ and $\Mbar_{\mathsf{a}}^{J_1}(X,L)$ simultaneously and which allows one to construct a global Kuranishi chart $\wt\cK_{\mathsf{a}}$ for the `cobordism' $\Mbar_{\mathsf{a}}^{\{J_\kappa\}}(X,L)$, admitting a submersion to $[0,1]$. Applying the constructions of this proof to the charts $\wt\cK_{\mathsf{a}}$ we obtain a homotopy from $\wt\scF^0$ to $\wt\scF^1$, the curved open-closed DMFTs obtained from the restrictions of the global Kuranishi charts $\wt{\cK_{\mathsf{a}}}$ (and the Kuranishi charts built from them) to $0$ and $1$ respectively. Finally, we can use the theorem for the different choices of auxiliary data for $\Mbar^{J_0}_{\mathsf{a}}(X,L)$ and $\Mbar^{J_1}(\mathsf{a})(X,L)$ to obtain homotopies from $\scF^0$ to $\wt\scF^0$ and $\wt\scF^1$ to $\scF^1$ respectively. This completes the proof of \eqref{independence-curved}.\\

\noindent
We will now incorporate bounding cochains into the homotopies. To this end, let $\scF^0$, $\scF^1$ be two curved open-closed DMFTs obtained from two different choices of auxiliary data and $\omega$-tame almost complex structures. Let $\scA_i$ be the curved $A_\infty$ algebra associated to $\scF^i$ by Lemma~\ref{lem:algebra-from-dmft}. 
The last paragraph of the previous proof yields global Kuranishi charts $\cKd{01}_\Gamma$ interpolating between the global Kuranishi charts used to define $\scF^0$ and $\scF^1$. In particular, taking $\Gamma$ to be a stable map graph lying over $\sigma = D_k$ we obtain correspondences 
\[
L^k \times [0,1] \leftarrow \cKd{01}_{\Gamma, D_k} \rightarrow L \times [0,1].
\]
Composing this with the diagonal map $L^k \times [0,1] \rightarrow L^k \times [0,1]^k$ yields the correspondence
\[
(L \times [0,1])^k \,\xleftarrow{\wt\eva^-}\, \cKd{01}_{\Gamma, D_k} \,\xra{\wt\eva^+}\,L \times [0,1].
\]
Similar to the proof of Lemma~\ref{lem:algebra-from-dmft}, these correspondences equip $\wt \scA \coloneqq \Omega^*(L \times [0,1]; \Lambda)$ with the structure of a curved $A_\infty$ algebra. Given a gauge equivalence $\wt \bc$ between $\bc^0$ and $\bc^1$, we deform the homotopy $T$ constructed in Lemma \ref{lem:homotopy-dmft} as follows. Each $\cK^{01}_{\textbf{i},\pi^{[k]}\lc\Gamma,\pi^{[k]}\lc\sigma}$ has $k$ additional boundary marked points, with evaluation maps to $L \times [0,1]$. Therefore, taking $(T^{\wt b})^0 =0$ and
\[
T^{\wt \bc}_{\lc\sigma}(\alpha) \;=\;  \sum_{k\ge 0} \frac{1}{k!}T_{\pi^{[k]}\lc\sigma}(\alpha \times \wt \bc^{\times k}).
\]
defines a well-defined pre-natural transformation and similar for the deformation of $P$ of Lemma~\ref{lem:compatible-symmetric-monoidal}. The same argument as in the proof of Lemma \ref{lem:homotopy-dmft} show that $T^{\wt \bc}$ is a homotopy of $A_\infty$ functors. Similarly, deforming the natural transformation $P$ of Lemma \ref{lem:compatible-symmetric-monoidal} shows this is a homotopy of symmetric monoidal $A_\infty$ functors. Together with Lemma \ref{lem:homotopy-to-uncurved}, this completes the proof.
\end{proof}

\begin{remark}
	We can use the correspondences $\cKd{01}_{\Gamma}$ to define a curved open-closed DMFT $\wt \scF$ with $\wt \scF(y) = \Omega^*(X^{y_c}\times L^{y_o} \times [0,1];\Lambda)$. Analogously, we call this a \emph{pseudo-isotopy} of curved open-closed DMFTs. The curved open-closed DMFT of Theorem~\ref{thm:dmft-exists} is thus also independent of choices up to pseudo-isotopy.
\end{remark}

\appendix
\section{Integration along singular submanifolds}\label{subsec:singular-integration} 
\noindent In this section we prove a major technical ingredient for the definition of the open-closed DMFT associated to a Lagrangian submanifold. The main result is Theorem \ref{thm:pushforward along chains}, from which we obtain that our generalised $\fq$-operations are well-defined.
\addtocontents{toc}{\protect\setcounter{tocdepth}{1}}

\subsection{Preliminaries on analytic spaces} A key ingredient of the proof is the resolution of singularities of an analytic space as defined/constructed in \cite{BM97}. We recall here the relevant definitions and extend them to the setting of our main application in \textsection\ref{subsec:lagrangian-dmft}. More details can be found in \cite{BM88}.

\begin{definition}\label{} A \emph{real analytic manifold with corners} $V$ is a Hausdorff space with an atlas $\{(\phi_i,U_i)\}_{i\in I}$ of homeomorphisms $\phi_i \cl U_i \to V_i \sub (\bR_{\geq 0})^n$, where $V_i$ is an open subspace of $ (\bR_{\geq 0})^n$ so that the transition functions extend to real analytic functions in a neighbourhood of $\phi_i(U_i\cap U_j)$ for any $i,j\in I$. 
\end{definition}

\begin{definition}\label{} A \emph{analytic map} between real analytic varieties with corners is a continuous map that in local coordinates extends to a real analytic map between affine real analytic varieties.
\end{definition}

\noindent
This is a adaption of the notion of a smooth map between manifolds with corners in \cite{Joy}. 

\begin{example} The standard simplex $\Delta^n$ is given by the intersection of $(\bR_{\geq 0})^{n+1}$ with the zero locus of $f(x) = 1-\sum_{i =0}^n x_i$ and thus admits the structure of a real analytic manifold with corners.
\end{example}

\begin{definition}\label{de:analytic} Let $V$ be a real analytic manifold with corners. A subspace $Z\sub V$ is \emph{analytic} if for each $z \in Z$ there exists an open neighbourhood $U$ of $z$ in $V$ and an analytic function $f \cl U \to \bR^k$ so that $Z\cap U  = f\inv(\{0\})$.
\end{definition}

\noindent
Let $Z\sub V$ be an analytic subspace of a real analytic manifold $V$ with corners. We call $x \in Z$ \emph{smooth (of dimension $k$)} if there exists a neighbourhood $U \sub Z$ of $x$ so that $Z \cap U$ is a real analytic submanifold (of dimension $k$) of $V$. We call $x$ \emph{singular} otherwise. We write $Z^{sm}$ and $Z^{sing}$ for the set of smooth, respectively singular points of $Z$. The \emph{dimension} of an analytic subspace is the supremum over the dimensions of its smooth points. As $Z^{sm}$ is analytic, so is $Z^{sing}$.

\begin{remark}\label{rem:analytic-fibre-product} If $\varphi \cl X\to Y$ is an analytic map between real analytic manifolds with corners and $Z\sub Y$ is a closed analytic subspace, then the preimage $\varphi\inv(Z)$ is an analytic subspace of $X$. In particular, the fibre product $X\times_Z Y$ of real analytic maps is an analytic subspace. \end{remark}

\noindent
However, the fact that we allow corners means that fibre product can be ill-behaved as the following example shows.

\begin{example}\label{ex:non-regular-fibre-product} Consider the analytic function $f \cl [0,1]^2\to [0,1]$ given by $f(t) = t_1t_2-1$. Then the fibre product $\{1\}\times_{[0,1]} [0,1]^2$ is given by $\{(1,1)\}$. .
\end{example}

\noindent
This is due to the fact that a map between a manifold with corners need not be a submersion at a boundary point $x$ even if its differential at $x$ is surjective (between the tangent spaces of the respective strata). The issue is that the maps we (have to) consider are only \emph{weakly smooth} in the language of \cite{Joy}, cf. Definition~3.2 and Example~3.5(f) op. cit..

\begin{definition}\label{de:corner-singularity}
	We say that a weakly smooth map $f\cl M\to N$ between manifolds with corners has a \emph{corner singularity} over a submanifold $Q\sub N$ if $df(x)$ is surjective for $x \in f\inv(Q)$ but $\dim(f\inv(Q)) < \dim(M)-\codim(Q)$, where $\dim$ on the left hand side denotes the Lebesgue dimension.
\end{definition}

\noindent
 The blow-down construction below will not identify corner singularities as the fibre product locally extends to a transverse fibre product of analytic manifolds without corners. By restricting the class of maps we consider, we ensure that corner singularities arise in codimension $\ge 2$.

\begin{definition}\label{de:admissible-map} An analytic map $\varphi\cl V\to W$ between analytic manifolds with corners is \emph{admissible} if near each $x\in V$ there exists a neighbourhood $U$ so that $\varphi|_U$ is the composition of a smooth map in the sense of \cite[Definition~3.2]{Joy} with a map of the form $f_k\times \ide$, where $f_k\cl [0,\infty)^k\to [0,\infty]$ is the multiplication $f_k(x) = x_1\cdots x_k$. We call $\varphi$ \emph{strongly smooth} if it is smooth in the sense of \cite[Definition~3.2]{Joy}.
\end{definition}

\begin{lemma}\label{cor:fibre-product-with-action} Suppose a real analytic Lie group $G$ acts analytically on a real analytic manifold $Y$ with corners and $\varphi \cl X\to Y$ is an analytic map. Then 
	$$V := \{(x,g,y)\in X\times G\times Y\mid g\cdot \varphi(x) = y\}$$
	is an analytic subspace of $X\times G\times Y$. The map 
	$$\sigma_V\cl G\times V\to V : (h,(x,g,y))\mapsto (x,hg,h\cdot y)$$ is analytic.
\end{lemma}

\begin{proof} By assumption, the map $\sigma\cl G\times Y\to Y : (g,y)\mapsto g\inv\cdot y$ is real analytic, so the first claim follows from Remark \ref{rem:analytic-fibre-product}. The second claim follows from the fact that $\sigma_V$ is the restriction of the analytic action of $G$ on $X\times G\times Y$.
\end{proof}

\begin{definition}\label{} A \emph{real analytic Lie groupoid with corners} $X = [X_1 \rightrightarrows X_0]$ is a Lie groupoid so that all $X_0$ and $X_1$ are real analytic varieties with corners and all structure maps are analytic. A morphism between two such Lie groupoids $X\to Y$ is \emph{analytic} if the maps $X_0 \to Y_0$ and $X_1\to Y_1$ are analytic. 
\end{definition}

\noindent
We define the category of \emph{real analytic orbifolds with corners} to have objects given by real analytic \'etale proper Lie groupoids with corners and morphisms $f = (R,\wt f)$ given by zigzags
\begin{center}\begin{tikzcd}
		&\wt X\arrow[dr,"\wt f"] \arrow[dl,"R"]\\  X&& Y\end{tikzcd} \end{center}
where $R$ is a refinement, and both $R$ and $f$ are analytic.

\begin{example} Suppose $G$ is a real analytic group acting analytically and almost freely on a real analytic manifold $V$ with corners. Then $[V/G]$ is a real analytic orbifold with corners.
\end{example}

\noindent
The definition of an analytic subgroupoid of a real analytic orbifold is analogous to the one of an analytic subspace.

\begin{cor}\label{cor:analytic-orbifold-fibre-product} Suppose $X = [V/G]$, $Y = [W/H]$ and $Z = [U/K]$ are analytic orbifolds with corners and $X\to Z$ and $Y\to Z$ are admissible analytic maps. Then the orbifold fibre product $X \times_Z Y$ is an analytic subgroupoid of $X\times K\times Y$.
\end{cor}

\begin{lemma}\label{} 
	The moduli space $\Mbar_{g,h;k,\ell}^{\,*,\,m}(\bC P^N,\bR P^N)$ of \cite[\S2.1]{HH25} is a real analytic orbifold with corners on which $\PGL_\bR(N+1)$ acts analytically. 
\end{lemma}

\begin{proof} In fact, it is real algebraic but we will only show real analyticity in order to save on definitions. Let $\fc$ be the anti-complex linear involution on $\Mbar_{2g+h-1,2k +|\ell|}(\bC P^N,2m)$ induced by complex conjugation on $\bC P^N$. Then $\Mbar_{g,h;k,\ell}^{\,*,\,m}(\bC P^N,\bR P^N)$ is a codimension $0$ submanifold of the fixed point locus $\Mbar_{2g+h-1,2k +|\ell|}(\bC P^N,2m)^\fc$. The first claim now follows from the discussion of the deformation theory of curves with boundary in \cite[\textsection3.3]{Liu20}. The second assertion is immediate.
\end{proof}

\begin{cor}\label{cor:moduli-analytic-orbifold} The moduli space $\Mbar_{g,h;k,\ell}$ is a real analytic orbifold with corners.
\end{cor}

\begin{definition}\label{} A \emph{relative real analytic manifold with corners} $X/B$ consists of a smooth manifold $X$ and a submersion $q\cl X\to B$ to a real analytic manifold with corners. A \emph{relative analytic subspace} of $X/B$ is the preimage $q\inv(Z)$ of an analytic subspace $Z \sub B$. 
\end{definition}

\begin{definition}\label{} A map $\varphi \cl X/B\to X'/B'$ between real analytic varieties with corners is \emph{analytic} if $\varphi$ is smooth and the induced map $B\to B'$ is analytic.
\end{definition}

\noindent
The following definition is not the most general but sufficient for our purposes.

\begin{definition}\label{de:rel-analytic-orbifold} A \emph{relative real analytic orbifold with corners} consists of a relative real analytic manifold with corners $X/B$ equipped with a smooth proper group action by a real analytic group $G$ so that the group action on $X$ is almost free and the action on $B$ is analytic. We write $\cX$ for the associated orbifold.
\end{definition}

\noindent
This definition is inspired by the theory of rel--$C^\infty$ manifolds although it does not require it. The point is that for our main use of resolution of singularities, the singular manifold we study is a relatively smooth manifold over a singular fibre product of analytic maps. We thus only have to apply the resolution of singularities to said fibre product and not to the actual space. This will become clear in Theorem~\ref{thm:pushforward along chains}.

\subsection{Integration over relative real analytic orbifolds} We will now prove the main theorem of this section, Theorem \ref{thm:pushforward along chains}, which is essential for the construction in \textsection\ref{subsec:lagrangian-dmft} of the DMFT associated to a Lagrangian submanifold.

\begin{definition}\label{} Let $V$ be a real analytic manifold with corners and $Z \sub V$ be a closed analytic subspace. An \emph{embedded resolution of singularities} of $Z$ is a proper analytic map $\sigma_Z \cl \wt V\to V$ so that 
	\begin{itemize}[leftmargin=20pt]
		\item $\sigma_Z$ is given by a locally finite sequence of blow-ups of $V$ with final exceptional hypersurface\footnote{As we are in the real analytic setting, we prefer not to use the word divisor, reserving it for submanifolds of real codimension $2$.} $E\sub \wt V$,
		\item the strict transform $\wt Z := \cc{\sigma_Z\inv(Z\sm Z^{sing})}$ is an analytic submanifold of $V$,
		\item the restriction $\sigma_Z|_{\wt Z\sm E}\cl  \wt Z\sm E\to Z\sm Z^{sing}$ is an isomorphism,
		\item $\wt Z$ and $E$ have normal crossing intersections.
	\end{itemize}
\end{definition}

\noindent
This is usually only defined in the case where $\del V = \emst$. However, the local construction of the blow-up in the real analytic category generalises to the setting of analytic manifolds with corners. In particular, we may arrange for the blow-down map to preserve corner strata.

By \cite[Theorem~1.1]{BM08}, there exists a resolution of singularities for any compact\footnote{The same is true for any open subspace of $Z$ with compact closure if $Z$ itself is not compact.} analytic subspace $Z$ of a real analytic manifold $V$. The resolution $\sigma_Z$ is \emph{universal} in the sense that any isomorphism $U \to U'$ between open subspaces of analytic sets $Z$ and $Z'$ lifts to an isomorphism $\sigma_Z\inv(U)\to \sigma_{Z'}\inv(U')$.

\begin{proposition}\label{prop:resolution-of-singularities-with-corners} Let $V$ be a real analytic manifold with corners and $Z \sub V$ be an analytic subspace. If $Z$ is a relatively compact open subset of an analytic space $Z'\sub V$, it admits a universal (embedded) resolution of singularities.
\end{proposition}

\begin{proof} The proof of \cite[Theorem~1.1]{BM08} relies on the resolution of singularities of a marked ideal, \cite[Theorem~1.3]{BM08}. A \emph{marked ideal} is a tuple $\lc{\cI} = (V,N,E_0,\cI,d)$, where $N\sub V$ is an analytic submanifold, $E_0$ is a hypersurface in $M$ with normal crossing singularities, $\cI$ is a coherent sheaf of ideals on $N$ and $d$ is a non-negative integer. This definition can be made verbatim in the setting with corners. Given $Z\sub V$ we let $\cI_Z\sub \cO_V$ be its ideal sheaf in $V$ and set $\lc{\cI}_Z := (V,V,\emst,\cI_Z,1)$. Since the proof of \cite[Theorem~1.3]{BM08} is phrased purely in terms of a marked ideal with the only geometric inputs, such as \cite[Lemma~3.1]{BM08} being local and thus generalising directly, we obtain the claim. 
\end{proof}

\begin{cor}\label{cor:equivariant-resolution} Suppose $Z \sub V$ satisfies the condition of Proposition \ref{prop:resolution-of-singularities-with-corners} and a compact real analytic Lie group $G$ acts almost freely and analytically on $V$, preserving $Z$ and $Z'$. Then we can choose the resolution of singularities to be $G$-equivariant, i.e., $\wt V$ admits an analytic $G$-action preserving $\wt Z$ and $E$ and $\sigma_Z$ is $G$-equivariant.
\end{cor}

\begin{proof} We will use the universality of the resolution in Proposition \ref{prop:resolution-of-singularities-with-corners}. If $Z$ is $G$-invariant, then the isomorphism $\phi_g \cl V\to V$ associated to $g \in G$ induces an isomorphism $Z\to Z$ and thus lifts uniquely to an isomorphism $\wt\phi_g\cl \wt Z \to \wt Z$. In particular, this yields an analytic $G$-action on $\wt Z$. To see that this action extends to an analytic action on $\wt V$, one can use the functoriality statement of \cite[Theorem~1.3]{BM08} on the action map $G\times V\to V$.\end{proof}

\begin{proposition}\label{prop:resolution-and-boundary} Suppose $\sigma_Z \cl \wt V\to V$ is an embedded resolution of singularities of an analytic subspace $Z$. 
	\begin{enumerate}[\normalfont 1),leftmargin=20pt,ref=\arabic*]
		\item\label{preimage-stratum}  The preimage $\sigma_Z\inv(S)$ of a corner stratum $S$ of $V$ is a corner stratum of $\wt V$ of the same dimension. In particular, 
		\begin{equation}\label{eq:boundary-of-resolution}\del\wt Z = \sigma_Z\inv(Z\cap \del V).
		\end{equation}
		\item\label{resolution-stratum} For any corner stratum $S\sub \del V$, the restriction 
		\begin{equation}\label{} \sigma_Z|_{\sigma_Z\inv(S)}\cl \sigma_Z\inv(S)\to  S\end{equation}
		is an embedded resolution of singularities that is isomorphic to the universal resolution of $Z\cap S$ in $S$.
	\end{enumerate}
 The corresponding statements hold in the equivariant setting.
\end{proposition}

\begin{proof} Since $\wt Z$ is an analytic submanifold of $\wt V$, we have that $\del \wt Z = \wt Z \cap \del V$. The equality in \eqref{eq:boundary-of-resolution} then follows from \cite[Proposition~3.2]{AK10} since the centers of the blow-ups used to define $\wt V$ are analytic submanifolds and thus are neatly embedded\footnote{That is, $C\cap \del V_j = \del C$ and $C\pf \del V_j$.} in the respective intermediate blow-up.\par The second assertion follows from the functoriality of the resolution of singularities in \cite{BM08} with respect to closed embeddings $V'\hkra V$. To be precise, \cite{BM08} does not prove this but refers to \cite[Theorem~1.0.1]{Wlo05}, see also \textsection 5.2 op. cit.. As the argument in \cite{Wlo05} only uses the notion of a marked ideal, it carries over to our setting. Note that, as discussed in \cite[\textsection 71]{Kol07}, this relies on us using very simple marked ideals.
\end{proof}

\begin{theorem}\label{thm:pushforward along chains}
	Suppose $(X'/B,G)$ is an oriented real analytic orbifold and $X\sub X'$ a $G$-invariant relatively compact open subspace. Let $f \cl X'\to Y$ be a $G$-invariant submersion to a smooth oriented manifold $Y$. Given analytic maps $\varphi \cl [B/G]\to \cV$ and $\psi \cl \cW\to \cV$, where $\cW$ and $\cV$ are compact oriented real analytic orbifolds with corners, we can define a pushforward 
	\begin{equation}\label{eq:modified-pushforward} (f_\psi)_*\cl \Omega_c^*(\cX) \to \Omega^{*-\normalfont\text{rdim}}(Y)\end{equation}
	which agrees with the composition
	\begin{equation}\label{eq:if-transverse} \Omega_c^*(\cX) \xra{p^*} \Omega_c^*(\cX\times_\cV\cW)\xra{(fp)_*} \Omega^*(Y)\end{equation}
	if $\varphi\pf \psi$, where $\cX := [X/G]$ and $p\cl \cX\times_\cV\cW\to \cX$ is the canonical projection. 
\end{theorem}

\begin{proof} Write $\cW = [W/\Gamma]$ for some analytic compact Lie group $\Gamma$ and let $\cB$ be the global quotient associated to the $(G\times\Gamma\times H)$-action on $X\times\Gamma\times W$.  By Corollary ~\ref{cor:analytic-orbifold-fibre-product} and Corollary~\ref{cor:equivariant-resolution}, we can find a morphism $\rho\cl\wt\cB\to \cB$ of relative analytic orbifolds that defines a universal resolution of singularities for $\cX\times_\cV\cW$. Set $\cZ:= \rho\inv(\cX\times_\cV\cW)$.  The map $\wt f:= f\g \pr_{\cX}\g \rho|_{\cZ}$ is still a co-oriented vertical submersion and we can thus define $(f_\psi)_*$ by the composition
	\begin{equation}\label{eq:def-modified-pf}  \Omega_c^*(\cX)\xra{\pr_{\cX}^*} \Omega_c^*(\cB) \xra{\rho^*} \Omega_c^*(\wt\cB)\to \Omega_c^*(\cZ) \xra{\wt f_*} \Omega^{*-k}(Y)
	\end{equation}
	where $k = \normalfont\text{rdim}(\wt f) = \dim(Y)-\dim(\wt Z) = \dim(Y)-\dim(Z)$. If the fibre product is transverse, we can take $\rho$ to be the identity, whence \eqref{eq:def-modified-pf} agrees with \eqref{eq:if-transverse}. 
\end{proof}

\begin{lemma}\label{lem:modified-stokes} Let $f'$ be the restriction of $f$ to $\del\cX$ and $\psi'$ be the restriction of $\psi$ to $\del \cW$. Then the pushforward $(f_\psi)_*$ satisfies 
	\begin{equation}\label{eq:modified-stokes} d(f_\psi)_*(\alpha) = (f_\psi)_*(d\alpha) + (-1)^{|\alpha|+\dim(\cZ)}\lbr{(f'_\psi)_*(\alpha|_{\del\cX}) + (-1)^{\dim(\cX)+\dim(\cV)}(f_{\psi'})_*(\alpha)} 
	\end{equation}
\end{lemma}

\begin{proof} Let $\alpha\in \Omega_c^*(\cX)$ be arbitrary and set $\wt\alpha := \rho^*\pr_\cX^*\alpha$. By Stokes' Formula, \cite[Proposition~2.2]{ST16}, we have 
	\begin{equation*}\label{}d(f_\psi)_*(\alpha) = (f_\psi)_*(d\alpha) + (-1)^{|\alpha|+\dim(\cZ)} (\wt f|_{\del\cZ})_*(\wt\alpha).\end{equation*}
	By abuse of notation we write $\cX\times\cW$ for the orbifold $\cB$ defined in the previous proof. By Proposition~\ref{prop:resolution-and-boundary}\eqref{preimage-stratum}, we have 
	$$\del\cZ = \cZ\cap \rho\inv((\del \cX)\times\cW) \sqcup (-1)^{\dim(\cX)}\cZ\cap\rho\inv(\cX\times \del\cW) =: \del_1 \cZ\sqcup \del_2\cZ.$$
	Let $f'$ be the pullback of $f$ to $\del X = X\times_B\del B$. This is still a submersion and by Proposition~\ref{prop:resolution-and-boundary}\eqref{resolution-stratum} and Lemma \ref{lem:multiplication-up-to-codim}, the composite
	\begin{equation}\label{eq:def-modified-pf-boundary-1}  \Omega_c^*(\cX)\xra{\pr_{\cX}^*} \Omega_c^*(\cX\times \del\cW) \xra{\rho^*} \Omega_c^*(\rho\inv(\cX\times \del\cW))\to \Omega_c^*(\del_2\cZ) \xra{\wt f_*} \Omega^{*-k}(Y)
	\end{equation}
	agrees with $(f'_\psi)_*$. Similarly, if $\psi'\cl \del \cW\to \cV$ is the restriction of $\psi$, then
	\begin{equation}\label{eq:def-modified-pf-boundary-2}  \Omega_c^*(\cX)\xra{\pr_{\cX}^*} \Omega_c^*(\cX\times \del\cW) \xra{\rho^*} \Omega_c^*(\rho\inv(\cX\times \del\cW))\to \Omega_c^*(\del_2\cZ) \xra{\wt f_*} \Omega^{*-k}(Y)
	\end{equation}
	agrees with $(f_{\psi'})_*$. Therefore,
	\begin{equation}\label{} (\wt f|_{\del\cZ})_*(\wt \alpha) = (f'_\psi)_*(\alpha|_{\del\cX}) +(-1)^{\dim(\cX)} (f_{\psi'})_*(\alpha).\end{equation}
	by \cite[Proposition~7.4(30)]{Joy12}.
\end{proof}

\begin{lemma}\label{lem:modified-push-pull} Suppose $f\cl \cX\to Y$ and $\varphi,\psi$ are as in Theorem \ref{thm:pushforward along chains} and
	\begin{center}\begin{tikzcd}
			\cX'\arrow[r,"f'"] \arrow[d,"p"]&Y' \arrow[d,"q"]\\ 
			\cX \arrow[r,"f"] & Y \end{tikzcd} \end{center}
	is a cartesian square of oriented with $q$ a smooth submersion between two smooth manifolds. Equipping $\cX'$ with the pullback orientation, we have
			\begin{equation}\label{eq:push-pull-modified}
				 (f'_{\psi})_*p^* = q^*(f_\psi)_*.
			\end{equation}
\end{lemma}

\begin{proof} The first assertion is immediate, while the second follows from the functoriality of the resolution of singularities with respect to smooth morphisms, \cite[Theorem~1.1(3)]{BM08} and \cite[Proposition~2.1(4)]{ST16}.
\end{proof}

\bibliographystyle{amsalpha}
\bibliography{bib}

\smallskip
\Addresses

\end{document}